\documentclass{amsart}%
\usepackage{amsfonts}
\usepackage{hyperref}
\usepackage{adjustbox}
\usepackage[table]{xcolor}
\usepackage{calrsfs}
\usepackage{tikz-cd}
\usepackage[all,cmtip]{xy}
\usepackage[margin=1.55in]{geometry}
\usepackage{amsmath}
\usepackage{amssymb}
\usepackage{graphicx}%
\newtheorem{theorem}{Theorem}
\numberwithin{theorem}{section}
\theoremstyle{plain}
\newtheorem*{acknowledgement}{Acknowledgement}

\newtheorem{theoremann}{Theorem}

\newtheorem{corollary}[theorem]{Corollary}
\newtheorem*{corollaryann}{Corollary}

\newtheorem{definition}[theorem]{Definition}

\newtheorem{lemma}[theorem]{Lemma}

\newtheorem{proposition}[theorem]{Proposition}
\theoremstyle{remark}
\newtheorem{remark}[theorem]{Remark}

\newtheorem{example}[theorem]{Example}

\numberwithin{equation}{section}
\newcommand*{\Scale}
[2][4]{\scalebox{#1}{$#2$}}
\definecolor{lg}{rgb}{0.8,0.8,0.8}
\begin{document}
\title[$K$-theoretic Poitou--Tate duality]{$K$-theoretic Poitou--Tate duality in\linebreak higher dimensions:\ proper case}
\author{Oliver Braunling}
\address{FH Dortmund, University of Applied Sciences and Arts, Dortmund,
Germany\linebreak\linebreak Instituto de Ciencias Matem\'{a}ticas (ICMAT),
Madrid, Spain}
\thanks{The author acknowledges support for this article as part of Grant
RYC2019-027693-I funded by MCIN/AEI/10.13039/501100011033.}

\begin{abstract}
We generalize Blumberg--Mandell's $K$-theoretic Poitou--Tate duality to
arithmetic schemes of arbitrary dimension, smooth and proper over
$S$-integers. As in our earlier papers on the subject, we discuss how to model
the compactly supported side via the $K$-theory of locally compact modules.

\end{abstract}
\maketitle

We generalize Blumberg--Mandell's $K$-theoretic Poitou--Tate duality
\cite{MR4121155} and the ideas of Clausen's Thesis \cite{clausenthesis} to
smooth $\mathcal{O}_{S}$-proper arithmetic schemes of higher dimension.

\textit{Dimension One:} We first motivate how this story began. Let $p^{r}$ be
a power of an odd prime, $r\geq1$, and $\mathcal{O}_{S}$ the localization of
the ring of integers of a number field $F$ at finitely many primes and such
that $\frac{1}{p}\in\mathcal{O}_{S}$. Let $\mathcal{F}$ be an \'{e}tale
$\mathbf{Z}/p^{r}$-module sheaf on $\operatorname*{Spec}\mathcal{O}_{S}$ and
$\mathcal{F}^{\ast}:=\operatorname*{RHom}(\mathcal{F},\mathbf{Z}/p^{r}(1))$
its dual. Then Poitou--Tate duality, or in this format better known as
Artin--Verdier duality, tells us that%
\begin{equation}
H^{s}(\mathcal{O}_{S},\mathcal{F})\otimes H_{c}^{3-s}(\mathcal{O}%
_{S},\mathcal{F}^{\ast})\longrightarrow H_{c}^{3}(\mathcal{O}_{S}%
,\mathbf{Z}/p^{r}(1))\overset{\operatorname*{tr}}{\longrightarrow}%
\mathbf{Z}/p^{r} \label{lint1}%
\end{equation}
is a perfect pairing of finite abelian groups. This pairing encompasses most
of classical class field theory, so it is of considerable interest. It might
also be seen as an arithmetic sibling of Poincar\'{e} duality and bewitchingly
lets $\operatorname*{Spec}\mathcal{O}_{F}$ appear \'{e}tale $3$-dimensional.
The original question was: Does this pairing have a $p$-complete $K(1)$-local
$K$-theory counterpart?

This is a plausible hope as Thomason's descent spectral sequence has the shape%
\[
E_{2}^{s,t}:=H^{s}\left(  \mathcal{O}_{S},\mathbf{Z}/p^{r}\left(  -\tfrac
{t}{2}\right)  \right)  \qquad\Longrightarrow\qquad\pi_{-s-t}L_{K(1)}%
K/p^{r}(\mathcal{O}_{S})\text{,}%
\]
so for $\mathcal{F}$ being Tate twists of $\mathbf{Z}/p^{r}$, the cohomology
groups on the left side of Poitou--Tate duality show up on the $E_{2}$-page
for $K(1)$-local $K$-theory.

Imagine there was a corresponding spectral sequence with compactly supported
\'{e}tale cohomology $H_{c}^{\bullet}$ instead, say converging to something we
may preliminarily denote by \textquotedblleft$L_{K(1)}K_{c}/p^{r}%
(\mathcal{O}_{S})$\textquotedblright\ and think of $K_{c}$ as compactly
supported (Nisnevich or Zariski) $K$-theory.\ Think of its $K(1)$-localization
as transforming it into compactly supported \'{e}tale $K$-theory. Then one
could hope, as an ad hoc idea, that the Poitou--Tate pairing of Eq.
\ref{lint1} lifts to a pairing of $K$-theory spectra%
\begin{equation}
L_{K(1)}K(\mathcal{O}_{S})\otimes L_{K(1)}K_{c}(\mathcal{O}_{S}%
)\longrightarrow L_{K(1)}K_{c}(\mathcal{O}_{S})\overset{\operatorname*{tr}%
}{\longrightarrow}L_{K(1)}\mathbb{S}/p^{r}\text{,} \label{lint2}%
\end{equation}
where \textquotedblleft$\otimes$\textquotedblright\ now means the $K(1)$-local
smash product and $\mathbb{S}$ the sphere spectrum. The vision then would be
that once we use the descent spectral sequence and its $K_{c}$-counterpart,
this spectral pairing of Eq. \ref{lint2} would induce a pairing of spectral
sequences%
\[
E_{\bullet}^{\bullet,\bullet}\otimes E_{c,\bullet}^{\bullet,\bullet
}\longrightarrow E_{c,\bullet}^{\bullet+\bullet,\bullet+\bullet}%
\]
which, on the $E_{2}$-page, induces the original Poitou--Tate pairing of Eq.
\ref{lint1}. And, even stronger, one could hope for the spectral pairing to be
a perfect pairing in the $K(1)$-local setting.\ Then this would be a
spectrum-lift of classical Poitou--Tate duality, with the original pairing
showing up on the $E_{2}$-page. An early version of this idea can be found in
Clausen's Thesis \cite{clausenthesis}, which was probably the original
starting point for this type of question.

Blumberg--Mandell \cite{MR4121155} developed a concrete implementation of
these ideas: They consider $\mathcal{O}_{S}=\mathcal{O}_{F}\left[  \tfrac
{1}{p}\right]  $, and define $L_{K(1)}K_{c}$ via an \'{e}tale hypersheaf
variant of $j_{!}$ for $j\colon\operatorname*{Spec}\mathcal{O}_{S}%
\hookrightarrow\operatorname*{Spec}\mathcal{O}_{F}$ and as a fiber to a
compactification. The spectral trace map cannot be defined before
$K(1)$-localization, but once we are $K(1)$-local, it takes the form%
\begin{equation}
L_{K(1)}K_{c}(\mathcal{O}_{S})\overset{\operatorname*{tr}}{\longrightarrow
}I_{\mathbf{Z}_{p}}\mathbb{S}_{\widehat{p}}\text{,} \label{lhuv1}%
\end{equation}
where $I_{\mathbf{Z}_{p}}$ is the $p$-adic Anderson dual of the $K(1)$-local
sphere. In Clausen's variant, he proceeds a bit differently at this point, but
the philosophy is the same.

Finally, Blumberg--Mandell prove all this by bootstrapping from the original
Poitou--Tate duality statement\footnote{a bit similar to Marc Levine's idea to
prove Thomason's results by first proving them for motivic cohomology and then
bootstrapping from the AH\ spectral sequence, as in \cite{MR1740880}.}. Their
approach does not give an independent proof of Poitou--Tate duality, but
constructs a spectral lift to $K(1)$-local $K$-theory. Cho \cite{cho2025} then
extended this picture also to $p=2$.

In our previous article on the subject \cite{klca1}, we have developed a
different variant of these ideas. Let $p$ be odd again. As $\mathcal{O}_{S}$
is a regular ring, one may define its $K$-theory as the $K$-theory of the
category of \textit{finitely generated} $\mathcal{O}_{S}$-modules
$\mathsf{Mod}_{\mathcal{O}_{S}}$. There is a corresponding category
${\mathsf{LCA}}_{\mathcal{O}_{S}}$ whose objects are \textit{locally compact}
$\mathcal{O}_{S}$-modules, i.e., each object is a possibly infinitely
generated topological $\mathcal{O}_{S}$-module such that $0$ admits a compact
neighbourhood. Let ${\mathsf{LC}}_{\mathcal{O}_{S}}$ be its quotient modulo
real vector space modules. In \cite{klca1} we show that one can take
$K_{c}:=\Sigma^{-1}K({\mathsf{LC}}_{\mathcal{O}_{S}})$, i.e., these locally
compact modules naturally model compactly supported $K$-theory.

There is a natural pairing%
\begin{equation}
\mathsf{Mod}_{\mathcal{O}_{S}}\times{\mathsf{LC}}_{\mathcal{O}_{S}%
}\longrightarrow{\mathsf{LC}}_{\mathcal{O}_{S}} \label{lint3}%
\end{equation}
by taking a tensor product of topological modules, and considering those in
$\mathsf{Mod}_{\mathcal{O}_{S}}$ as equipped with the discrete topology. This
induces a $K(\mathcal{O}_{S})$-module structure on $K({\mathsf{LC}%
}_{\mathcal{O}_{S}})$, and this pairing%
\[
K(\mathcal{O}_{S})\otimes K({\mathsf{LC}}_{\mathcal{O}_{S}})\longrightarrow
K({\mathsf{LC}}_{\mathcal{O}_{S}})
\]
turns out to be exactly what one needs to realize the pairing in Eq.
\ref{lint2}. Next, inspired by an idea that we have learned from Clausen,
there turns out to be a general $!$-pushforward, an exact functor%
\begin{equation}
j_{!}\colon{\mathsf{LC}}_{\mathcal{O}_{S}}\longrightarrow{\mathsf{LC}%
}_{\mathcal{O}_{F}}\text{,} \label{lhuv3}%
\end{equation}
which enters the construction of the evaluation map in Eq. \ref{lhuv1}. The
existence of such a pushforward is cheap once we go $K(1)$-local or \'{e}tale
hypersheafify and may use the \'{e}tale $6$-functor formalism and all its
adjunctions $-$ but here we have a lift of this map to the Zariski/Nisnevich
setting. This is surprising as \textit{no transfer} $K(\mathcal{O}%
_{S})\rightarrow K(\mathcal{O}_{F})$ \textit{exists} (the underlying map is
not proper!). This is part of our general conjectural belief that locally
compact modules should be able to model compactly supported $K$-theory in the
Nisnevich setting. However, this is not known and there are no such results in
the literature yet.

The reader might dislike this sudden introduction of topological modules into
the game, but we would not know how to write down such a $!$-pushforward when
defining the compactly supported side as in Blumberg--Mandell \cite{MR4121155}%
. Moreover, this angle portrays the ordinary and compactly supported side of
the original Poitou--Tate duality, Eq. \ref{lint1}, as an incarnation of a
dualism%
\[
\text{finitely generated }\mathcal{O}_{S}\text{-modules}\quad\leftrightarrow
\quad\text{locally compact }\mathcal{O}_{S}\text{-modules,}%
\]
and can be stated as%
\begin{equation}
I_{\mathbf{Z}_{p}}L_{K(1)}K(\mathcal{O}_{S})\overset{\sim}{\longrightarrow
}L_{K(1)}K({\mathsf{LC}}_{\mathcal{O}_{S}})\text{.} \label{llv3}%
\end{equation}
The \textit{finite generation} vs.\textit{ local compactness} dualism is also
seen to be true for finite extensions of $\mathbf{F}_{q}$ or $\mathbf{Q}_{p}$
in \cite{klca1}:%
\begin{equation}%
\begin{array}
[c]{rclccl}%
I_{\mathbf{Z}_{p}}L_{K(1)}K(L) & \overset{\sim}{\longrightarrow} &
L_{K(1)}K({\mathsf{LC}}_{L}) & \qquad & \qquad & \text{for }L/\mathbf{F}%
_{q}\text{ a finite extension,}\\
I_{\mathbf{Z}_{p}}L_{K(1)}K(L) & \overset{\sim}{\longrightarrow} &
L_{K(1)}K({\mathsf{LC}}_{L}) & \qquad & \qquad & \text{for }L/\mathbf{Q}%
_{p}\text{ a finite extension,}\\
I_{\mathbf{Z}_{p}}L_{K(1)}K(\mathcal{O}_{S}) & \overset{\sim}{\longrightarrow}
& L_{K(1)}K({\mathsf{LC}}_{\mathcal{O}_{S}}) & \qquad & \qquad & \text{for
}\mathcal{O}_{S}\text{ the }S\text{-integers of }F\text{.}%
\end{array}
\label{lhub1}%
\end{equation}
In the first two cases, there are no topological $L$-modules which are also
real vector spaces (in both cases any $x$ must satisfy $p^{m}x\longrightarrow
0$ when $m\longrightarrow+\infty$, but this is false over the reals). Hence,
there is no difference between ${\mathsf{LC}}_{L}$ or ${\mathsf{LCA}}_{L}$.
The $K$-theory spectrum $K({\mathsf{LCA}}_{\mathcal{O}_{S}})$ exists before
$K(1)$-localization, and outputs the automorphic/idelic objects as one would
expect them to be from class field theory in $\pi_{1}$:%
\[%
\begin{array}
[c]{lccl}%
\pi_{1}K({\mathsf{LCA}}_{L})\cong\mathbf{Z} & \qquad & \qquad & \text{for
}L/\mathbf{F}_{q}\text{ a finite extension,}\\
\pi_{1}K({\mathsf{LCA}}_{L})\cong L^{\times} & \qquad & \qquad & \text{for
}L/\mathbf{Q}_{p}\text{ a finite extension,}\\
\pi_{1}K({\mathsf{LCA}}_{\mathcal{O}_{S}})\cong\dfrac{\mathbb{A}^{\times}%
}{F^{\times}\cdot\prod_{w\notin S}\mathcal{O}_{w}^{\times}} & \qquad & \qquad
& \text{for }\mathcal{O}_{S}\text{ the }S\text{-integers of }F\text{,}%
\end{array}
\]
where $\mathbb{A}$ denotes the ad\`{e}les of $F$.

The uniformity of the formulation of dualities in Eq. \ref{lhub1} will
inevitably entice everyone to hope for a more general all-encompassing
statement. Note also that the above fields have different \'{e}tale
cohomological dimensions, yet no reference to this dimension is seen in the
duality statements. This suggests that the general statement ought to be
entirely independent of such dimensions, as long as they are finite.

\textit{Higher Dimensions:} Saito \cite{MR1045856} and Geisser--Schmidt
\cite{MR3867292} have generalized classicial Poitou--Tate duality in \'{e}tale
cohomology to higher-dimensional arithmetic schemes $\pi\colon
\mathcal{X\rightarrow O}_{S}$ smooth (for Saito) or flat regular (for
Geisser--Schmidt) of constant relative dimension $d$ over $\mathcal{O}_{S}$.
In the smooth case it amounts to using the \'{e}tale $6$-functor formalism
combining $j_{!}$ from Poitou--Tate duality with Poincar\'{e} duality for
$\pi_{!}$, and in the flat regular case using the motivic dualizing complex of
Geisser \cite{MR2680487}. They obtain a perfect pairing%
\begin{equation}
H^{s}(\mathcal{X},\mathcal{F})\otimes H_{c}^{2d+3-s}(\mathcal{X}%
,\mathcal{F}^{\ast})\longrightarrow H_{c}^{2d+3}(\mathcal{X},\mathbf{Z}%
/p^{r}(d+1))\overset{\operatorname*{tr}}{\longrightarrow}\mathbf{Z}/p^{r}
\label{lint5}%
\end{equation}
of finite abelian groups for $\mathcal{F}$ an \'{e}tale $\mathbf{Z}/p^{r}%
$-module sheaf on $\mathcal{X}$ and the generalized dual $\mathcal{F}^{\ast
}:=\operatorname*{RHom}(\mathcal{F},\mathbf{Z}/p^{r}(d+1))$. One can just take
$\mathcal{X}:=\mathcal{O}_{S}$ so that $d=0$ and this specializes to classical
Poitou--Tate duality.

Obviously, one must wonder whether there is a $K$-theoretic spectral lift of
this higher-dimensional Poitou--Tate duality in the spirit of
Blumberg--Mandell. Moreover, how should one define ${\mathsf{LC}}%
_{\mathcal{X}}$ so that the underlying pairing can again just taken to be%
\begin{equation}
\mathsf{Mod}_{\mathcal{X}}\times{\mathsf{LC}}_{\mathcal{X}}\longrightarrow
{\mathsf{LC}}_{\mathcal{X}} \label{lint4}%
\end{equation}
as in Eq. \ref{lint3}? Here, $\mathsf{Mod}_{\mathcal{X}}$ should be coherent
$\mathcal{O}_{\mathcal{X}}$-module sheaves and ${\mathsf{LC}}_{\mathcal{X}}$,
presumably, some kind of locally compact $\mathcal{O}_{\mathcal{X}}$-module
sheaves on $\mathcal{X}$. And then there should be a $!$-pushforward%
\[
(j\circ\pi)_{!}\colon{\mathsf{LC}}_{\mathcal{X}}\longrightarrow\mathsf{LC}%
_{\mathcal{O}_{F}}%
\]
as in Eq. \ref{lhuv3}. As before, we would expect this to exist already on the
Nisnevich level, but to agree with the \'{e}tale $!$-pushforward\footnote{by
\emph{\'{e}tale }$!$\emph{-pushforward} we mean the co-unit transformation
induced by $f_{!}f^{!}\longrightarrow1$ for the adjoint pair $(f_{!},f^{!})$
of the \'{e}tale $6$-functor formalism.} once we go $K(1)$-local or \'{e}tale hypersheafify.

In this paper, we present a positive and a negative result regarding these
expectations. The positive result is as follows:

\begin{theoremann}
\label{introThmA}Let $S$ be finite such that $\frac{1}{p}\in\mathcal{O}_{S}$.
Suppose $\pi\colon\mathcal{X\rightarrow O}_{S}$ is a smooth proper arithmetic
scheme. Define the Lurie tensor product%
\[
\underline{{\mathsf{LC}}}_{\mathcal{X}}:=\operatorname*{Perf}\mathcal{X}%
\otimes_{\mathcal{O}_{S}}{\mathsf{LC}}_{\mathcal{O}_{S}}\text{.}%
\]
Then there is a $!$-pushforward exact functor%
\[
\rho_{!}\colon\underline{{\mathsf{LC}}}_{\mathcal{X}}\longrightarrow
\underline{{\mathsf{LC}}}_{\mathbf{Z}[\frac{1}{p}]}%
\]
and a natural pairing as in Eq. \ref{lint4},%
\[
L_{K(1)}K(\operatorname*{Perf}\mathcal{X})\otimes L_{K(1)}K(\underline
{{\mathsf{LC}}}_{\mathcal{X}})\longrightarrow L_{K(1)}K(\underline
{{\mathsf{LC}}}_{\mathcal{X}})\overset{\rho_{!}}{\longrightarrow}%
L_{K(1)}K(\underline{{\mathsf{LC}}}_{\mathbf{Z}[\frac{1}{p}]})\overset
{\operatorname*{tr}}{\longrightarrow}I_{\mathbf{Z}_{p}}\mathbb{S}_{\widehat
{p}}\text{.}%
\]

\begin{enumerate}
\item This is a perfect pairing.

\item There are descent spectral sequences for $K(\mathcal{X})$ and
$K(\underline{{\mathsf{LC}}}_{\mathcal{X}})$, the pairing induces a pairing of
spectral sequences, and on the $E_{2}$-page, we recover the original
higher-dimensional Poitou--Tate pairing of Eq. \ref{lint5}.
\end{enumerate}
\end{theoremann}

We conjecture that after identifying $L_{K(1)}K(\underline{{\mathsf{LC}}%
}_{\mathcal{X}})$ with the global sections of compactly supported
Bott-inverted (\'{e}tale hypersheaf) $K$-theory, Prop. \ref{prop_Equiv1}, the
above functor $\rho_{!}$ between categories induces the same map as the
\'{e}tale $!$-pushforward along $\rho\colon\mathcal{X}\longrightarrow
\operatorname*{Spec}\mathbf{Z}[\frac{1}{p}]$. We do not prove this, but verify
it to the extent needed to establish the above theorem.

Equation \ref{llv3} generalizes as follows:

\begin{corollaryann}
Under the assumptions of Theorem \ref{introThmA}, the pairing induces an
equivalence%
\[
I_{\mathbf{Z}_{p}}L_{K(1)}K(\mathcal{X})\overset{\sim}{\longrightarrow
}L_{K(1)}K(\underline{{\mathsf{LC}}}_{\mathcal{X}})\text{.}%
\]

\end{corollaryann}

Compare this to what an optimist could expect to be true: There is a (mild)
smoothness assumption and a rather strong properness assumption, while only
regular flat should be needed. Indeed, the properness assumption cannot be
removed using the techniques of this paper since the definition of
$\underline{{\mathsf{LC}}}_{\mathcal{X}}$ is only the correct one under this hypothesis.

Still, the properness assumption is far weaker than asking for properness over
$\operatorname*{Spec}\mathbf{Z}$ (the only smooth proper schemes over
$\mathbf{Z}$ known to the author all motivically reduce to shifts and twists
of Tate motives). But our $\mathcal{X}$ could for example be an integral model
of an elliptic curve with the places of bad reduction put into $S$. This is
not motivically cellular.

The second result of this paper is of negative nature: The Lurie tensor
product $\underline{{\mathsf{LC}}}_{\mathcal{X}}=\operatorname*{Perf}%
\mathcal{X}\otimes_{\mathcal{O}_{S}}{\mathsf{LC}}_{\mathcal{O}_{S}}$ is very
non-explicit. Can you say what an object in this category is? One should
wonder whether $\underline{{\mathsf{LC}}}_{\mathcal{X}}$ can also be taken to
be locally compact $\mathcal{O}_{\mathcal{X}}$-module sheaves. We proceed as
follows: on an affine open $U$ define $K{\mathsf{LC}}_{\mathcal{X}%
,\operatorname*{pre}}^{\operatorname*{naive}}(U)$ to be the $K$-theory of the
category of locally compact $\mathcal{O}_{\mathcal{X}}(U)$-modules. Then
define $K{\mathsf{LC}}_{\mathcal{X}}^{\operatorname*{naive}}$ via Zariski
co-descent\footnote{Co-descent is the correct thing to demand: The value-wise
dual $U\mapsto I_{\mathbf{Z}_{p}}K(U)$ of a Zariski sheaf should be a Zariski
co-sheaf. In fact, one could consider the Nisnevich topology here, but we
stick to the more classical viewpoint.} from affine opens on $\mathcal{X}$. We
show that on sufficiently small affine opens the value of this co-sheaf agrees
with $K{\mathsf{LC}}_{\mathcal{X},\operatorname*{pre}}^{\operatorname*{naive}%
}(U)$, but we do not know if this is true for all affines. If $\pi
\colon\mathcal{X\rightarrow O}_{S}$ has relative dimension zero,
$K{\mathsf{LC}}_{\mathcal{X}}^{\operatorname*{naive}}(\mathcal{X})$ agrees
with $K(\underline{{\mathsf{LC}}}_{\mathcal{X}})$. However, if we allow
ourselves to look at the case of $S$ being all places, the following strongly
suggests that $K{\mathsf{LC}}_{\mathcal{X}}^{\operatorname*{naive}}$ is the
wrong object to work with:

Write $F_{v}$ for the metric completion of the number field $F$ at a place
$v\in S$.

\begin{theoremann}
\label{introThmB}Let $S$ be the set of all places and suppose $\pi
\colon\mathcal{X}\rightarrow\operatorname*{Spec}F$ is an integral separated
scheme of finite type over the number field $F$. Let $U\subseteq\mathcal{X}$
be any open. Then $K{\mathsf{LC}}_{\mathcal{X}}^{\operatorname*{naive}}(U)$
agrees with%
\[
\operatorname{cofib}\left(  \bigoplus\limits_{x\in U_{(0)}}K\left(
\kappa(x)\right)  \longrightarrow\left.  \underset{v\in S\;}{\prod
\nolimits^{\prime}}\right.  \bigoplus\limits_{x\in U_{v,(0)}}K\left(
\kappa(x)\right)  :\bigoplus\limits_{x\in U_{v,(0)}}K(\mathcal{O}_{\kappa
(x)})\right)  \text{,}%
\]
where $U_{(0)}$ denotes the set of closed points in the open $U$,
$U_{v}:=U\times_{F}F_{v}$ and $U_{v,(0)}$ the closed points of $U_{v}$, and
$\kappa(x)$ the respective residue fields. For $x\in U_{v,(0)}$, the residue
field $\kappa(x)$ is a finite field extension of $F_{v}$, and $\mathcal{O}%
_{\kappa(x)}$ denotes its ring of integers (or $F_{v}$ itself in the case $v$
is a real or complex place).
\end{theoremann}

The restricted product is to be read as in number theory:%
\[
\left.  \underset{v\in S\;}{\prod\nolimits^{\prime}}\right.  A_{v}%
:B_{v}:=\operatorname*{colim}_{T\subseteq S\text{, }\#T<\infty}\left(
\prod_{v\in T}A_{v}\oplus\prod_{v\in S\setminus T}B_{v}\right)  \text{,}%
\]
the colimit indexed over the poset of finite subsets of $S$ under inclusion.
This colimit makes sense even when $A_{v},B_{v}$ are spectra and we are given
maps $B_{v}\rightarrow A_{v}$. These are the maps $K(\mathcal{O}_{\kappa
(v)})\rightarrow K(\kappa(x))$, ring of integers inside its field fraction, in
the theorem. The proof of Theorem \ref{introThmB} relies on our earlier work
jointly with Arndt \cite{kthyartin}. No smoothness or properness is needed for
this result. We see that $K{\mathsf{LC}}_{\mathcal{X}}^{\operatorname*{naive}%
}$ is evidently far away from $K(\underline{{\mathsf{LC}}}_{\mathcal{X}})$. To
wit, once our scheme has infinitely many closed points, the \'{e}tale
cohomology groups of the descent spectral sequence for $L_{K(1)}K{\mathsf{LC}%
}_{\mathcal{X}}^{\operatorname*{naive}}/p^{r}$ have the same direct sums over
all closed points, making these groups absurdly big.\medskip

In the future, we plan to prove a version of Theorem \ref{introThmA}\ without
the properness assumption, requiring a better definition of $\underline
{{\mathsf{LC}}}_{\mathcal{X}}$. Using Geisser--Schmidt \cite{MR3867292}, it
should also be possible to permit $S$ to be an infinite set of places and to
replace $\pi$ smooth by merely being flat finite type and $\mathcal{X}$
regular. Following \cite{cho2025} it might also be possible to incorporate the
prime $p=2$ into the whole picture. A function field version ought to exist,
where $\pi$ gets replaced by a smooth fibration of a variety to a punctured
curve over a finite field. Going into more exotic directions, there also ought
to be a version for Azumaya sheaves on $\mathcal{X}$, as suggested by our
joint work with Henrard and van Roosmalen \cite{MR4358282,noncomclassgroup}.

\section{Setup}

\subsection{Running Conventions\label{sect_Setup}}

Fix a universe once and for all. All our rings will be unital and commutative,
and ring homomorphisms must preserve the ring unit. Write $\operatorname*{Cat}%
\nolimits_{\infty}^{\operatorname*{perf}}$ for the $\infty$-category of small
idempotent complete stable $\infty$-categories and exact functors.

We also fix an \textit{odd} prime $p$, a number field $F$, and a possibly
infinite set of places $S$ in $F$. Write $\mathcal{O}_{F}$ for the ring of
integers in $F$, and $S_{\infty}$ for its set of infinite places. If $v$ is a
place, $F_{v}$ denotes the metric completion of $F$ at $v$. Whenever $v$ is a
finite place, $\mathcal{O}_{v}$ denotes the ring of integers in $F_{v}$, or
equivalently the metric completion of $\mathcal{O}_{F}$ at $v$. We use the
convention that $\mathcal{O}_{v}:=F_{v}$ if $v$ is an infinite place. We write
$\mathcal{O}_{S}$ for the ring of $S$-integers. It can be defined as%
\begin{equation}
\mathcal{O}_{S}:=\left\{  x\in F\mid x\in\mathcal{O}_{v}\text{ for all
}v\notin S\right\}  \text{.} \label{l_ringSIntegers}%
\end{equation}
If $S$ contains all places, $\mathcal{O}_{S}=F$. We will always assume that
$S$ is chosen such that $\frac{1}{p}\in\mathcal{O}_{S}$ and such that
$S_{\infty}\subseteq S$. Let%
\begin{equation}
\pi\colon\mathcal{X}\longrightarrow\mathcal{S}:=\operatorname*{Spec}%
\mathcal{O}_{S} \label{l_q1}%
\end{equation}
be a smooth and separated finite type $\mathcal{S}$-scheme of constant
relative dimension $d\geq0$. By an \emph{arithmetic scheme} over $\mathcal{S}$
we mean any such $\pi$. Define $\mathcal{X}_{v}:=\mathcal{X}\times
_{\operatorname*{Spec}\mathcal{O}_{S}}\operatorname*{Spec}F_{v}$. Writing
$s_{v}$ for the scheme map induced from $\mathcal{O}_{S}\subseteq F\subseteq
F_{v}$, and letting $i_{v}:=\mathcal{X}\times s_{v}$, we obtain the pullback
diagram%
\begin{equation}%
%%%%%%%%%%
%\begin{tikzcd}
%	{\mathcal{X}_v} & {\mathcal{X}} \\
%	{\operatorname{Spec}F_v} & {\mathcal{S}.}
%	\arrow["{i_v}", from=1-1, to=1-2]
%	\arrow[from=1-1, to=2-1]
%	\arrow["\pi", from=1-2, to=2-2]
%	\arrow["{s_v}"', from=2-1, to=2-2]
%\end{tikzcd}
%}} }%
%%%%%%%%%%%%%%%%%
{
\begin{tikzcd}
	{\mathcal{X}_v} & {\mathcal{X}} \\
	{\operatorname{Spec}F_v} & {\mathcal{S}.}
	\arrow["{i_v}", from=1-1, to=1-2]
	\arrow[from=1-1, to=2-1]
	\arrow["\pi", from=1-2, to=2-2]
	\arrow["{s_v}"', from=2-1, to=2-2]
\end{tikzcd}
}
%%%%%%%%%%%%%%%%%
\label{l_d3}%
\end{equation}
The term \textquotedblleft$K$-theory\textquotedblright\ always refers to
non-connective $K$-theory, as in \cite{MR2762556, MR3070515}.

\subsection{Arguments on the motivic side}

\subsubsection{Generalities\label{subsubsect_Generalities}}

Let $R$ be a commutative ring\footnote{in the old-school sense of
\cite{MR1011461}, but in order to line it up with the viewpoint of
\cite{LurieHA,MR3607274}, we should regard it as a discrete $\mathbb{E}%
_{\infty}$-ring for the duration of this section.}. We shall employ the
calculus of $R$-linear stable $\infty$-categories and non-commutative
localizing motives over the base $R$. All we need has been developed and laid
out in detail by Hoyois, Scherotzke and Sibilla in \cite[\S 4-5]{MR3607274}
(see also \cite[Appendix B]{linmota}).

Their framework is far more general than what we need here, so let us give a
contracted overview covering the very few tools we need: Let $\mathcal{E}%
:=\operatorname*{Perf}(R)\in\operatorname*{CAlg}(\operatorname*{Cat}%
\nolimits_{\infty}^{\operatorname*{perf}})$ be the small, stable, idempotent
complete, symmetric monoidal $\infty$-category of perfect complexes over $R$.
Its tensor product is bi-exact. We may write $\mathbf{1}_{\mathcal{E}}$ for
its tensor unit, or more explicitly take the complex $\mathbf{1}_{\mathcal{E}%
}:=\left[  \cdots0\rightarrow R\rightarrow0\cdots\right]  $ with $R$
concentrated in degree zero. Every object is dualizable, $\mathcal{F}^{\vee
}:=\operatorname*{Hom}(\mathcal{F},\mathbf{1}_{\mathcal{E}})$ is the dual, in
particular $\mathcal{E}$ has a rigid tensor structure.

Since this is the only case we need, we slightly simplify the notation of
\cite{MR3607274} and write

\begin{enumerate}
\item $\operatorname*{Cat}\nolimits_{\infty}^{\operatorname*{perf}%
}(R):=\operatorname*{Mod}_{\mathcal{E}}(\operatorname*{Cat}\nolimits_{\infty
}^{\operatorname*{perf}})$, and call its objects $R$\emph{-linear} (small
idempotent complete) \emph{stable }$\infty$\emph{-categories},\footnote{in
\cite[Definition 4.3]{MR3607274} this would be denoted by $\operatorname*{Cat}%
\nolimits_{\infty}^{\operatorname*{perf}}(\mathcal{E})$.}

\item and $\otimes_{R}$ for its symmetric monoidal $\infty$%
-structure.\footnote{which would be denoted by $\otimes_{\mathcal{E}}$
\textit{loc. cit.}}
\end{enumerate}

\begin{remark}
[{\cite[\S 4.1]{MR3607274}}]\label{rmk_BiexactFun}Write $\operatorname*{Fun}%
\nolimits_{\mathcal{E}}^{\operatorname*{ex}}(\mathsf{A},\mathsf{B})$ for the
full subcategory of functors which are exact and $\mathcal{E}$-linear. There
is an adjunction%
\[
\operatorname*{Fun}\nolimits_{\mathcal{E}}^{\operatorname*{ex}}(\mathsf{A}%
\otimes_{R}\mathsf{B},\mathsf{C})\overset{\sim}{\longrightarrow}%
\operatorname*{Fun}\nolimits_{\mathcal{E}}^{\operatorname*{ex}}(\mathsf{A}%
,\operatorname*{Fun}\nolimits_{\mathcal{E}}^{\operatorname*{ex}}%
(\mathsf{B},\mathsf{C}))
\]
for all $\mathsf{A},\mathsf{B},\mathsf{C}$ in $\operatorname*{Cat}%
\nolimits_{\infty}^{\operatorname*{perf}}(R)$. It allows us to define
$R$-linear functors out of $\mathsf{A}\otimes_{R}\mathsf{B}$ by setting up
functors $\mathsf{A}\times\mathsf{B}\longrightarrow\mathsf{C}$, which are
exact and $R$-linear in both arguments.
\end{remark}

A sequence%
\begin{equation}
\mathsf{A}\longrightarrow\mathsf{B}\longrightarrow\mathsf{C}
\label{l_def_exseq}%
\end{equation}
in $\operatorname*{Cat}\nolimits_{\infty}^{\operatorname*{perf}}(R)$ is called
\emph{exact} if its underlying sequence in $\operatorname*{Cat}%
\nolimits_{\infty}^{\operatorname*{perf}}$ (i.e., forgetting the $R$-linear
structure) is exact (see \cite[Definition 5.3]{MR3607274}). In particular,
this means it is also an exact sequence in the sense of
Blumberg--Gepner--Tabuada \cite{MR3070515}.

\begin{remark}
All in all, the forgetful functor $\operatorname*{Cat}\nolimits_{\infty
}^{\operatorname*{perf}}(R)\rightarrow\operatorname*{Cat}\nolimits_{\infty
}^{\operatorname*{perf}}$ preserves finite limits and colimits and preserves
and reflects exactness. However, the monoidal structures are incompatible.
\end{remark}

In light notational variation from \cite[Definition 5.14]{MR3607274}, write
$\left.  \mathbf{Mot}(R)\right.  $ for the $\infty$\emph{-category of
localizing }$R$\emph{-linear motives}\footnote{We have no use for additive
motives in this text.}. This can be thought of as non-commutative motives as
in \cite{MR3070515}, just so that everything is additionally $R$-linear
(instead of having spectra as the `base'). Write%
\[
\mathcal{U}^{\operatorname*{loc}}\colon\operatorname*{Cat}\nolimits_{\infty
}^{\operatorname*{perf}}(R)\longrightarrow\left.  \mathbf{Mot}(R)\right.
\]
for the universal localizing invariant, as in \cite[Theorem 5.17]{MR3607274}.
We shall need the following properties, which go hand in hand (\cite[Theorem
5.18]{MR3607274}):\newline(1) One can enrich $\left.  \mathbf{Mot}(R)\right.
$ to a symmetric monoidal $\infty$-category so that $\mathcal{U}%
^{\operatorname*{loc}}$ is symmetric monoidal.\newline(2) If $\mathsf{A}%
\longrightarrow\mathsf{B}\longrightarrow\mathsf{C}$ is exact in
$\operatorname*{Cat}\nolimits_{\infty}^{\operatorname*{perf}}(R)$ and
$\mathsf{Z}\in\operatorname*{Cat}\nolimits_{\infty}^{\operatorname*{perf}}%
(R)$, then%
\[
\mathsf{A}\otimes_{R}\mathsf{Z}\longrightarrow\mathsf{B}\otimes_{R}%
\mathsf{Z}\longrightarrow\mathsf{C}\otimes_{R}\mathsf{Z}%
\]
is exact.\newline(3) Tensoring with $\mathcal{E}$ acts as the tensor unit,
$\mathsf{A}\otimes_{R}\operatorname*{Perf}(R)\overset{\sim}{\longrightarrow
}\mathsf{A}$.

Suppose $\rho\colon X\rightarrow\operatorname*{Spec}R$. Then $\rho$ renders
$\operatorname*{Perf}(X)$ a $\operatorname*{Perf}(R)$-module, and we may
regard $\operatorname*{Perf}(X)$ as an object in $\operatorname*{Cat}%
\nolimits_{\infty}^{\operatorname*{perf}}(R)$.

If $X_{i}$ are qcqs schemes and $p_{i}\colon X_{i}\rightarrow
\operatorname*{Spec}R$ are \textit{flat}, write $X_{1}\times_{R}X_{2}%
:=X_{1}\times_{\operatorname*{Spec}R}X_{2}$ for the ordinary fiber product of
schemes as in \cite{MR0463157}\footnote{when reading Lurie, one must be
careful that $X_{1}\times_{R}X_{2}$ (in the equation below this line) could
also refer to a derived product as in spectral algebraic geometry. However,
when everything is flat, this possible pitfall is always safely avoided. This
is the only reason for our flatness assumption here.}. Then we also have the
equivalence%
\begin{equation}
\operatorname*{Perf}(X_{1})\otimes_{R}\operatorname*{Perf}(X_{2})\overset
{\sim}{\longrightarrow}\operatorname*{Perf}(X_{1}\times_{R}X_{2}) \label{l_d2}%
\end{equation}
in $\operatorname*{Cat}\nolimits_{\infty}^{\operatorname*{perf}}(R)$, which is
induced from the exact functor%
\[
\mathcal{F}_{1}\otimes\mathcal{F}_{2}\longmapsto p_{1}^{\ast}\mathcal{F}%
_{1}\otimes_{\mathcal{O}_{X_{1}\times_{R}X_{2}}}p_{2}^{\ast}\mathcal{F}%
_{2}\text{.}%
\]
Exact categories, as described in B\"{u}hler \cite{MR2606234}, are a
convenient source of stable $\infty$-categories. We shall use this in the
following format:

An $R$\emph{-linear exact category} is an exact category enriched in
$R$-modules. Every small $R$-linear exact category $\mathsf{E}$ gives rise to
an $R$-linear dg category of complexes (for boundedness conditions $\ast
\in\{+,-,b\}$), which we may call $\operatorname*{D}%
\nolimits_{\operatorname*{dg}}^{\ast}(\mathsf{E})$, defined as%
\[
\operatorname*{D}\nolimits_{\operatorname*{dg}}^{\ast}(\mathsf{E}%
):=\operatorname*{K}\nolimits^{\ast}(\mathsf{E})/\operatorname*{Ac}%
\nolimits^{\ast}(\mathsf{E})\text{,}%
\]
where $\operatorname*{K}\nolimits^{\ast}(\mathsf{E})$ is the category of
suitably bounded complexes $\cdots\longrightarrow X_{i}\longrightarrow\cdots$
with $X_{i}\in\mathsf{E}$ (which does not depend on the exact structure of
$\mathsf{E}$ at all and can be formed for any additive $R$-linear category),
and $\operatorname*{Ac}\nolimits^{\ast}(\mathsf{E})$ denotes the complexes
which need to be considered as acyclic given the exact structure on
$\mathsf{E}$. The details for this construction can be found in \cite[\S 10]%
{MR2606234}.

Then $\operatorname*{D}\nolimits_{\operatorname*{dg}}^{\ast}(\mathsf{E})$ is a
pre-triangulated $R$-linear dg category and its dg nerve, which we denote by
$\operatorname*{D}\nolimits_{\infty}^{\ast}(\mathsf{E})$, is a small stable
$R$-linear $\infty$-category. If $\mathsf{E}$ is idempotent complete in the
sense of additive categories (\cite[\S 6]{MR2606234}), $\operatorname*{D}%
\nolimits_{\infty}^{b}(\mathsf{E})$ is also idempotent complete by
\cite[Theorem 2.8]{MR1813503} and therefore $\operatorname*{D}%
\nolimits_{\infty}^{b}(\mathsf{E})\in\operatorname*{Cat}\nolimits_{\infty
}^{\operatorname*{perf}}(R)$.

\begin{example}
\label{ex_DbPfVersusPerfectComplexes}For any ring $R$, write $\mathsf{P}%
_{f}(R)$ for the exact category of finitely generated projective $R$-modules.
Then $\operatorname*{D}\nolimits_{\infty}^{b}(\mathsf{P}_{f}(R))\cong
\operatorname*{Perf}(R)$. This uses that perfect complexes are
quasi-isomorphic to two-sided bounded complexes of finitely generated
projective $R$-modules, and conversely all these are compact objects in
$\operatorname*{D}\nolimits_{\infty}(\mathsf{Mod}_{R})$.
\end{example}

\section{Locally compact modules over rings of integers}

Let $R$ be a commutative ring. For us, compact spaces are tacitly required to
be Hausdorff.

\begin{definition}
\label{def_LCA_OS}We write $\mathsf{LCA}_{R}$ for the category whose

\begin{enumerate}
\item objects are $R$-modules $X$ along with a locally compact topology $\tau$,

\item and morphisms are continuous $R$-linear maps.
\end{enumerate}
\end{definition}

Alternatively, one may regard this as a category of LCA (locally compact
abelian) groups equipped with an $R$-module structure such that each
$\alpha\in R$ acts linearly and continuously,%
\[
\mathsf{LCA}_{R}\cong\operatorname*{Fun}([R],\mathsf{LCA}_{\mathbf{Z}%
})\text{,}%
\]
where $[R]$ denotes the $1$-category with a single object $\ast$ with
endomorphism ring $\operatorname*{End}(\ast):=R$. However, note that also when
regarded from this viewpoint, we must demand that homomorphisms in
$\mathsf{LCA}_{\mathbf{Z}}$ are \textit{continuous} group homomorphisms of LCA\ groups.

This is a quasi-abelian $R$-linear category\footnote{See \cite[\S 4]%
{MR2606234} for the concept of quasi-abelian categories.}. In particular, all
kernels and cokernels exist, but for morphisms $\varphi\colon X\rightarrow Y$
the induced map $\overline{\varphi}$ in%
\[
X\rightarrow X/\ker\varphi\overset{\overline{\varphi}}{\longrightarrow
}\operatorname*{im}\varphi\rightarrow Y
\]
need not be an isomorphism (this is the key failing which causes
$\mathsf{LCA}_{R}$ not to be an abelian category). Every quasi-abelian
category has a canonical exact structure (\cite[Prop. 4.4]{MR2606234}).\ It is
easy to make it explicit for $\mathsf{LCA}_{R}$:

(1) Admissible monics are the injective closed maps,

(2) Admissible epics are the surjective open maps.

We write $\operatorname*{Hom}\nolimits_{R}(X,Y)$ for morphisms in this
category, but remember that besides being $R$-linear, this tacitly includes
the condition to be continuous.

\begin{example}
As another indicator how this category fails to be abelian, if $\mathbf{R}%
_{\operatorname*{disc}}$ denotes the real numbers with the discrete topology,
$\operatorname*{id}\colon\mathbf{R}_{\operatorname*{disc}}\rightarrow
\mathbf{R}$ is a map in $\mathsf{LCA}_{\mathbf{Z}}$, and it is both an
admissible monic and admissible epic, yet fails to be an isomorphism. The two
objects $\mathbf{R}_{\operatorname*{disc}}$, $\mathbf{R}$ only differ on the
level of the underlying topology.
\end{example}

Whenever we talk about \emph{algebraic }$R$-modules, this is meant to stress
that we either forget an existing topology or have not even defined a topology
on the module to start with.

\subsection{Symmetric monoidal structure over the base $R$%
\label{subsubsect_SymmMonoidalLCAOS}}

Suppose $X,Y,Z\in\mathsf{LCA}_{R}$. We always equip $\operatorname*{Hom}%
\nolimits_{R}(X,Y)$ with the compact-open topology. Then $\operatorname*{Hom}%
\nolimits_{R}(X,Y)$ is always a topological $R$-module, but depending on $X$
and $Y$, may or may not be locally compact. Whenever it is, we can treat
$\operatorname*{Hom}\nolimits_{R}(X,Y)$ as an inner Hom-object. But as this is
not generally possible for all objects $X$ and $Y$, there is no satisfactory
symmetric monoidal structure on $\mathsf{LCA}_{R}$. Moskowitz \cite[\S 4]%
{MR0215016} contains various sufficient conditions to ensure
$\operatorname*{Hom}\nolimits_{R}(X,Y)$ is locally compact. Guided by the
Hom-tensor adjunction, one can use the partially defined inner Hom-objects to
set up a partially defined ($1$-categorical) symmetric monoidal structure
$X\otimes_{R}Y$ for objects $X,Y\in\mathsf{LCA}_{R}$. Its defining property is
to realize the adjunction%
\[
\operatorname*{Hom}\nolimits_{R}(X\otimes_{R}Y,Z)\cong\operatorname*{Hom}%
\nolimits_{R}(X,\operatorname*{Hom}\nolimits_{R}(Y,Z))
\]
whenever all involved objects exist and lie in $\mathsf{LCA}_{R}$. Inevitably,
the reader will anticipate that there is a reasonable symmetric monoidal
structure lurking in the background, but closed under tensoring only in a
bigger category than $\mathsf{LCA}_{R}$. This is true, but we will not pursue
this point further in this text. There is a powerful notion of \textit{dual
object }available in $\mathsf{LCA}_{R}$, even though it does not have the
shape of an inner Hom in general: Define the \emph{Pontryagin dual} $X^{\vee
}:=\operatorname*{Hom}\nolimits_{\mathbf{Z}}(X,\mathbb{T})$, the group of all
continuous abelian group homomorphisms\footnote{read the $\operatorname*{Hom}%
\nolimits_{\mathbf{Z}}$ as the Hom-object of $\mathsf{LCA}_{\mathbf{Z}}$},
where $\mathbb{T}:=\mathbf{R}/\mathbf{Z}$ is the circle with its usual
topology.\footnote{It feels noteworthy that already Krull had been looking for
a version of Pontryagin duality for locally compact modules over Dedekind
domains, worked out to some extent by his student St\"{o}hr
\cite{MR262223,MR291157} as early as 1969.} Then $X^{\vee}$ carries a
continuous $R$-module action via precomposition, i.e.,%
\[
(\alpha\cdot\varphi)(x):=\varphi(\alpha\cdot x)\qquad\text{for }\alpha\in
R\text{ and }x\in X\text{,}%
\]
so that $X^{\vee}\in\mathsf{LCA}_{R}$. The natural map to the double dual
$X\longrightarrow X^{\vee\vee}$ is always an isomorphism. Moreover,
$\operatorname*{Hom}_{R}(X,Y)\cong\operatorname*{Hom}_{R}(Y^{\vee},X^{\vee})$.
For $R=\mathbf{Z}$, the circle $\mathbb{T}$ is naturally a locally compact
$\mathbf{Z}$-module and then this dual happens to exist in the format of an
inner Hom object, but not in general.

For all rings $R$, $(-)^{\vee}\colon\mathsf{LCA}_{R}^{\operatorname*{op}%
}\longrightarrow\mathsf{LCA}_{R}$ is an exact equivalence of exact categories.
This renders $\mathsf{LCA}_{R}$ an exact category with duality in the sense of
\cite[Definition 2.1]{MR2600285}, \cite{MR4293796}. For $R=\mathbf{Z}$, it
also shows that $\mathbb{T}$ is an injective object.

\begin{example}
\label{ex_PropertiesInLCAOS}Suppose $R=\mathcal{O}_{S}$, the ring of
$S$-integers as in \S \ref{sect_Setup}. As a quick introduction what relevant
objects in $\mathsf{LCA}_{\mathcal{O}_{S}}$ look like, let us summarize:

\begin{enumerate}
\item All discrete $\mathcal{O}_{S}$-modules naturally form objects in
$\mathsf{LCA}_{\mathcal{O}_{S}}$, there is a fully faithful exact functor
$\mathsf{Mod}_{\mathcal{O}_{S}}\rightarrow\mathsf{LCA}_{\mathcal{O}_{S}}$.

\item Every discrete projective $\mathcal{O}_{S}$-module is also projective in
$\mathsf{LCA}_{\mathcal{O}_{S}}$.

\item However, discrete injective $\mathcal{O}_{S}$-modules will (unless they
are zero) never be injective in $\mathsf{LCA}_{\mathcal{O}_{S}}$.

\item Pontryagin duality swaps projectives with injectives, so the easiest way
to exhibit injectives is to take a discrete projective $\mathcal{O}_{S}%
$-module $X$, and then consider its dual $X^{\vee}$. This will be a compact module.

\item There are projective and injective objects in $\mathsf{LCA}%
_{\mathcal{O}_{S}}$ which do not arise as in (2) or (4), for example for
$F=\mathbf{Q}$ and $S$ the single real infinite place so that $\mathcal{O}%
_{S}=\mathbf{Z}$, the real vector spaces $\mathbf{R}^{n}$ are the only objects
in $\mathsf{LCA}_{\mathbf{Z}}$ which are both injective and projective
\cite{MR1620000}.

\item If $F$ is a number field and $S$ contains all places, the ad\`{e}le
sequence $F\hookrightarrow\mathbb{A}_{F}\twoheadrightarrow\mathbb{A}_{F}/F$ is
an injective resolution of $F$, and simultaneously a projective resolution of
$\mathbb{A}_{F}/F$.
\end{enumerate}

Facts (1)-(5) can be found in \cite[\S 4]{klca1}, and (6) is from
\cite{noncomclassgroup}.
\end{example}

\subsection{Vector modules}

Let $\mathsf{LCA}_{\mathcal{O}_{S},\mathbf{R}}$ be the full subcategory of
$\mathsf{LCA}_{\mathcal{O}_{S}}$ of objects $X$ whose underlying LCA group
admits an isomorphism to a finite-dimensional real vector space. Write
$\mathsf{Mod}_{\mathcal{O}_{S}\otimes\mathbf{R}}^{fg}$ for the abelian
category of finitely generated $(\mathcal{O}_{S}\otimes_{\mathbf{Z}}%
\mathbf{R})$-modules, without topology.

By \cite[Theorem 2.8]{MR1813503} both $\operatorname*{D}\nolimits_{\infty}%
^{b}\mathsf{LCA}_{\mathcal{O}_{S},\mathbf{R}}$ and $\operatorname*{D}%
\nolimits_{\infty}^{b}\mathsf{LCA}_{\mathcal{O}_{S}}$ are idempotent complete.

\begin{lemma}
\label{lem_TT1}The full subcategory $\mathsf{LCA}_{\mathcal{O}_{S},\mathbf{R}%
}\subseteq\mathsf{LCA}_{\mathcal{O}_{S}}$

\begin{enumerate}
\item is closed under Pontryagin duality,

\item is a Serre subcategory of $\mathsf{LCA}_{\mathcal{O}_{S}}$,

\item $\mathsf{LCA}_{\mathcal{O}_{S},\mathbf{R}}\overset{\sim}{\longrightarrow
}\mathsf{Mod}_{\mathcal{O}_{S}\otimes\mathbf{R}}^{fg}$ as abelian categories,

\item $\operatorname*{D}\nolimits_{\infty}^{b}\mathsf{LCA}_{\mathcal{O}%
_{S},\mathbf{R}}\longrightarrow\operatorname*{D}\nolimits_{\infty}%
^{b}\mathsf{LCA}_{\mathcal{O}_{S}}$ is fully faithful.
\end{enumerate}
\end{lemma}

\begin{proof}
(1) We only need to observe that the Pontryagin dual of $\mathbf{R}$ is
isomorphic to $\mathbf{R}$ by the standard identification $\mathbf{R}%
\rightarrow\operatorname*{Hom}\nolimits_{\mathbf{Z}}(\mathbf{R},\mathbb{T})$,
$x\mapsto(y\mapsto e^{2\pi ixy})$. (2) It is clear that $\mathsf{LCA}%
_{\mathcal{O}_{S},\mathbf{R}}$ is closed under extensions in $\mathsf{LCA}%
_{\mathcal{O}_{S}}$.\ We need to show that if $V\in\mathsf{LCA}_{\mathcal{O}%
_{S},\mathbf{R}}$, then once regarded as an object in $\mathsf{LCA}%
_{\mathcal{O}_{S}}$, all admissible subobjects\footnote{In an exact category,
an \emph{admissible subobject} is a subobject $U\hookrightarrow V$ where one
(then all) representing monics can be taken to be an admissible monic with
respect to the exact structure.} must also lie in $\mathsf{LCA}_{\mathcal{O}%
_{S},\mathbf{R}}$. Let $U\hookrightarrow V$ be an admissible subobject. Then
$U$ is a closed subgroup in a real vector space. Following \cite[Thm.
6]{MR0442141} this implies that $U\simeq\mathbf{Z}^{m}\oplus\mathbf{R}^{n}$
for suitable values $0\leq m,n<\infty$ in the category of LCA groups. Let
$U_{\mathbf{R}}$ be the subgroup corresponding to the $\mathbf{R}^{n}%
$-summand, hence evidently connected. As the scalar action is continuous, for
all $\alpha\in\mathcal{O}_{S}$, $\alpha U_{\mathbf{R}}$ must be a connected
subgroup of the LCA group $U$, and therefore also lie inside $U_{\mathbf{R}}$.
We conclude that $U_{\mathbf{R}}$ is an algebraic $\mathcal{O}_{S}$-submodule
of $U$. It follows that $U/U_{\mathbf{R}}$, whose underlying subgroup is
$\mathbf{Z}^{m}$, is an algebraic $\mathcal{O}_{S}$-module. Assume
$U/U_{\mathbf{R}}\neq0$. Since $\frac{1}{p}\in\mathcal{O}_{S}$, we must have
$\frac{1}{p}(U/U_{\mathbf{R}})=U/U_{\mathbf{R}}$, so $\frac{1}{p}$ is integral
over the ring $\mathbf{Z}$, which is absurd. Hence, $U/U_{\mathbf{R}}=0$ and
it follows that $U\in\mathsf{LCA}_{\mathcal{O}_{S},\mathbf{R}}$. Conversely,
if $V\twoheadrightarrow U$ is an admissible quotient, Pontryagin duality
transforms it into an admissible subobject $U^{\vee}\hookrightarrow V^{\vee}$,
and the same argument applies. (4) As $\mathcal{O}_{S}\otimes_{\mathbf{Z}%
}\mathbf{R}$ is a semisimple ring, all $\operatorname*{Ext}^{s}$-groups with
$s\geq1$ vanish in $\mathsf{Mod}_{\mathcal{O}_{S}\otimes\mathbf{R}}^{fg}$. Our
claim follows once $\operatorname*{Ext}_{\mathcal{O}_{S}\otimes_{\mathbf{Z}%
}\mathbf{R}}^{s}(X,Y)\longrightarrow\operatorname*{Ext}_{\mathsf{LCA}%
_{\mathcal{O}_{S}}}^{s}(X,Y)$ is an isomorphism for all $X,Y\in\mathsf{LCA}%
_{\mathcal{O}_{S},\mathbf{R}}$ and $s\geq0$, but this is easy to check for
$s=0$, since all continuous $\mathcal{O}_{S}$-linear maps on a
finite-dimensional real vector space must be $\mathbf{R}$-linear as
well\footnote{and conversely all $\mathbf{R}$-linear maps on a
finite-dimensional real vector space are automatically continuous}, and for
$s\geq1$ use that objects in $\mathsf{LCA}_{\mathcal{O}_{S},\mathbf{R}}$ are
projective (and injective) in $\mathsf{LCA}_{\mathcal{O}_{S}}$ \cite[Theorem
4.19 (5)]{klca1}. (3) Use the previous argument for $s=0$. It follows that the
category $\mathsf{LCA}_{\mathcal{O}_{S},\mathbf{R}}$ is equivalent, as an
exact category, to the category of finitely generated $(\mathcal{O}_{S}%
\otimes_{\mathbf{Z}}\mathbf{R})$-modules, and thus both are abelian.
\end{proof}

\begin{definition}
\label{def_C1}We define the $\mathcal{O}_{S}$-linear stable $\infty$-category
$\mathsf{LC}_{\mathcal{O}_{S}}$ as the quotient
\[
\mathsf{LC}_{\mathcal{O}_{S}}:=\operatorname*{D}\nolimits_{\infty}%
^{b}\mathsf{LCA}_{\mathcal{O}_{S}}/\operatorname*{D}\nolimits_{\infty}%
^{b}\mathsf{LCA}_{\mathcal{O}_{S},\mathbf{R}}%
\]
inside $\operatorname*{Cat}\nolimits_{\infty}^{\operatorname*{perf}%
}(\mathcal{O}_{S})$. Pontryagin duality induces an exact equivalence of stable
$\infty$-categories $(-)^{\vee}\colon\mathsf{LC}_{\mathcal{O}_{S}%
}^{\operatorname*{op}}\overset{\sim}{\longrightarrow}\mathsf{LC}%
_{\mathcal{O}_{S}}$.
\end{definition}

The quotient exists and Pontryagin duality extends thanks to Lemma
\ref{lem_TT1}.

\begin{lemma}
\label{lem_TT2}Suppose $S$ is a finite set. There is a fiber sequence%
\[
\mathcal{U}^{\operatorname*{loc}}(\operatorname*{Perf}\mathcal{O}_{S}%
)\overset{\alpha}{\longrightarrow}\bigoplus_{v\in S\setminus S_{\infty}%
}\mathcal{U}^{\operatorname*{loc}}(\operatorname*{Perf}F_{v})\overset{\beta
}{\longrightarrow}\mathcal{U}^{\operatorname*{loc}}({\mathsf{LC}}%
_{\mathcal{O}_{S}})
\]
in $\left.  \mathbf{Mot}(\mathcal{O}_{S})\right.  $, where $\alpha$ is induced
from the exact functor $X\mapsto X\otimes_{\mathcal{O}_{S}}F_{v}$.
\end{lemma}

\begin{proof}
Consider the diagram%
\[%
%%%%%%%%%%
%\begin{tikzcd}
%	&& {\bigoplus\limits_{v \in S_{\infty}} \mathcal{U}^{\operatorname{loc}
%}(\operatorname{Perf} F_{v})} && {\mathcal{U}^{\operatorname{loc} }%
%(\mathsf{LCA}_{\mathcal{O}_{S},\mathbf{R} })} \\
%	\\
%	{\mathcal{U}^{\operatorname{loc} }(\operatorname{Perf} \mathcal{O}_{S})}
%&& {\bigoplus\limits_{v \in S} \mathcal{U}^{\operatorname{loc} }%
%(\operatorname{Perf} F_{v})} && {\mathcal{U}^{\operatorname{loc} }%
%(\mathsf{LCA}_{\mathcal{O}_{S} })} \\
%	\\
%	{\mathcal{U}^{\operatorname{loc} }(\operatorname{Perf} \mathcal{O}_{S})}
%&& {\bigoplus\limits_{v \in S \setminus S_{\infty} } \mathcal{U}%
%^{\operatorname{loc} }(\operatorname{Perf} F_{v})} && {\mathcal{U}%
%^{\operatorname{loc} }(\mathsf{LC}_{\mathcal{O}_{S} }),}
%	\arrow["\lambda"', equals, from=1-3, to=1-5]
%	\arrow[from=1-3, to=3-3]
%	\arrow[from=1-5, to=3-5]
%	\arrow[from=3-1, to=3-3]
%	\arrow[equals, from=3-1, to=5-1]
%	\arrow[from=3-3, to=3-5]
%	\arrow[from=3-3, to=5-3]
%	\arrow[from=3-5, to=5-5]
%	\arrow[from=5-1, to=5-3]
%	\arrow[from=5-3, to=5-5]
%\end{tikzcd}
%}}}%
%%%%%%%%%%%%%%%%%
\adjustbox{max width=\textwidth}{
\begin{tikzcd}
	&& {\bigoplus\limits_{v \in S_{\infty}} \mathcal{U}^{\operatorname{loc}
}(\operatorname{Perf} F_{v})} && {\mathcal{U}^{\operatorname{loc} }%
(\mathsf{LCA}_{\mathcal{O}_{S},\mathbf{R} })} \\
	\\
	{\mathcal{U}^{\operatorname{loc} }(\operatorname{Perf} \mathcal{O}_{S})}
&& {\bigoplus\limits_{v \in S} \mathcal{U}^{\operatorname{loc} }%
(\operatorname{Perf} F_{v})} && {\mathcal{U}^{\operatorname{loc} }%
(\mathsf{LCA}_{\mathcal{O}_{S} })} \\
	\\
	{\mathcal{U}^{\operatorname{loc} }(\operatorname{Perf} \mathcal{O}_{S})}
&& {\bigoplus\limits_{v \in S \setminus S_{\infty} } \mathcal{U}%
^{\operatorname{loc} }(\operatorname{Perf} F_{v})} && {\mathcal{U}%
^{\operatorname{loc} }(\mathsf{LC}_{\mathcal{O}_{S} }),}
	\arrow["\lambda"', equals, from=1-3, to=1-5]
	\arrow[from=1-3, to=3-3]
	\arrow[from=1-5, to=3-5]
	\arrow[from=3-1, to=3-3]
	\arrow[equals, from=3-1, to=5-1]
	\arrow[from=3-3, to=3-5]
	\arrow[from=3-3, to=5-3]
	\arrow[from=3-5, to=5-5]
	\arrow[from=5-1, to=5-3]
	\arrow[from=5-3, to=5-5]
\end{tikzcd}
}%
%%%%%%%%%%%%%%%%%
\]
where $\lambda$ is induced from the equivalence of abelian categories of Lemma
\ref{lem_TT1} (3). The middle row is the fiber sequence of \cite[Theorem
4.30]{klca1}. \textit{Loc. cit.} this is only formulated for
$\operatorname*{Cat}\nolimits_{\infty}^{\operatorname*{perf}}$, but it is easy
to see that it lifts to the $\mathcal{O}_{S}$-linear setting, and the
definition of exactness (see Eq. \ref{l_def_exseq}) reduces to exactness in
$\operatorname*{Cat}\nolimits_{\infty}^{\operatorname*{perf}}$. The downward
arrows are fiber sequences since the initial arrows of both middle and right
column are induced from fully faithful exact functors. This is clear for the
middle, and true for the right column by Lemma \ref{lem_TT1} (4).
\end{proof}

\section{Definition via Lurie tensor product}

In this entire section we shall assume that $S$ is a \textit{finite} set, and
otherwise as in \S \ref{sect_Setup}. We will occasionally repeat this
assumption in the statement of propositions to avoid forgetting about it when
citing a statement.

Relying on the map $\pi$ of Eq. \ref{l_q1}, we have $\operatorname*{Perf}%
\mathcal{X}\in\operatorname*{Cat}\nolimits_{\infty}^{\operatorname*{perf}%
}(\mathcal{O}_{S})$, so using the intrinsic symmetric monoidal structure of
$\operatorname*{Cat}\nolimits_{\infty}^{\operatorname*{perf}}(\mathcal{O}%
_{S})$, we may proceed as follows.

\begin{definition}
\label{def_LCA_X}We call $\underline{{\mathsf{LC}}}_{\mathcal{X}%
}:=\operatorname*{Perf}\mathcal{X}\otimes_{\mathcal{O}_{S}}{\mathsf{LC}%
}_{\mathcal{O}_{S}}$ the stable $\infty$-category of (horizontally)
\emph{locally compact }$\mathcal{O}_{\mathcal{X}}$\emph{-modules}%
.\footnote{This definition works for any $\pi\colon\mathcal{X}\rightarrow
\operatorname*{Spec}\mathcal{O}_{S}$ as in our Setup \S \ref{sect_Setup}, but
philosophically only yields the correct object when $\pi$ is proper. When
$\pi$ is non-proper, its $K$-theory represents something with
\textquotedblleft compact horizontal support\textquotedblright\ (but arbitrary
support \textquotedblleft vertically\textquotedblright\ in the fibers). It is
conceptually roughly orthogonal to the compact vertical cohomology of
\cite[Ch. 1, \S 6]{MR658304}.}
\end{definition}

When $\mathcal{X}$ is affine of relative dimension zero over
$\operatorname*{Spec}\mathcal{O}_{S}$, we now have two competing definitions:
We could consider ${\mathsf{LC}}_{\mathcal{X}}$, as locally compact modules
over this extension ring modulo reals (Definition \ref{def_C1}), or take
$\underline{{\mathsf{LC}}}_{\mathcal{X}}$ as in Definition \ref{def_LCA_X}.
Corollary \ref{cor1} below will show that we obtain the same non-commutative
motive either way.

\begin{proposition}
\label{prop_W1}Suppose $S$ is finite. Then the sequence%
\[
\mathcal{U}^{\operatorname*{loc}}(\operatorname*{Perf}\mathcal{X}%
)\underset{\alpha^{\prime}}{\longrightarrow}\bigoplus\limits_{v\in S\setminus
S_{\infty}}\mathcal{U}^{\operatorname*{loc}}(\operatorname*{Perf}%
\mathcal{X}_{v})\underset{\beta^{\prime}}{\longrightarrow}\mathcal{U}%
^{\operatorname*{loc}}(\underline{{\mathsf{LC}}}_{\mathcal{X}})\text{,}%
\]
where $\alpha^{\prime}$ is induced by the exact functors $i_{v}^{\ast}$ (from
Diagram \ref{l_d3}), and $\beta^{\prime}$ constructed in the proof below, is
exact in $\left.  \mathbf{Mot}(\mathcal{O}_{S})\right.  $.
\end{proposition}

If we additionally assume that $\pi$ is proper, we may set up a little more
structure: Since $\pi$ is flat, then whenever $\pi$ is proper, we get an
induced pushforward for perfect complexes \cite[Thm. 6.1.3.2]{LurieSAG},%
\begin{equation}
\pi_{\ast}\colon\operatorname*{Perf}\mathcal{X}\longrightarrow
\operatorname*{Perf}\mathcal{O}_{S}\text{,} \label{l_h0}%
\end{equation}
and this induces a map to the respective tensor products $-\otimes
_{\mathcal{O}_{S}}{\mathsf{LC}}_{\mathcal{O}_{S}}$, which by abuse of language
we shall still denote by $\pi_{\ast}$,%
\begin{equation}
\pi_{\ast}\colon\underline{{\mathsf{LC}}}_{\mathcal{X}}\longrightarrow
{\mathsf{LC}}_{\mathcal{O}_{S}}\text{.} \label{l_h0a}%
\end{equation}

\begin{corollary}
\label{cor_v1}If additionally $\pi$ is proper, Prop. \ref{prop_W1} can be
extended to: The diagram%
\begin{equation}%
%%%%%%%%%%
%\begin{tikzcd}
%	{\mathcal{U}^{\operatorname*{loc}}(\operatorname*{Perf}\mathcal{X})}
%&& {\underset{v\in S \setminus S_{\infty}}{\bigoplus}\mathcal{U}%
%^{\operatorname*{loc}}(\operatorname*{Perf}{\mathcal{X}}_{v})} && {\mathcal
%{U}^{\operatorname*{loc}}({\underline{{\mathsf{LC}}}_{\mathcal{X}}}),} \\
%	\\
%	{\mathcal{U}^{\operatorname*{loc}}(\operatorname*{Perf}\mathcal{O} _{S})}
%&& {\underset{v\in S \setminus S_{\infty}}{\bigoplus}\mathcal{U}%
%^{\operatorname*{loc}}(\operatorname*{Perf}F_{v})} && {\mathcal{U}%
%^{\operatorname*{loc}}({{{\mathsf{LC}}}_{\mathcal{O}_{S}}}),}
%	\arrow["{{{\alpha^{\prime}}}}", from=1-1, to=1-3]
%	\arrow["{{{\pi}_{*}}}", from=1-1, to=3-1]
%	\arrow["{{{\beta^{\prime}}}}", from=1-3, to=1-5]
%	\arrow["{{{\pi}_{*}}}", from=1-3, to=3-3]
%	\arrow["{{{\pi}_{*}}}", from=1-5, to=3-5]
%	\arrow[from=3-1, to=3-3]
%	\arrow[from=3-3, to=3-5]
%\end{tikzcd}
%}} }%
%%%%%%%%%%%%%%%%%
{
\begin{tikzcd}
	{\mathcal{U}^{\operatorname*{loc}}(\operatorname*{Perf}\mathcal{X})}
&& {\underset{v\in S \setminus S_{\infty}}{\bigoplus}\mathcal{U}%
^{\operatorname*{loc}}(\operatorname*{Perf}{\mathcal{X}}_{v})} && {\mathcal
{U}^{\operatorname*{loc}}({\underline{{\mathsf{LC}}}_{\mathcal{X}}}),} \\
	\\
	{\mathcal{U}^{\operatorname*{loc}}(\operatorname*{Perf}\mathcal{O} _{S})}
&& {\underset{v\in S \setminus S_{\infty}}{\bigoplus}\mathcal{U}%
^{\operatorname*{loc}}(\operatorname*{Perf}F_{v})} && {\mathcal{U}%
^{\operatorname*{loc}}({{{\mathsf{LC}}}_{\mathcal{O}_{S}}}),}
	\arrow["{{{\alpha^{\prime}}}}", from=1-1, to=1-3]
	\arrow["{{{\pi}_{*}}}", from=1-1, to=3-1]
	\arrow["{{{\beta^{\prime}}}}", from=1-3, to=1-5]
	\arrow["{{{\pi}_{*}}}", from=1-3, to=3-3]
	\arrow["{{{\pi}_{*}}}", from=1-5, to=3-5]
	\arrow[from=3-1, to=3-3]
	\arrow[from=3-3, to=3-5]
\end{tikzcd}
}
%%%%%%%%%%%%%%%%%
\label{l_d0}%
\end{equation}
commutes and has exact rows in $\left.  \mathbf{Mot}(\mathcal{O}_{S})\right.
$.
\end{corollary}

We shall prove both the proposition and the corollary in one go.

\begin{proof}
(Hypothesis: $\pi$ need not be proper) We commence with the fiber sequence of
Lemma \ref{lem_TT2}:
\begin{equation}
\mathcal{U}^{\operatorname*{loc}}(\operatorname*{Perf}\mathcal{O}_{S}%
)\overset{\alpha}{\longrightarrow}\bigoplus_{v\in S\setminus S_{\infty}%
}\mathcal{U}^{\operatorname*{loc}}(\operatorname*{Perf}F_{v})\overset{\beta
}{\longrightarrow}\mathcal{U}^{\operatorname*{loc}}({\mathsf{LC}}%
_{\mathcal{O}_{S}})\text{,} \label{l_d1}%
\end{equation}
in $\left.  \mathbf{Mot}(\mathcal{O}_{S})\right.  $, where $\alpha$ is induced
from the exact functor $M\mapsto M\otimes_{\mathcal{O}_{S}}F_{v}$ (using that
$F_{v}$ is $\mathcal{O}_{S}$-flat), and $\beta$ from the exact functor to
regard a finite-dimensional $F_{v}$-vector space as itself, equipped with the
canonical topology induced from $F_{v}$. More concretely:\ We have
$\mathcal{U}^{\operatorname*{loc}}(\operatorname*{Perf}F_{v})\cong
\mathcal{U}^{\operatorname*{loc}}(\mathsf{Vect}_{F_{v}}^{fd})$, where
$\mathsf{Vect}_{F_{v}}^{fd}$ is the abelian category of finite-dimensional
$F_{v}$-vector spaces. Define an additive functor $\beta\colon\mathsf{Vect}%
_{F_{v}}^{fd}\rightarrow{\mathsf{LCA}}_{\mathcal{O}_{S}}\rightarrow
\mathsf{LC}_{\mathcal{O}_{S}}$ by sending $F_{v}$ to itself, equipped with the
topology coming from the metric attached to $v$ (as $v\notin S_{\infty}$, this
will be a non-archimedean topology). There is a unique extension of this
description to an additive functor, and since the source category is split
exact, as an exact functor. Equivalently, equip any finite-dimensional $F_{v}%
$-vector space with an $F_{v}$-norm and take the induced topology. It is known
that all norms on finite-dimensional $F_{v}$-spaces are equivalent and
therefore induce the same topology (\cite[Thm. 6.2.1]{MR4175370}). Now tensor
Eq. \ref{l_d1} with $\mathcal{U}^{\operatorname*{loc}}(\operatorname*{Perf}%
\mathcal{X})$ from the left, regarded $\mathcal{O}_{S}$-linearly via the fixed
map of Eq. \ref{l_q1} in $\left.  \mathbf{Mot}(\mathcal{O}_{S})\right.  $.
This yields the top row of%
\begin{equation}%
%%%%%%%%%%
%\begin{tikzcd}
%	{\mathcal{U}^{\operatorname*{loc}}(\operatorname*{Perf}\mathcal{X}%
%)\otimes_{\mathcal{O}_S}\mathcal{U}^{\operatorname*{loc}}(\operatorname
%*{Perf}\mathcal{O} _{S})} && {\mathcal{U}^{\operatorname*{loc}}(\operatorname
%*{Perf}\mathcal{X})\otimes_{\mathcal{O}_S}\underset{v\in S \setminus
%S_{\infty}}{\bigoplus}\mathcal{U}^{\operatorname*{loc}}(\operatorname
%*{Perf}F_{v})} && {\mathcal{U}^{\operatorname*{loc}}({\underline{{\mathsf{LC}%
%}}_{\mathcal{X}}})} \\
%	{\mathcal{U}^{\operatorname*{loc}}(\operatorname*{Perf}\mathcal{X})}
%&& {\underset{v\in S \setminus S_{\infty}}{\bigoplus}\mathcal{U}%
%^{\operatorname*{loc}}(\operatorname*{Perf}{\mathcal{X}}_{v})} && {\mathcal
%{U}^{\operatorname*{loc}}({\underline{{\mathsf{LC}}}_{\mathcal{X}}}),}
%	\arrow["\alpha", from=1-1, to=1-3]
%	\arrow[from=1-1, to=2-1]
%	\arrow["\beta", from=1-3, to=1-5]
%	\arrow[from=1-3, to=2-3]
%	\arrow[equals, from=1-5, to=2-5]
%	\arrow["{{\alpha^{\prime}}}"', from=2-1, to=2-3]
%	\arrow["{{\beta^{\prime}}}"'{pos=0.7}, from=2-3, to=2-5]
%\end{tikzcd}
%}
%}} }%
%%%%%%%%%%%%%%%%%
{\adjustbox{scale=0.724}{
\begin{tikzcd}
	{\mathcal{U}^{\operatorname*{loc}}(\operatorname*{Perf}\mathcal{X}%
)\otimes_{\mathcal{O}_S}\mathcal{U}^{\operatorname*{loc}}(\operatorname
*{Perf}\mathcal{O} _{S})} && {\mathcal{U}^{\operatorname*{loc}}(\operatorname
*{Perf}\mathcal{X})\otimes_{\mathcal{O}_S}\underset{v\in S \setminus
S_{\infty}}{\bigoplus}\mathcal{U}^{\operatorname*{loc}}(\operatorname
*{Perf}F_{v})} && {\mathcal{U}^{\operatorname*{loc}}({\underline{{\mathsf{LC}%
}}_{\mathcal{X}}})} \\
	{\mathcal{U}^{\operatorname*{loc}}(\operatorname*{Perf}\mathcal{X})}
&& {\underset{v\in S \setminus S_{\infty}}{\bigoplus}\mathcal{U}%
^{\operatorname*{loc}}(\operatorname*{Perf}{\mathcal{X}}_{v})} && {\mathcal
{U}^{\operatorname*{loc}}({\underline{{\mathsf{LC}}}_{\mathcal{X}}}),}
	\arrow["\alpha", from=1-1, to=1-3]
	\arrow[from=1-1, to=2-1]
	\arrow["\beta", from=1-3, to=1-5]
	\arrow[from=1-3, to=2-3]
	\arrow[equals, from=1-5, to=2-5]
	\arrow["{{\alpha^{\prime}}}"', from=2-1, to=2-3]
	\arrow["{{\beta^{\prime}}}"'{pos=0.7}, from=2-3, to=2-5]
\end{tikzcd}
}
}
%%%%%%%%%%%%%%%%%
\label{DiagW1}%
\end{equation}
where the downward arrows use (1) the equivalences%
\begin{align*}
&  \operatorname*{Perf}\mathcal{X}\otimes_{\mathcal{O}_{S}}%
\operatorname*{Perf}\mathcal{O}_{S}\overset{\sim}{\longrightarrow
}\operatorname*{Perf}\mathcal{X}\\
&  \operatorname*{Perf}\mathcal{X}\otimes_{\mathcal{O}_{S}}%
\operatorname*{Perf}F_{v}\overset{\sim}{\longrightarrow}\operatorname*{Perf}%
\mathcal{X}_{v}%
\end{align*}
in $\operatorname*{Cat}\nolimits_{\infty}^{\operatorname*{perf}}%
(\mathcal{O}_{S})$ from Eq. \ref{l_d2} (note that the flatness assumption is
met), followed by (2) the property of $\mathcal{U}^{\operatorname*{loc}}$ to
be a symmetric monoidal functor, \S \ref{subsubsect_Generalities}. The map is
induced from the exact functor $i_{v}^{\ast}$ of Diag. \ref{l_d3}. The map
$\beta^{\prime}$ is a little more complicated:\ We have no better description
at this point than saying that it is induced from $\operatorname*{id}%
_{\mathcal{X}}\otimes\beta$ in $\left.  \mathbf{Mot}(\mathcal{O}_{S})\right.
$ with $\beta$ from Eq. \ref{l_d1}. (Hypothesis: $\pi$ is proper)\ We now
prove Cor. \ref{cor_v1}. Diagram \ref{l_d0} arises from considering two copies
of the above argument, once for $\mathcal{X}$ itself, and then for
$\mathcal{X}$ replaced by $\operatorname*{Spec}\mathcal{O}_{S}$, and the
downward arrows stem from the map $\pi_{\ast}$ of Eq. \ref{l_h0}.
\end{proof}

Suppose $F^{\prime}/F$ is a finite field extension. Let $S^{\prime}$ denote
the set of all places of $F^{\prime}$ which lie above places in $F$. Then
$S^{\prime}$ contains all infinite places of $F^{\prime}$. Differing from the
notation elsewhere in this text, we temporarily write $\mathcal{O}_{F,S}$
(resp. $\mathcal{O}_{F^{\prime},S^{\prime}}$) for the rings of $S$-integers of
$F$ (resp. $S^{\prime}$-integers of $F^{\prime}$).

There is a base change functor, which is most conveniently defined indirectly
by%
\begin{align}
\mathcal{O}_{F^{\prime},S^{\prime}}\otimes-\colon{\mathsf{LCA}}_{\mathcal{O}%
_{F,S}}  &  \longrightarrow{\mathsf{LCA}}_{\mathcal{O}_{F^{\prime},S^{\prime}%
}}\label{l_d4}\\
X  &  \longmapsto\operatorname*{Hom}\nolimits_{\mathcal{O}_{F}}(\mathcal{O}%
_{F^{\prime}},X^{\vee})^{\vee}\text{,}\nonumber
\end{align}
where the inner $(-)^{\vee}$ is the Pontryagin dual in ${\mathsf{LCA}%
}_{\mathcal{O}_{F,S}}$, then $\operatorname*{Hom}\nolimits_{\mathcal{O}_{F}%
}(\mathcal{O}_{F^{\prime}},X^{\vee})$ carries a natural $\mathcal{O}%
_{F^{\prime},S^{\prime}}$-module structure and is locally compact\footnote{to
get local compactness, see Moskowitz \cite[Theorem 4.3 (2$^{^{\prime}}$%
)]{MR0215016}, and the paragraph before: Note that $\mathcal{O}_{F^{\prime}}$
is finitely generated as an abelian group.}, and the outer $(-)^{\vee}$ is the
Pontryagin dual of ${\mathsf{LCA}}_{\mathcal{O}_{F^{\prime},S^{\prime}}}$. The
functor is exact since Pontryagin duals are exact\footnote{The circle
$\mathbb{T}$ is an injective object in ${\mathsf{LCA}}_{\mathbf{Z}}$. One way
to see this is that by Example \ref{ex_PropertiesInLCAOS} (2), $\mathbf{Z}$ is
projective in ${\mathsf{LCA}}_{\mathbf{Z}}$, but $\mathbb{T}=\mathbf{Z}^{\vee
}$, so we conclude using (4).} and $\mathcal{O}_{F^{\prime}}$ (with the
discrete topology) is a projective object in ${\mathsf{LCA}}_{\mathcal{O}%
_{F,S}}$, so the $\operatorname*{Hom}\nolimits_{\mathcal{O}_{F}}%
(\mathcal{O}_{F^{\prime}},-)$ is exact\footnote{Finitely presented flat
modules are always projective, and by Example \ref{ex_PropertiesInLCAOS} (2)
discrete projective modules are also projective objects when considered in
${\mathsf{LCA}}_{\mathcal{O}_{F,S}}$.}.

\begin{remark}
The definition of base change in Eq. \ref{l_d4} circumvents having to work
with the partially defined monoidal structure from
\S \ref{subsubsect_SymmMonoidalLCAOS}. However, if we instead define Eq.
\ref{l_d4} by literally using the tensor product (this necessitates showing
that in the case at hand its topology is indeed locally compact), the
Hom-tensor adjunction shows agreement of definitions:
\begin{align*}
\mathcal{O}_{F^{\prime},S^{\prime}}\otimes_{\mathcal{O}_{F,S}}X  &
\cong\mathcal{O}_{F^{\prime}}\otimes_{\mathcal{O}_{F}}X\cong
\operatorname*{Hom}\nolimits_{\mathbf{Z}}(\mathcal{O}_{F^{\prime}}%
\otimes_{\mathcal{O}_{F}}X,\mathbb{T})^{\vee}\\
&  \cong\operatorname*{Hom}\nolimits_{\mathcal{O}_{F}}(\mathcal{O}_{F^{\prime
}},\operatorname*{Hom}\nolimits_{\mathbf{Z}}(X,\mathbb{T}))^{\vee}%
\cong\operatorname*{Hom}\nolimits_{\mathcal{O}_{F}}(\mathcal{O}_{F^{\prime}%
},X^{\vee})^{\vee}\text{.}%
\end{align*}
In fact, this a posteriori settles the local compactness of the tensor product.
\end{remark}

We may uniquely extend the base change functor to a pairing and to
${\mathsf{LC}}_{\mathcal{O}_{F,(-)}}$: Write $\mathsf{P}_{f}(R)$ for the exact
category of finitely generated projective $R$-modules. There is a unique
bi-exact pairing of exact categories%
\begin{equation}
\mathsf{P}_{f}(\mathcal{O}_{F^{\prime},S^{\prime}})\times{\mathsf{LCA}%
}_{\mathcal{O}_{F,S}}\longrightarrow{\mathsf{LCA}}_{\mathcal{O}_{F^{\prime
},S^{\prime}}} \label{l_d6}%
\end{equation}
such that restricting to the fixed object $\mathcal{O}_{F^{\prime},S^{\prime}%
}$ in the first slot, the functor agrees with the one from Eq. \ref{l_d4}.
Said differently: There is a unique pairing such that $(\mathcal{O}%
_{F^{\prime},S^{\prime}},X)\mapsto\mathcal{O}_{F^{\prime},S^{\prime}}\otimes
X$ for all objects $X\in{\mathsf{LC}}_{\mathcal{O}_{F,S}}$ and $(\varphi
_{1},\varphi_{2})\mapsto\varphi_{1}\varphi_{2}$ for
endomorphisms.\footnote{Concretely, exactness in the first variable uniquely
extends it to finite rank free modules $\mathsf{F}_{f}(\mathcal{O}_{F^{\prime
},S^{\prime}})$. This constructs a bi-exact pairing $\mathsf{F}_{f}%
(\mathcal{O}_{F^{\prime},S^{\prime}})\times{\mathsf{LC}}_{\mathcal{O}_{F,S}%
}\rightarrow{\mathsf{LC}}_{\mathcal{O}_{F^{\prime},S^{\prime}}}$ for the split
exact structure on $\mathsf{F}_{f}(\mathcal{O}_{F^{\prime},S^{\prime}})$. Then
take the idempotent completion $(-)^{\operatorname*{ic}}$ of this pairing
(this is $2$-functorial, \cite[\S 6]{MR2606234}), and $\mathsf{F}%
_{f}(\mathcal{O}_{F^{\prime},S^{\prime}})^{\operatorname*{ic}}\cong
\mathsf{P}_{f}(\mathcal{O}_{F^{\prime},S^{\prime}})$, while ${\mathsf{LC}%
}_{\mathcal{O}_{F^{\prime},S^{\prime}}}$ is already idempotent complete as it
is quasi-abelian. We obtain Eq. \ref{l_d6}.}

If $X\in{\mathsf{LCA}}_{\mathcal{O}_{F,S}}$ is a real vector space, so is
$\mathcal{O}_{F^{\prime},S^{\prime}}\otimes X$. Thus, the pairing is
well-defined on ${\mathsf{LC}}_{\mathcal{O}_{F,(-)}}$.

Using the translation from Example \ref{ex_DbPfVersusPerfectComplexes}, this
pairing induces a pairing on the derived $\infty$-categories
$\operatorname*{Perf}(\mathcal{O}_{F^{\prime},S^{\prime}})\times{\mathsf{LC}%
}_{\mathcal{O}_{F,S}}$ and via Remark \ref{rmk_BiexactFun} it uniquely turns
${\mathsf{LC}}_{\mathcal{O}_{F,S}}$ into a $\operatorname*{Perf}%
(\mathcal{O}_{F^{\prime},S^{\prime}})$-module:
\begin{equation}
\operatorname*{Perf}(\mathcal{O}_{F^{\prime},S^{\prime}})\otimes
_{\mathcal{O}_{F,S}}{\mathsf{LC}}_{\mathcal{O}_{F,S}}\longrightarrow
{\mathsf{LC}}_{\mathcal{O}_{F^{\prime},S^{\prime}}}\text{.} \label{l_d5}%
\end{equation}
Defining $\mathcal{X}$ to be $\operatorname*{Spec}\mathcal{O}_{F^{\prime
},S^{\prime}}$, the natural map $\pi\colon\mathcal{X}:=\operatorname*{Spec}%
\mathcal{O}_{F^{\prime},S^{\prime}}\rightarrow\operatorname*{Spec}%
\mathcal{O}_{F,S}$ is a relative dimension zero example of our setup in
\S \ref{sect_Setup}, and Eq. \ref{l_d5} could alternatively be written as a
map%
\begin{equation}
\underline{{\mathsf{LC}}}_{\mathcal{O}_{F^{\prime},S^{\prime}}}\longrightarrow
{\mathsf{LC}}_{\mathcal{O}_{F^{\prime},S^{\prime}}}\text{,} \label{l_cd2}%
\end{equation}
comparing the two definitions (Def. \ref{def_C1} vs. Def. \ref{def_LCA_X}).

\begin{corollary}
\label{cor1}Suppose $S$ is finite. Then the morphism of Eq. \ref{l_cd2}
induces an equivalence of non-commutative motives $\mathcal{U}%
^{\operatorname*{loc}}\underline{{\mathsf{LC}}}_{\mathcal{O}_{F^{\prime
},S^{\prime}}}\overset{\sim}{\longrightarrow}\mathcal{U}^{\operatorname*{loc}%
}{\mathsf{LC}}_{\mathcal{O}_{F^{\prime},S^{\prime}}}$.
\end{corollary}

\begin{proof}
The constructions above the statement of the lemma functorially extend to a
morphism of exact sequences in $\left.  \mathbf{Mot}(\mathcal{O}_{S})\right.
$,%
\[%
\adjustbox{scale=0.85}{
\begin{tikzcd}
	{\mathcal{U}^{\operatorname*{loc}}(\operatorname*{Perf}\mathcal{O}%
_{F^{\prime},S^{\prime}})} & {\underset{v^{\prime} \in S^{\prime} \setminus
{S}^{\prime}_{\infty}}{\bigoplus}\mathcal{U}^{\operatorname*{loc}%
}(\operatorname*{Perf}{\mathcal{O}}_{F^{\prime},v^{\prime}})} & {\mathcal
{U}^{\operatorname*{loc}}({\underline{{\mathsf{LC}}}_{\mathcal{O}_{F^{\prime
},S^{\prime}}}})} \\
	{\mathcal{U}^{\operatorname*{loc}}(\operatorname*{Perf}\mathcal{O}%
_{F^{\prime},S^{\prime}})} & {\mathcal{U}^{\operatorname*{loc}}(\operatorname
*{Perf}\mathcal{O}_{F^{\prime},S^{\prime} })\otimes_{\mathcal{O}_{F,S}%
}\underset{v\in S \setminus S_{\infty}}{\bigoplus}\mathcal{U}^{\operatorname
*{loc}}(\operatorname*{Perf}{\mathcal{O}}_{F,v})} & {\mathcal{U}%
^{\operatorname*{loc}}({\underline{{\mathsf{LC}}}_{\mathcal{O}_{F^{\prime
},S^{\prime}}}})} \\
	{\mathcal{U}^{\operatorname*{loc}}(\operatorname*{Perf}\mathcal{O}%
_{F^{\prime},S^{\prime}})} & {\underset{v^{\prime} \in S^{\prime} \setminus
S^{\prime}_{\infty}}{\bigoplus}\mathcal{U}^{\operatorname*{loc}}%
(\operatorname*{Perf}{\mathcal{O}}_{F^{\prime},v^{\prime}})} & {\mathcal
{U}^{\operatorname*{loc}}({{{\mathsf{LC}}}_{\mathcal{O}_{F^{\prime},S^{\prime
}}}}),}
	\arrow[from=1-1, to=1-2]
	\arrow[from=1-2, to=1-3]
	\arrow[equals, from=2-1, to=1-1]
	\arrow[from=2-1, to=2-2]
	\arrow[equals, from=2-1, to=3-1]
	\arrow[from=2-2, to=1-2]
	\arrow[from=2-2, to=2-3]
	\arrow[from=2-2, to=3-2]
	\arrow[equals, from=2-3, to=1-3]
	\arrow[from=2-3, to=3-3]
	\arrow[from=3-1, to=3-2]
	\arrow[from=3-2, to=3-3]
\end{tikzcd}
}%
%%%%%%%%%%%%%%%%%
\]
where the two top rows are Diagram \ref{DiagW1} (of Prop. \ref{prop_W1}
applied to $\mathcal{X}:=\operatorname*{Spec}\mathcal{O}_{F^{\prime}%
,S^{\prime}}$) but flipped top-down, and the bottom row is from \cite[Theorem
4.30]{klca1}, applied to $F^{\prime}$. The right lower downward arrow is from
Eq. \ref{l_cd2}, and the middle downward arrow comes from evaluating the base
change $Y\longmapsto\mathcal{O}_{F^{\prime},S^{\prime}}\otimes_{\mathcal{O}%
_{F,S}}Y$, but just as ordinary algebraic modules without a topology. It is
easy to see that the diagram commutes. In the top row all upward arrows are
equivalences (that was part of the proof of Prop. \ref{prop_W1}), and in the
bottom row the left and middle downward arrow are equivalences. Hence, so must
be the right downward arrow.
\end{proof}

Specializing from the non-commutative motive to $K$-theory, Prop.
\ref{prop_W1} immediately implies the following.

\begin{corollary}
\label{cor2}Suppose $S$ is finite. There is an equivalence of spectra%
\[
\operatorname*{cofib}\left(  K(\operatorname*{Perf}\mathcal{X})\underset
{\alpha^{\prime}}{\longrightarrow}\bigoplus\limits_{v\in S\setminus S_{\infty
}}K(\operatorname*{Perf}\mathcal{X}_{v})\right)  \underset{\beta^{\prime}%
}{\longrightarrow}K(\underline{{\mathsf{LC}}}_{\mathcal{X}})\text{,}%
\]
where $\alpha^{\prime}$ is induced by the exact functors $i_{v}^{\ast}$ (from
Diagram \ref{l_d3}).
\end{corollary}

\subsection{Arguments on the descent side}

Choose a finite type separated model%
\[
\overline{\pi}\colon\overline{\mathcal{X}}\longrightarrow\operatorname*{Spec}%
\mathcal{O}_{F}%
\]
such that there exists an isomorphism of $\mathcal{O}_{S}$-schemes
$\overline{\mathcal{X}}\times_{\mathcal{O}_{F}}\mathcal{O}_{S}\overset{\sim
}{\longrightarrow}\mathcal{X}$. We consider the closed-open complement
decomposition $\operatorname*{Spec}\mathcal{O}_{F}\hookleftarrow
\operatorname*{Spec}\mathcal{O}_{S}$, along with its base change along $\pi$,
i.e.,%
\begin{equation}%
%%%%%%%%%%
%\begin{tikzcd}
%	{\overline{\mathcal{X}}_Z} & {\overline{\mathcal{X}}} & {\mathcal{X}} \\
%	Z & {\operatorname{Spec}(\mathcal{O}_F)} & {\operatorname{Spec}(\mathcal
%{O}_S),}
%	\arrow["{i'}", hook, from=1-1, to=1-2]
%	\arrow[from=1-1, to=2-1]
%	\arrow["{\overline{\pi}}", from=1-2, to=2-2]
%	\arrow["{\tilde{j}}"', hook', from=1-3, to=1-2]
%	\arrow["\pi", from=1-3, to=2-3]
%	\arrow["i"', hook, from=2-1, to=2-2]
%	\arrow["j", hook', from=2-3, to=2-2]
%\end{tikzcd}
%}} }%
%%%%%%%%%%%%%%%%%
{
\begin{tikzcd}
	{\overline{\mathcal{X}}_Z} & {\overline{\mathcal{X}}} & {\mathcal{X}} \\
	Z & {\operatorname{Spec}(\mathcal{O}_F)} & {\operatorname{Spec}(\mathcal
{O}_S),}
	\arrow["{i'}", hook, from=1-1, to=1-2]
	\arrow[from=1-1, to=2-1]
	\arrow["{\overline{\pi}}", from=1-2, to=2-2]
	\arrow["{\tilde{j}}"', hook', from=1-3, to=1-2]
	\arrow["\pi", from=1-3, to=2-3]
	\arrow["i"', hook, from=2-1, to=2-2]
	\arrow["j", hook', from=2-3, to=2-2]
\end{tikzcd}
}
%%%%%%%%%%%%%%%%%
\label{lstructdiag1}%
\end{equation}
where $\mathcal{O}_{S}$ is the ring of $S$-integers from \S \ref{sect_Setup},
and $Z$ is the reduced closed complement. This means that $Z=\{v_{1}%
,\ldots,v_{r}\}$ is a finite set of closed points.

\begin{example}
$\operatorname*{Spec}\mathcal{O}_{F}$ has the special property that every open
subscheme is affine and then of the form $\operatorname*{Spec}\mathcal{O}_{S}$
for $S$ being the set of closed points in the complement plus the infinite places.
\end{example}

Let $\mathsf{L}$ denote either the $\infty$-category of $p$-primary torsion
abelian groups or of $p$-primary torsion spectra. Write $\mathsf{Sh}%
_{\mathrm{\acute{e}t}}^{\operatorname*{hyp}}(X,\mathsf{L})$ for hypercomplete
\'{e}tale sheaves with values in $\mathsf{L}$ for the small \'{e}tale site. We
also use the $6$-functor formalism attached to hypercomplete sheaves.

\begin{remark}
[{\cite[Thm. 7.13]{sixfunct}}]\label{rmk_hypersheaves}We elaborate a
little:\ Given a morphism $f\colon X\rightarrow Y$, one can define $f_{\ast}$
just as for presheaves (it sends sheaves to sheaves, and hypersheaves to
hypersheaves). Then one defines $f^{\ast}$ as the left adjoint. One may do
this in the $\infty$-category of \'{e}tale sheaves or in the setting of
\'{e}tale hypersheaves, giving a priori two distinct formalisms $-$ not just
for $(f^{\ast},f_{\ast})$, also for building the rest of the $6$-functor
formalism. We always use the hypersheaf version. However, by Clausen--Mathew
for all schemes of relevance in this text (and more generally for: qcqs
schemes $X$ of finite Krull dimension and such that there exists some $D$ such
that the virtual $p$-cohomological dimension $\operatorname*{vcd}_{p}%
(\kappa(x))\leq D$ for all $x\in X$) the inclusion $\mathsf{Sh}%
_{\mathrm{\acute{e}t}}^{\operatorname*{hyp}}(X,\mathsf{L})\hookrightarrow
\mathsf{Sh}_{\mathrm{\acute{e}t}}(X,\mathsf{L})$ of hypercomplete sheaves into
all sheaves is an equivalence \cite{MR4296353}.
\end{remark}

\begin{remark}
[{\cite[Prop. 7.14]{sixfunct}}]\label{rmk_hypersheaves2}Suppose $X=\lim
_{r}X_{r}$ is a cofiltered inverse limit along affine transition maps in the
category of qcqs schemes. Then the functor $X\mapsto\mathsf{Sh}%
_{\mathrm{\acute{e}t}}(X,\mathsf{L})$ induces an equivalence%
\[
\mathsf{Sh}_{\mathrm{\acute{e}t}}(X,\mathsf{L})\longrightarrow\lim
\limits_{r}\mathsf{Sh}_{\mathrm{\acute{e}t}}(X_{r},\mathsf{L})\text{.}%
\]
In particular, for any closed point $x\in X$, $\mathsf{Sh}_{\mathrm{\acute
{e}t}}(\operatorname*{Spec}\mathcal{O}_{X,x}^{\operatorname*{h}},\mathsf{L})$
can equivalently be expressed as the limit over \'{e}tale neighbourhoods of
$x$. By Clausen--Mathew \cite{MR4296353} as described in Rmk.
\ref{rmk_hypersheaves}, this fact is also available for hypersheaves in the
setting of this text.
\end{remark}

\begin{remark}
\label{rmk_CompactlySupportedEtaleCohomology}If $\mathcal{F}\in\mathsf{Sh}%
_{\mathrm{\acute{e}t}}^{\operatorname*{hyp}}(\mathcal{X},\mathsf{L})$ with
$\pi\colon\mathcal{X}\longrightarrow\mathcal{S}$, we write%
\[
H_{c}^{s}(\mathcal{X},\mathcal{F}):=\pi_{-s}\operatorname*{RHom}%
\nolimits_{\mathcal{O}_{F}}(1_{\mathcal{O}_{F}},j_{!}\pi_{!}\mathcal{F})
\]
with $j\colon\operatorname*{Spec}\mathcal{O}_{S}\hookrightarrow
\operatorname*{Spec}\mathcal{O}_{F}$ for \emph{compactly supported \'{e}tale
cohomology}. There is a subtlety: Usually compactly supported cohomology is
defined a little different from merely using $j_{!}$ for rings of integers in
order to incorporate the infinite places of $F$. We will \emph{not} do this.
These about to be ignored contributions at infinite places are Tate cohomology
groups $H_{T}^{\bullet}(F_{v},\mathcal{F})$ for $\mathcal{F}$ an \'{e}tale
$p$-torsion sheaf. In our setting, for $F_{v}=\mathbf{C}$ these groups are
zero, so it is fine to drop them, and for $F_{v}=\mathbf{R}$ the groups are
$2$-torsion, but also $p$-torsion and since we assume $p$ to be odd, they also
vanish.\ Thus, we still get the same compactly supported cohomology as for
example in \cite{milne2006,MR1045856,MR1327282,MR3867292}.
\end{remark}

We need some version of the mechanism to invert Bott elements. Classically,
choose an element $\tau\in$ $\pi_{2\ell}K/p^{r}(\mathbf{Z}\left[  \frac{1}%
{p}\right]  )$ that becomes a power of a classical Bott element after
adjoining a primitive $p$-th root $\zeta_{p}$ \cite[\S A.7]{MR826102},%
\[
\tau=\beta^{\ell}\in\pi_{2}K/p\left(  \mathbf{Z}\left[  \zeta_{p},\frac{1}%
{p}\right]  \right)  \mapsto\zeta_{p}\in\pi_{1}K/p\left(  \mathbf{Z}\left[
\zeta_{p},\frac{1}{p}\right]  \right)  \text{,}%
\]
or, following the ideas of \cite{MR1740880}, \cite[\S 6]{MR4444265} pick a
$\mathsf{\operatorname*{MGL}}$-theoretic Bott element which induces a
Bott-inverted $\mathsf{\operatorname*{KGL}}/p^{r}[\tau^{-1}]$ which, when
evaluating it on a $\mathbf{Z}\left[  \frac{1}{p}\right]  $-scheme $Y$, agrees
with the localization of $K/p^{r}(Y)$ at the former sense of an element $\tau
$. Yet another recent viewpoint is \cite{etalemotivicspectravoevodskys}. Any
choice of $\tau$ in either mechanism of constructing the localization, yields
the same $K/p^{r}(-)[\tau^{-1}]$. For any $\mathcal{O}_{S}$-scheme $Y$ we may
now speak of \textquotedblleft inverting \textit{the} Bott
element\textquotedblright\ as $K/p^{r}(Y)$ is a module over $K/p^{r}%
(\mathcal{O}_{S})$, and in particular over $K/p^{r}(\mathbf{Z}[\frac{1}{p}])$.
Analogously for a sheaf:

\begin{definition}
\bigskip Pick $r\geq1$. Let%
\begin{equation}
{\mathcal{K}}:=K/p^{r}[\tau^{-1}] \label{l_h1}%
\end{equation}
be the sheaf of mod $p^{r}$ $K$-theory on $\operatorname*{Spec}\mathcal{O}%
_{S}$ (i.e., $K(Y):=K(\operatorname*{Perf}Y)$ for all qcqs $\mathcal{O}_{S}%
$-schemes $Y$) and $(-)[\tau^{-1}]$ refers to inverting a Bott element.
\end{definition}

Then $\mathcal{K}$ has \'{e}tale hyperdescent, for example by \cite[Theorem
1.3]{MR4296353} or \cite[Theorem 1.2]{MR4444265}.

\begin{proposition}
\label{prop_Equiv1}Suppose $S$ is finite. There is a natural equivalence of
spectra%
\[
\Sigma(j_{!}\pi_{\ast}{{\pi}^{\ast}\mathcal{K}})(\operatorname*{Spec}%
\mathcal{O}_{F})\overset{\sim}{\longrightarrow}(K/p^{r})(\underline
{{\mathsf{LC}}}_{\mathcal{X}})[\tau^{-1}]
\]
identifying

\begin{enumerate}
\item the $\Sigma$-shift of the global sections of $j_{!}\pi_{\ast}{{\pi
}^{\ast}\mathcal{K}}$ on $\operatorname*{Spec}\mathcal{O}_{F}$, with

\item the mod $p^{r}$ $K$-theory of $\underline{{\mathsf{LC}}}_{\mathcal{X}}$,
also with the Bott element inverted.
\end{enumerate}
\end{proposition}

This proposition is our variant of \cite[Theorem 2.5]{MR4121155}. We note that
we may invert the Bott element for $K(\underline{{\mathsf{LC}}}_{\mathcal{X}%
})$ since it is also a $K(\mathcal{O}_{S})$-module, Eq. \ref{l_d5}.

\begin{corollary}
\label{cor_v2}Suppose we are in the situation of the previous proposition. If
$\pi$ is proper, we additionally get the commutative square%
\[%
%%%%%%%%%%
%\begin{tikzcd}
%	{\Sigma(j_{!}\pi_{\ast}{{\pi}^{\ast}\mathcal{K}})(\operatorname{Spec}
%\mathcal{O}_{F})} && {(K/p^{r})(\underline{{\mathsf{LC}}}_{\mathcal{X}}%
%)[\tau^{-1}]} \\
%	\\
%	{\Sigma(j_{!}{\mathcal{K}})(\operatorname{Spec} \mathcal{O}_{F})}
%&& {(K/p^{r})(\underline{{\mathsf{LC}}}_{\mathcal{O}_S})[\tau^{-1}],}
%	\arrow["\sim", from=1-1, to=1-3]
%	\arrow["{{{\pi_!}{\pi^{!}}\rightarrow1}}"', from=1-1, to=3-1]
%	\arrow["{{{\pi}_{*}}}", from=1-3, to=3-3]
%	\arrow["\sim"', from=3-1, to=3-3]
%\end{tikzcd}
%}}}%
%%%%%%%%%%%%%%%%%
{
\begin{tikzcd}
	{\Sigma(j_{!}\pi_{\ast}{{\pi}^{\ast}\mathcal{K}})(\operatorname{Spec}
\mathcal{O}_{F})} && {(K/p^{r})(\underline{{\mathsf{LC}}}_{\mathcal{X}}%
)[\tau^{-1}]} \\
	\\
	{\Sigma(j_{!}{\mathcal{K}})(\operatorname{Spec} \mathcal{O}_{F})}
&& {(K/p^{r})(\underline{{\mathsf{LC}}}_{\mathcal{O}_S})[\tau^{-1}],}
	\arrow["\sim", from=1-1, to=1-3]
	\arrow["{{{\pi_!}{\pi^{!}}\rightarrow1}}"', from=1-1, to=3-1]
	\arrow["{{{\pi}_{*}}}", from=1-3, to=3-3]
	\arrow["\sim"', from=3-1, to=3-3]
\end{tikzcd}
}%
%%%%%%%%%%%%%%%%%
\]
where $\pi_{\ast}$ on the right is defined as in Eq. \ref{l_h0}-\ref{l_h0a}.
\end{corollary}

We shall prove both the proposition and the corollary in one go.

\begin{proof}
(Hypothesis: $\pi$ need not be proper) In this proof we shall use base change
theorems for the \'{e}tale site. These are available as we exclusively work
with torsion spectra, notably ${{\pi}^{\ast}\mathcal{K}}$, and various
transfers applied to it. Let $i,j$ be as in Diagram \ref{lstructdiag1}. For
any object $\mathcal{F}\in\mathsf{Sh}_{\mathrm{\acute{e}t}}%
^{\operatorname*{hyp}}(X,\mathsf{L})$ there is a standard recollement fiber
sequence%
\[
j_{!}j^{\ast}\mathcal{F}\longrightarrow\mathcal{F}\longrightarrow i_{\ast
}i^{\ast}\mathcal{F}%
\]
and by adjunction%
\begin{equation}
i_{\ast}i^{!}\mathcal{F}\longrightarrow\mathcal{F}\longrightarrow j_{\ast
}j^{\ast}\mathcal{F}\text{.} \label{lch0}%
\end{equation}
Applied to $\mathcal{F}:=j_{!}\pi_{\ast}{{\pi}^{\ast}\mathcal{K}}$ and using
$j^{\ast}j_{!}=\operatorname*{id}$ as well as the adjunction $(j^{\ast
},j_{\ast})$ in the rightmost term, we obtain a fiber sequence of mapping
spectra%
\begin{equation}
\operatorname*{Hom}\nolimits_{\mathcal{O}_{F}}(1_{\mathcal{O}_{F}},i_{\ast
}i^{!}(j_{!}\pi_{\ast}{{\pi}^{\ast}\mathcal{K}}))\longrightarrow
\operatorname*{Hom}\nolimits_{\mathcal{O}_{F}}(1_{\mathcal{O}_{F}},j_{!}%
\pi_{\ast}{{\pi}^{\ast}\mathcal{K}})\longrightarrow\operatorname*{Hom}%
\nolimits_{\mathcal{O}_{S}}(j^{\ast}1_{\mathcal{O}_{F}},\pi_{\ast}{{\pi}%
^{\ast}\mathcal{K}}) \label{lch1}%
\end{equation}
in $\mathsf{Sp}$. Now, consider the commutative cube of scheme morphisms%
\begin{equation}%
%%%%%%%%%%
%{
%\begin{tikzcd}
%	{\overline{\mathcal{X} }} &&& {{\mathcal{X} }} \\
%	& {\underset{v}{\coprod} \overline{\mathcal{X} } \times_{\mathcal{O}_F}
%	{\operatorname{Spec}(\mathcal{O}_{F,v}^{\operatorname{h} })}} &&& {\underset
%{v}{\coprod}{\mathcal{X} } \times_{\mathcal{O}_F} {\operatorname{Spec}%
%(F_{v}^{\operatorname{h} })}} \\
%	\\
%	{\operatorname{Spec}(\mathcal{O}_F)} &&& {\operatorname{Spec}(\mathcal{O}_S)}
%\\
%	& {\underset{v}{\coprod}\operatorname{Spec}(\mathcal{O}_{F,v}^{\operatorname
%{h} })} &&& {\underset{v}{\coprod}\operatorname{Spec}(F_{v}^{\operatorname{h}
%	})}
%	\arrow[from=1-1, to=4-1]
%	\arrow[from=1-4, to=1-1]
%	\arrow["\pi"{pos=0.6}, from=1-4, to=4-4]
%	\arrow[from=2-2, to=1-1]
%	\arrow[from=2-2, to=5-2]
%	\arrow["{{w'}}"', from=2-5, to=1-4]
%	\arrow[from=2-5, to=2-2]
%	\arrow["{{\pi'}}"{pos=0.6}, from=2-5, to=5-5]
%	\arrow["j"{pos=0.3}, from=4-4, to=4-1]
%	\arrow["w", from=5-2, to=4-1]
%	\arrow["{{\tilde{w}}}"', from=5-5, to=4-4]
%	\arrow["{{j'}}"{pos=0.4}, from=5-5, to=5-2]
%\end{tikzcd}
%}
%}} }%
%%%%%%%%%%%%%%%%%
{\adjustbox{max width=\textwidth}
{
\begin{tikzcd}
	{\overline{\mathcal{X} }} &&& {{\mathcal{X} }} \\
	& {\underset{v}{\coprod} \overline{\mathcal{X} } \times_{\mathcal{O}_F}
	{\operatorname{Spec}(\mathcal{O}_{F,v}^{\operatorname{h} })}} &&& {\underset
{v}{\coprod}{\mathcal{X} } \times_{\mathcal{O}_F} {\operatorname{Spec}%
(F_{v}^{\operatorname{h} })}} \\
	\\
	{\operatorname{Spec}(\mathcal{O}_F)} &&& {\operatorname{Spec}(\mathcal{O}_S)}
\\
	& {\underset{v}{\coprod}\operatorname{Spec}(\mathcal{O}_{F,v}^{\operatorname
{h} })} &&& {\underset{v}{\coprod}\operatorname{Spec}(F_{v}^{\operatorname{h}
	})}
	\arrow[from=1-1, to=4-1]
	\arrow[from=1-4, to=1-1]
	\arrow["\pi"{pos=0.6}, from=1-4, to=4-4]
	\arrow[from=2-2, to=1-1]
	\arrow[from=2-2, to=5-2]
	\arrow["{{w'}}"', from=2-5, to=1-4]
	\arrow[from=2-5, to=2-2]
	\arrow["{{\pi'}}"{pos=0.6}, from=2-5, to=5-5]
	\arrow["j"{pos=0.3}, from=4-4, to=4-1]
	\arrow["w", from=5-2, to=4-1]
	\arrow["{{\tilde{w}}}"', from=5-5, to=4-4]
	\arrow["{{j'}}"{pos=0.4}, from=5-5, to=5-2]
\end{tikzcd}
}
}
%%%%%%%%%%%%%%%%%
\label{vchdiag1}%
\end{equation}

which arises as follows:\ (1) the back face is the right square from Diagram
\ref{lstructdiag1}, and (2) the top face stems from base change by the
Henselization $(-)^{\operatorname*{h}}$ of $\mathcal{O}_{F}$ at all $v\in
S\setminus S_{\infty}$. In our notation, this spells out as the fiber product
$-\times_{\operatorname*{Spec}\mathcal{O}_{F}}Z$ (with $Z$ as we had
introduced it in Diagram \ref{lstructdiag1}), and $F_{v}^{\operatorname*{h}%
}:=\operatorname*{Frac}\mathcal{O}_{v}^{\operatorname*{h}}$. Rewriting the
meaning of each term in Eq. \ref{lch1} in more classical terms as cohomology
with support in $Z$, and global sections on $\operatorname*{Spec}%
\mathcal{O}_{F}$ (resp. the open $\operatorname*{Spec}\mathcal{O}_{S}$), this
yields the usual localization long exact sequence attached to the lower row in
Diagram \ref{lstructdiag1} for $j_{!}\pi_{\ast}{{\pi}^{\ast}\mathcal{K}}$.
Rotating it once to the left, this becomes the fiber sequence going downwards
the left column of the following diagram, with boundary map $\partial$,%
\begin{equation}%
%%%%%%%%%%
%\begin{tikzcd}
%	{R\Gamma(\operatorname*{Spec}\mathcal{O}_{F},j_{!}\pi_{\ast}{{\pi}%
%^{*}\mathcal{K}})}
%& {\bigoplus_{v\in S\setminus S_{\infty}}R\Gamma(\operatorname*{Spec}
%\mathcal{O}_{F}^{\operatorname*{h}},w^{\ast}j_{!}\pi_{\ast}{{\pi}^{*}%
%\mathcal{K}})} \\
%	{R\Gamma(\operatorname*{Spec}\mathcal{O}_{S},\pi_{\ast}{{\pi}^{*}\mathcal{K}%
%})}
%& {\bigoplus_{v\in S\setminus S_{\infty}}R\Gamma(\operatorname*{Spec}%
%F_{v}^{\operatorname*{h}} ,\tilde{w}^{\ast}\pi_{\ast}{{\pi}^{*}\mathcal{K}})}
%\\
%	{\Sigma R\Gamma_{Z}(\operatorname*{Spec}\mathcal{O}_{F},j_{!}\pi_{\ast}{{\pi
%}^{*}\mathcal{K}})} & {\bigoplus_{v\in S\setminus S_{\infty}}\Sigma
%R\Gamma_{\{v\}}(\operatorname*{Spec}\mathcal{O}_{F}^{\operatorname*{h}
%},w^{\ast}j_{!}\pi_{\ast}{{\pi}^{*}\mathcal{K}}),}
%	\arrow["w", from=1-1, to=1-2]
%	\arrow[from=1-1, to=2-1]
%	\arrow[from=1-2, to=2-2]
%	\arrow["{\tilde{w}}", from=2-1, to=2-2]
%	\arrow["\partial"', from=2-1, to=3-1]
%	\arrow["{\partial^{\operatorname{h} }}", from=2-2, to=3-2]
%	\arrow["w"', from=3-1, to=3-2]
%\end{tikzcd}
%}} }%
%%%%%%%%%%%%%%%%%
{
\begin{tikzcd}
	{R\Gamma(\operatorname*{Spec}\mathcal{O}_{F},j_{!}\pi_{\ast}{{\pi}%
^{*}\mathcal{K}})}
& {\bigoplus_{v\in S\setminus S_{\infty}}R\Gamma(\operatorname*{Spec}
\mathcal{O}_{F}^{\operatorname*{h}},w^{\ast}j_{!}\pi_{\ast}{{\pi}^{*}%
\mathcal{K}})} \\
	{R\Gamma(\operatorname*{Spec}\mathcal{O}_{S},\pi_{\ast}{{\pi}^{*}\mathcal{K}%
})}
& {\bigoplus_{v\in S\setminus S_{\infty}}R\Gamma(\operatorname*{Spec}%
F_{v}^{\operatorname*{h}} ,\tilde{w}^{\ast}\pi_{\ast}{{\pi}^{*}\mathcal{K}})}
\\
	{\Sigma R\Gamma_{Z}(\operatorname*{Spec}\mathcal{O}_{F},j_{!}\pi_{\ast}{{\pi
}^{*}\mathcal{K}})} & {\bigoplus_{v\in S\setminus S_{\infty}}\Sigma
R\Gamma_{\{v\}}(\operatorname*{Spec}\mathcal{O}_{F}^{\operatorname*{h}
},w^{\ast}j_{!}\pi_{\ast}{{\pi}^{*}\mathcal{K}}),}
	\arrow["w", from=1-1, to=1-2]
	\arrow[from=1-1, to=2-1]
	\arrow[from=1-2, to=2-2]
	\arrow["{\tilde{w}}", from=2-1, to=2-2]
	\arrow["\partial"', from=2-1, to=3-1]
	\arrow["{\partial^{\operatorname{h} }}", from=2-2, to=3-2]
	\arrow["w"', from=3-1, to=3-2]
\end{tikzcd}
}
%%%%%%%%%%%%%%%%%
\label{vchdiag2}%
\end{equation}
where we wrote $R\Gamma$ (which, since previously we always worked $\infty
$-categorically, could also reasonably just be denoted by $\Gamma$ resp.
$\Gamma_{Z}$). The localization sequence can also be formed along the open
immersion $j^{\prime}$ from the front face of the cube, leading to the right
column in Diagram \ref{vchdiag2}, with the horizontal arrows induced from $w$
(or rather the unit $\operatorname*{id}\rightarrow w_{\ast}w^{\ast}$). By
\'{e}tale excision \cite[Prop. 5.6.12]{MR3380806} (adapted to hypersheaves of
spectra, using Rmk. \ref{rmk_hypersheaves2}) the lower horizontal arrow is an
isomorphism:\ Concretely, on the level of its homotopy groups%
\begin{align*}
H_{Z}^{m}(\operatorname*{Spec}\mathcal{O}_{F},j_{!}\pi_{\ast}{{\pi}^{\ast
}\mathcal{K}})  &  \cong\bigoplus\limits_{v\in S\setminus S_{\infty}}%
H_{\{v\}}^{m}(\operatorname*{Spec}\mathcal{O}_{F},j_{!}\pi_{\ast}{{\pi}^{\ast
}\mathcal{K}})\\
&  \cong\bigoplus\limits_{v\in S\setminus S_{\infty}}H_{\{v\}}^{m}%
(\operatorname*{Spec}\mathcal{O}_{F,v}^{\operatorname*{h}},w^{\ast}j_{!}%
\pi_{\ast}{{\pi}^{\ast}\mathcal{K}})
\end{align*}
by first using that the closed complement of $\operatorname*{Spec}%
\mathcal{O}_{S}$ in $\operatorname*{Spec}\mathcal{O}_{F}$ decomposes into
isolated closed points $\{v\}$ for $v\in S\setminus S_{\infty}$. Next, by
\cite[Ch. II, \S 1, Prop. 1.1 (a)]{milne2006} and its proof the arrow
$\partial^{\operatorname*{h}}$ on the right side in Diagram \ref{vchdiag2} is
an isomorphism and\footnote{using basechange $w^{\ast}j_{!}\rightarrow
j_{!}^{\prime}\tilde{w}^{\ast}$ for the open immersion $j$}%
\[
R\Gamma(\operatorname*{Spec}\mathcal{O}_{F}^{\operatorname*{h}},w^{\ast}%
j_{!}\pi_{\ast}{{\pi}^{\ast}\mathcal{K}})\cong R\Gamma(\operatorname*{Spec}%
\mathcal{O}_{F}^{\operatorname*{h}},j_{!}^{\prime}(\tilde{w}^{\ast}\pi_{\ast
}{{\pi}^{\ast}\mathcal{K}}))=0\text{.}%
\]
This allows us to simplify and evaluate Diagram \ref{vchdiag2} to%
\begin{equation}%
%%%%%%%%%%
%\begin{tikzcd}
%	{(j_{!}\pi_{\ast}{{\pi}^{*}\mathcal{K}})({\operatorname*{Spec}\mathcal{O}%
%_{F}})} \\
%	{K/p^r({\operatorname{Perf}}\mathcal{X})[{\tau}^{-1}]} & {\bigoplus_{v\in
%S\setminus S_{\infty}}K/p^r({\operatorname{Perf}}(\mathcal{X}\times
%_{\operatorname{Spec}\mathcal{O}_{F }}\operatorname{Spec}F_{v}^{\operatorname
%{h} }))[{\tau}^{-1}]} \\
%	{\Sigma R\Gamma_{Z}(\operatorname*{Spec}\mathcal{O}_{F},j_{!}\pi_{\ast}{{\pi
%}^{*}\mathcal{K}})} & {\bigoplus_{v\in S\setminus S_{\infty}}\Sigma
%R\Gamma_{\{v\}}(\operatorname*{Spec}\mathcal{O}_{F}^{\operatorname*{h}
%},w^{\ast}j_{!}\pi_{\ast}{{\pi}^{*}\mathcal{K}})}
%	\arrow[from=1-1, to=2-1]
%	\arrow["{\tilde{w}}", from=2-1, to=2-2]
%	\arrow["\partial"', from=2-1, to=3-1]
%	\arrow["{\partial^{\operatorname{h} }}%
%", Rightarrow, no head, from=2-2, to=3-2]
%	\arrow["w"', Rightarrow, no head, from=3-1, to=3-2]
%\end{tikzcd}
%}} }%
%%%%%%%%%%%%%%%%%
{
\begin{tikzcd}
	{(j_{!}\pi_{\ast}{{\pi}^{*}\mathcal{K}})({\operatorname*{Spec}\mathcal{O}%
_{F}})} \\
	{K/p^r({\operatorname{Perf}}\mathcal{X})[{\tau}^{-1}]} & {\bigoplus_{v\in
S\setminus S_{\infty}}K/p^r({\operatorname{Perf}}(\mathcal{X}\times
_{\operatorname{Spec}\mathcal{O}_{F }}\operatorname{Spec}F_{v}^{\operatorname
{h} }))[{\tau}^{-1}]} \\
	{\Sigma R\Gamma_{Z}(\operatorname*{Spec}\mathcal{O}_{F},j_{!}\pi_{\ast}{{\pi
}^{*}\mathcal{K}})} & {\bigoplus_{v\in S\setminus S_{\infty}}\Sigma
R\Gamma_{\{v\}}(\operatorname*{Spec}\mathcal{O}_{F}^{\operatorname*{h}
},w^{\ast}j_{!}\pi_{\ast}{{\pi}^{*}\mathcal{K}})}
	\arrow[from=1-1, to=2-1]
	\arrow["{\tilde{w}}", from=2-1, to=2-2]
	\arrow["\partial"', from=2-1, to=3-1]
	\arrow["{\partial^{\operatorname{h} }}%
", Rightarrow, no head, from=2-2, to=3-2]
	\arrow["w"', Rightarrow, no head, from=3-1, to=3-2]
\end{tikzcd}
}
%%%%%%%%%%%%%%%%%
\label{vchdiag3}%
\end{equation}
by evaluation of global sections and by base change\footnote{$\pi$ is qc and
finite type, and via Rmk. \ref{rmk_hypersheaves2} the maps reduce to limits of
termwise \'{e}tale maps} on the right face of the cube in Diagram
\ref{vchdiag1},
\begin{align*}
R\Gamma(\operatorname*{Spec}F_{v}^{\operatorname*{h}},\tilde{w}^{\ast}%
\pi_{\ast}{{\pi}^{\ast}\mathcal{K}})  &  \cong R\Gamma(F_{v}%
^{\operatorname*{h}},\pi_{\ast}^{\prime}w^{\prime\ast}{{\pi}^{\ast}%
\mathcal{K}})\\
&  \cong R\Gamma(\mathcal{X\times}_{\operatorname*{Spec}\mathcal{O}_{F}%
}\operatorname*{Spec}F_{v}^{\operatorname*{h}},w^{\prime\ast}{{\pi}^{\ast
}\mathcal{K}})\\
&  \cong K/p^{r}(\operatorname*{Perf}(\mathcal{X\times}_{\operatorname*{Spec}%
\mathcal{O}_{F}}\operatorname*{Spec}F_{v}^{\operatorname*{h}}))[\tau
^{-1}]\text{.}%
\end{align*}
We observe that by invariance of \'{e}tale cohomology the\ Henselization can
be replaced by completion, so both source and target as well as the map
$\tilde{w}$ in the middle row of Diagram \ref{vchdiag3} can be replaced
isomorphically by%
\[
K/p^{r}(\operatorname*{Perf}\mathcal{X})[\tau^{-1}]\longrightarrow
\bigoplus\limits_{v\in S\setminus S_{\infty}}K/p^{r}(\operatorname*{Perf}%
(\mathcal{X\times}_{\operatorname*{Spec}\mathcal{O}_{S}}\operatorname*{Spec}%
F_{v}))[\tau^{-1}]\text{.}%
\]
As this agrees with the cofiber along $\alpha^{\prime}$ in Corollary
\ref{cor2}, the fiber sequence in the left column of Diagram \ref{vchdiag3}
can be identified with%
\begin{equation}
K/p^{r}(\operatorname*{Perf}\mathcal{X})[\tau^{-1}]\rightarrow\bigoplus
\limits_{v\in S\setminus S_{\infty}}K/p^{r}(\operatorname*{Perf}%
\mathcal{X}_{v})[\tau^{-1}]\text{.} \label{lch2}%
\end{equation}
Our claim follows.\newline(Hypothesis: $\pi$ proper) We now prove Cor.
\ref{cor_v2}. We construct the following diagram for the proper smooth
morphism $\pi\colon\mathcal{X}\longrightarrow\operatorname*{Spec}%
\mathcal{O}_{S}$: The adjunction $\pi_{!}\dashv\pi^{!}$ has the co-unit
transformation $\pi_{!}\pi^{!}\longrightarrow1$, which induces the lower left
vertical map%
\begin{equation}%
%%%%%%%%%%
%\begin{tikzcd}
%	{{\pi}^{*}\mathcal{K}{(\mathcal{X})} } && { K/{p^r}[\tau^{-1}]({\mathcal{X}%
%}) } \\
%	\\
%	{{\pi}_{*}\Sigma^{2d}{\pi}^{*}\mathcal{K}\otimes\mathbf{Z}/{{p^r}%
%}(d){(\mathcal{O}_S)}} \\
%	\\
%	{{\pi}_{!}{\pi}^{!}\mathcal{K}{(\mathcal{O}_S)}} \\
%	\\
%	{\mathcal{K}{(\mathcal{O}_S)}} && { K/{p^r}[\tau^{-1}]({\mathcal{O}_S})}
%	\arrow[equals, from=1-1, to=1-3]
%	\arrow[equals, from=1-1, to=3-1]
%	\arrow["{{{{{{\pi}_{*}[\tau^{-1}]}}}}}", from=1-3, to=7-3]
%	\arrow["{{{{{{\pi}_{!}={\pi}_{*}}}}}}", equals, from=3-1, to=5-1]
%	\arrow["{{{{{{{{\pi}_{!}{\pi}^{!}\rightarrow1}}}}}}}", from=5-1, to=7-1]
%	\arrow[equals, from=7-1, to=7-3]
%\end{tikzcd}
%}} }%
%%%%%%%%%%%%%%%%%
\adjustbox{max width=\textwidth}{
\begin{tikzcd}
	{{\pi}^{*}\mathcal{K}{(\mathcal{X})} } && { K/{p^r}[\tau^{-1}]({\mathcal{X}%
}) } \\
	\\
	{{\pi}_{*}\Sigma^{2d}{\pi}^{*}\mathcal{K}\otimes\mathbf{Z}/{{p^r}%
}(d){(\mathcal{O}_S)}} \\
	\\
	{{\pi}_{!}{\pi}^{!}\mathcal{K}{(\mathcal{O}_S)}} \\
	\\
	{\mathcal{K}{(\mathcal{O}_S)}} && { K/{p^r}[\tau^{-1}]({\mathcal{O}_S})}
	\arrow[equals, from=1-1, to=1-3]
	\arrow[equals, from=1-1, to=3-1]
	\arrow["{{{{{{\pi}_{*}[\tau^{-1}]}}}}}", from=1-3, to=7-3]
	\arrow["{{{{{{\pi}_{!}={\pi}_{*}}}}}}", equals, from=3-1, to=5-1]
	\arrow["{{{{{{{{\pi}_{!}{\pi}^{!}\rightarrow1}}}}}}}", from=5-1, to=7-1]
	\arrow[equals, from=7-1, to=7-3]
\end{tikzcd}
}
%%%%%%%%%%%%%%%%%
\label{lset1}%
\end{equation}
To avoid confusion: Recall that going modulo $p$ and inverting $\tau$ was part
of our definition of $\mathcal{K}$ in Eq. \ref{l_h1}, which is why these
operations do not show up in the left column. Since $\pi$ is assumed to be
smooth,%
\begin{equation}
\pi^{!}\mathcal{K}\cong\Sigma^{2d}\pi^{\ast}\mathcal{K}\otimes\mathbf{Z}%
/p^{r}(d)\cong\pi^{\ast}\mathcal{K} \label{lset2}%
\end{equation}
by Bott periodicity (this is the \'{e}tale sheaf incarnation of%
\[
\pi^{!}\mathrm{KGL}_{\mathcal{O}_{S}}\cong\Sigma^{2d,d}\pi^{\ast}%
\mathrm{KGL}_{\mathcal{O}_{S}}\cong\Sigma^{2d,d}\mathrm{KGL}_{\mathcal{X}%
}\cong\mathrm{KGL}_{\mathcal{X}}%
\]
in the motivic category).\ Since $\pi$ is proper, $\pi_{!}=\pi_{\ast}$. The
right downward arrow agrees with the arrow induced by the pushforward of
perfect complexes%
\[
\operatorname*{Perf}\mathcal{X}\longrightarrow\operatorname*{Perf}%
\mathcal{O}_{S}\text{,}%
\]
see \cite[Lemma 3.3.4]{jinpaper}, as we had used it in Eq. \ref{l_h0}. The
same argument applies to the (also proper smooth) morphisms $\pi
\colon\mathcal{X}_{v}\longrightarrow\operatorname*{Spec}F_{v}$. This being
settled, we may consider the co-unit transformation $\pi_{!}\pi^{!}%
\longrightarrow1$ in all of the above proof, notably Eq. \ref{vchdiag1}. As we
had shown that the left column in Diagram \ref{vchdiag3} can be identified
with Eq. \ref{lch2}, we obtain the solid arrows of%
\begin{equation}%
\adjustbox{max width=\textwidth}{
\begin{tikzcd}
	{K/p^r(\underline{{\mathsf{LC}}}_{\mathcal{X}})[\tau^{-1}]}
&& {K/p^r(\underline{{\mathsf{LC}}}_{\mathcal{O}_S})[\tau^{-1}]} \\
	& {R\Gamma(\operatorname*{Spec}\mathcal{O}_{F},j_{!}\pi_{\ast}{{\pi}%
^{*}\mathcal{K}})} && {R\Gamma(\operatorname*{Spec}\mathcal{O}_{F}%
,j_{!}{\mathcal{K}})} \\
	{K/p^r(\operatorname{Perf}\mathcal{X})[\tau^{-1}]} && {K/p^r(\operatorname
{Perf}\mathcal{O}_S)[\tau^{-1}]} \\
	& {R\Gamma(\operatorname*{Spec}\mathcal{O}_{S},\pi_{\ast}{{\pi}^{*}%
\mathcal{K}})} && {R\Gamma(\operatorname*{Spec}\mathcal{O}_{S},{\mathcal{K}})}
\\
	{K/p^r(\operatorname{Perf}\mathcal{X}_{v})[\tau^{-1}]}
&& {K/p^r(\operatorname{Perf}F_{v})[\tau^{-1}]} \\
	& {\Sigma R\Gamma_{Z}(\operatorname*{Spec}\mathcal{O}_{F},j_{!}\pi_{\ast
}{{\pi}^{*}\mathcal{K}})} && {\Sigma R\Gamma_{Z}(\operatorname*{Spec}%
\mathcal{O}_{F},j_{!}{\mathcal{K}})}
	\arrow["{{{\pi_{*}}}}", dashed, from=1-1, to=1-3]
	\arrow[from=1-1, to=3-1]
	\arrow[from=1-3, to=3-3]
	\arrow[equals, from=2-2, to=1-1]
	\arrow["{{{{\pi}_{!}{\pi}^{!}\rightarrow1}}}"'{pos=0.8}, from=2-2, to=2-4]
	\arrow[from=2-2, to=4-2]
	\arrow[equals, from=2-4, to=1-3]
	\arrow[from=2-4, to=4-4]
	\arrow["{{{\pi_{*}}}}"{pos=0.2}, from=3-1, to=3-3]
	\arrow[from=3-1, to=5-1]
	\arrow[from=3-3, to=5-3]
	\arrow[equals, from=4-2, to=3-1]
	\arrow["{{{{\pi}_{!}{\pi}^{!}\rightarrow1}}}"'{pos=0.8}, from=4-2, to=4-4]
	\arrow["\partial"'{pos=0.3}, from=4-2, to=6-2]
	\arrow[equals, from=4-4, to=3-3]
	\arrow["\partial"{pos=0.3}, from=4-4, to=6-4]
	\arrow["{{{\pi_{*}}}}"{pos=0.2}, from=5-1, to=5-3]
	\arrow[equals, from=6-2, to=5-1]
	\arrow["{{{{\pi}_{!}{\pi}^{!}\rightarrow1}}}"', from=6-2, to=6-4]
	\arrow[equals, from=6-4, to=5-3]
\end{tikzcd}
}
%%%%%%%%%%%%%%%%%
\label{vchdiag4}%
\end{equation}
where

\begin{itemize}
\item the front face is induced by $\pi_{!}\pi^{!}\longrightarrow1$,

\item the downward arrows are fiber sequences, coming from the above proof
(concretely from the localization recollement of Eq. \ref{lch0})

\item and the back face stems from the identifications which we have proven in
the first part of the proof.
\end{itemize}

The commutativity of the pushforwards%
\[
\pi_{\ast}\colon\operatorname*{Perf}\mathcal{X}\longrightarrow
\operatorname*{Perf}\mathcal{O}_{S}\qquad\text{(resp. }\pi_{\ast}%
\colon\operatorname*{Perf}\mathcal{X}_{v}\longrightarrow\operatorname*{Perf}%
F_{v}\text{)}%
\]
with the maps induced from $\pi_{!}\pi^{!}\longrightarrow1$ is the
commutativity of Diagram \ref{lset1}. Since the back face downward arrows are
also fiber sequences, the diagram induces a unique map (the dashed arrow).
However, we can identify the entire back face with the diagram in Cor.
\ref{cor_v1} (after applying $K$-theory, mod $p^{r}$ and inverting $\tau$).
This finishes the proof of Diagram \ref{vchdiag4}.
\end{proof}

The descent spectral sequence for the sheaf $\mathcal{K}$ has the $E_{2}$-page%
\begin{equation}%
%%%%%%%%%%
%\begin{tikzcd}
%	{E^{s,t}_{2}(\mathcal{X}):=H^{s}\left( \mathcal{O}_F , j_{!} {\pi}_{*} {\pi
%}^{*}
%\mathbf{Z}/{p^r}(-\tfrac{t}{2})\right)} & \Longrightarrow& {{\pi}_{-s-t}%
%\left( j_{!} {\pi}_{*} {\pi}^{*} \mathcal{K} \right)(\operatorname
%{Perf}\mathcal{O}_F)} \\
%	{H^{s}_{c}\left( \mathcal{X} , \mathbf{Z}/{p^r}(-\tfrac{t}{2})\right)}
%&& {{\pi}_{-s-t}\Sigma^{-1}K/{p^r}(\underline{{\mathsf{LC}}}_{\mathcal{X}%
%})[{\tau}^{-1}]}
%	\arrow[shift left=3, equals, from=1-1, to=2-1]
%	\arrow[equals, from=1-3, to=2-3]
%\end{tikzcd}
%}} }%
%%%%%%%%%%%%%%%%%
{
\begin{tikzcd}
	{E^{s,t}_{2}(\mathcal{X}):=H^{s}\left( \mathcal{O}_F , j_{!} {\pi}_{*} {\pi
}^{*}
\mathbf{Z}/{p^r}(-\tfrac{t}{2})\right)} & \Longrightarrow& {{\pi}_{-s-t}%
\left( j_{!} {\pi}_{*} {\pi}^{*} \mathcal{K} \right)(\operatorname
{Perf}\mathcal{O}_F)} \\
	{H^{s}_{c}\left( \mathcal{X} , \mathbf{Z}/{p^r}(-\tfrac{t}{2})\right)}
&& {{\pi}_{-s-t}\Sigma^{-1}K/{p^r}(\underline{{\mathsf{LC}}}_{\mathcal{X}%
})[{\tau}^{-1}]}
	\arrow[shift left=3, equals, from=1-1, to=2-1]
	\arrow[equals, from=1-3, to=2-3]
\end{tikzcd}
}
%%%%%%%%%%%%%%%%%
\label{lcih1}%
\end{equation}
with bidegree $(r,-r+1)$ differential\footnote{A spectral sequence `of
cohomological type', see \cite[Def. 5.2.3]{MR1269324}, starting on page $a=2$.
The conventions used\textit{ loc. cit.} are the same as we employ.}%
\[
d_{r}\colon E_{r}^{s,t}\longrightarrow E_{r}^{s+r,t-r+1}\text{,}%
\]
and $\mathbf{Z}/p^{r}(-\tfrac{t}{2})$ is to be read \textit{as a zero entry}
whenever $t$ is odd.\footnote{This is an old, by now standard, convention in
this area.} On the left, for the identifications expressed within Eq.
\ref{lcih1}, we used our convention for compactly supported \'{e}tale
cohomology (Remark \ref{rmk_CompactlySupportedEtaleCohomology}), and on the
right Prop. \ref{prop_Equiv1}.

\begin{remark}
The spectral sequence in Eq. \ref{lcih1} originates from Thomason
\cite[Theorem 4.1]{MR826102}, who however uses Bousfield--Kan indexing:
Replace our $j$ by $-j$ to obtain his notation (and differential $E_{r}%
^{i,j}\longrightarrow E_{r}^{i+r,j+r-1}$).
\end{remark}

An entirely analogous spectral sequence exists for $\operatorname*{Spec}%
\mathcal{O}_{S}$ itself, and they are connected through $\pi_{\ast}$ if $\pi$
is proper:

\begin{lemma}
\label{lemma_pi_proper1}Suppose $\pi$ is proper. Then the diagram%
\[%
%%%%%%%%%%
%\begin{tikzcd}
%	{E^{s,t}_{2}(\mathcal{X}):=H^{s+2d}_{c}\left( \mathcal{X} , \mathbf{Z}%
%/{p^r}(d-\tfrac{t}{2})\right)} & \Longrightarrow& {{\pi}_{-s-t}\Sigma
%^{-1}K/{p^r}(\underline{{\mathsf{LC}}}_{\mathcal{X}})[{\tau}^{-1}]} \\
%	\\
%	{E^{s,t}_{2}(\mathcal{O}_S):=H^{s}_{c}\left( \mathcal{O}_S , \mathbf{Z}%
%/{p^r}(-\tfrac{t}{2})\right)} & \Longrightarrow& {{\pi}_{-s-t}\Sigma
%^{-1}K/{p^r}(\underline{{\mathsf{LC}}}_{\mathcal{O}_S})[{\tau}^{-1}]}
%	\arrow["{{\pi}_{!}{\pi}^{!}\rightarrow1}"', from=1-1, to=3-1]
%	\arrow["{{\pi}_{*}}", from=1-3, to=3-3]
%\end{tikzcd}
%}}}%
%%%%%%%%%%%%%%%%%
{
\begin{tikzcd}
	{E^{s,t}_{2}(\mathcal{X}):=H^{s+2d}_{c}\left( \mathcal{X} , \mathbf{Z}%
/{p^r}(d-\tfrac{t}{2})\right)} & \Longrightarrow& {{\pi}_{-s-t}\Sigma
^{-1}K/{p^r}(\underline{{\mathsf{LC}}}_{\mathcal{X}})[{\tau}^{-1}]} \\
	\\
	{E^{s,t}_{2}(\mathcal{O}_S):=H^{s}_{c}\left( \mathcal{O}_S , \mathbf{Z}%
/{p^r}(-\tfrac{t}{2})\right)} & \Longrightarrow& {{\pi}_{-s-t}\Sigma
^{-1}K/{p^r}(\underline{{\mathsf{LC}}}_{\mathcal{O}_S})[{\tau}^{-1}]}
	\arrow["{{\pi}_{!}{\pi}^{!}\rightarrow1}"', from=1-1, to=3-1]
	\arrow["{{\pi}_{*}}", from=1-3, to=3-3]
\end{tikzcd}
}%
%%%%%%%%%%%%%%%%%
\]
commutes, where $\pi_{\ast}$ on the right is defined as in Eq. \ref{l_h0}%
-\ref{l_h0a}.
\end{lemma}

\begin{proof}
We must study what happens to the $K$-theory terms in the proof of Cor.
\ref{cor_v2} when we use the descent spectral sequence. The adjunction
$\pi_{!}\dashv\pi^{!}$ has the co-unit transformation $\pi_{!}\pi
^{!}\longrightarrow1$, which provides the vertical arrows in square\textsf{
(A)} below:%
\[%
%%%%%%%%%%
%\begin{tikzcd}
%	{E_{2}^{s,t}(\mathcal{X}):=H^{s}_{c}\left({\mathcal{X}},\mathbf{Z}/{{p^r}%
%}(d-\tfrac{t}{2}
%)[2d]\right)} & \Longrightarrow& {\pi_{-s-t}\left({j_{!}}{\pi}_{*}{\pi}%
%^{*}\mathcal{K}{(\mathcal{X})} \right)} \\
%	& {\mathsf{(C)}} \\
%	{E_{2}^{s,t}(\mathcal{X}):=H^{s}_{c}\left({\mathcal{X}},{\pi}^{!}\mathbf
%{Z}/{{p^r}}(-\tfrac{t}{2} )\right)} & \Longrightarrow& {\pi_{-s-t}\left
%({j_{!}}{\pi}_{*}\Sigma^{2d}{\pi}^{*}\mathcal{K}\otimes\mathbf{Z}/{{p^r}%
%}(d){(\mathcal{O}_S)} \right)} \\
%	& {\mathsf{(B)}} \\
%	{E_{2}^{s,t}(\mathcal{X}):=H^{s}\left({\mathcal{O}_F},{j_{!}}{\pi}_{!}{\pi
%}^{!}\mathbf{Z}/{{p^r}}(-\tfrac{t}{2} )\right)} & \Longrightarrow& {\pi
%_{-s-t}\left({j_{!}}{\pi}_{!}{\pi}^{!}\mathcal{K}{(\mathcal{O}_S)}\right)} \\
%	& {\mathsf{(A)}} \\
%	{E_{2}^{s,t}{(\mathcal{O}_S)}:=H^{s}\left({\mathcal{O}_F},{j_{!}}\mathbf
%{Z}/{{p^r}}(-\tfrac{t}{2} )\right)} & \Longrightarrow& {\pi_{-s-t}\left
%({j_{!}}\mathcal{K}{(\mathcal{O}_F)}\right)}
%	\arrow[equals, from=1-1, to=3-1]
%	\arrow[equals, from=1-3, to=3-3]
%	\arrow["{{{{{\pi}_{!}={\pi}_{*}}}}}", equals, from=3-1, to=5-1]
%	\arrow["{{{{{\pi}_{!}={\pi}_{*}}}}}", equals, from=3-3, to=5-3]
%	\arrow["{{{{{{{\pi}_{!}{\pi}^{!}\rightarrow1}}}}}}", from=5-1, to=7-1]
%	\arrow["{{{{{{{\pi}_{!}{\pi}^{!}\rightarrow1}}}}}}", from=5-3, to=7-3]
%\end{tikzcd}
%}}}%
%%%%%%%%%%%%%%%%%
\adjustbox{max width=\textwidth}{
\begin{tikzcd}
	{E_{2}^{s,t}(\mathcal{X}):=H^{s}_{c}\left({\mathcal{X}},\mathbf{Z}/{{p^r}%
}(d-\tfrac{t}{2}
)[2d]\right)} & \Longrightarrow& {\pi_{-s-t}\left({j_{!}}{\pi}_{*}{\pi}%
^{*}\mathcal{K}{(\mathcal{X})} \right)} \\
	& {\mathsf{(C)}} \\
	{E_{2}^{s,t}(\mathcal{X}):=H^{s}_{c}\left({\mathcal{X}},{\pi}^{!}\mathbf
{Z}/{{p^r}}(-\tfrac{t}{2} )\right)} & \Longrightarrow& {\pi_{-s-t}\left
({j_{!}}{\pi}_{*}\Sigma^{2d}{\pi}^{*}\mathcal{K}\otimes\mathbf{Z}/{{p^r}%
}(d){(\mathcal{O}_S)} \right)} \\
	& {\mathsf{(B)}} \\
	{E_{2}^{s,t}(\mathcal{X}):=H^{s}\left({\mathcal{O}_F},{j_{!}}{\pi}_{!}{\pi
}^{!}\mathbf{Z}/{{p^r}}(-\tfrac{t}{2} )\right)} & \Longrightarrow& {\pi
_{-s-t}\left({j_{!}}{\pi}_{!}{\pi}^{!}\mathcal{K}{(\mathcal{O}_S)}\right)} \\
	& {\mathsf{(A)}} \\
	{E_{2}^{s,t}{(\mathcal{O}_S)}:=H^{s}\left({\mathcal{O}_F},{j_{!}}\mathbf
{Z}/{{p^r}}(-\tfrac{t}{2} )\right)} & \Longrightarrow& {\pi_{-s-t}\left
({j_{!}}\mathcal{K}{(\mathcal{O}_F)}\right)}
	\arrow[equals, from=1-1, to=3-1]
	\arrow[equals, from=1-3, to=3-3]
	\arrow["{{{{{\pi}_{!}={\pi}_{*}}}}}", equals, from=3-1, to=5-1]
	\arrow["{{{{{\pi}_{!}={\pi}_{*}}}}}", equals, from=3-3, to=5-3]
	\arrow["{{{{{{{\pi}_{!}{\pi}^{!}\rightarrow1}}}}}}", from=5-1, to=7-1]
	\arrow["{{{{{{{\pi}_{!}{\pi}^{!}\rightarrow1}}}}}}", from=5-3, to=7-3]
\end{tikzcd}
}%
%%%%%%%%%%%%%%%%%
\]
The right square is Diagram \ref{lset1}. Since $\pi$ is proper, $\pi_{!}%
=\pi_{\ast}$, and since $\pi$ is smooth of relative dimension $d$, $\pi
^{!}\mathcal{F}\cong\Sigma^{2d}\pi^{\ast}\mathcal{F}\otimes\mathbf{Z}%
/p^{r}(d)$, providing us with the squares \textsf{(B)},\textsf{ (C)}. Finally,
the identifications with compactly supported cohomology rely just on Remark
\ref{rmk_CompactlySupportedEtaleCohomology}, both schemes $\mathcal{X}$ and
$\mathcal{S}=\operatorname*{Spec}\mathcal{O}_{S}$ come with natural structure
morphisms to $\mathcal{S}$.
\end{proof}

\begin{lemma}
\label{lemma_vanish_at_0_0}$H_{c}^{0}\left(  \mathcal{O}_{S},\mathbf{Z}%
/p^{r}(0)\right)  =0$.
\end{lemma}

\begin{proof}
There is an exact sequence%
\[
0\rightarrow H_{c}^{0}(\mathcal{O}_{S},\mathbf{Z})\rightarrow\mathbf{Z}%
\rightarrow\bigoplus_{v\in S\setminus S_{\infty}}H^{0}(F_{v},\mathbf{Z}%
)\oplus\bigoplus_{v\in S\cap S_{\infty}}H_{T}^{0}(F_{v},\mathbf{Z})\rightarrow
H_{c}^{1}(\mathcal{O}_{S},\mathbf{Z})\rightarrow0
\]
and the middle arrow is injective as $S$ contains at least one finite place by
our standing assumption that $\frac{1}{p}\in\mathcal{O}_{S}$. For this, see
the proof of \cite[Ch. II, Cor. 2.11]{milne2006}, and note that there is a
slight typo in the statement of the corollary (but the matter gets clear from
inspecting the proof). Moreover, in Milne's convention $H^{\bullet}(F_{v},-)$
refers to Tate cohomology for local fields. To avoid misunderstandings, we
have therefore (unlike \textit{loc. cit.}) split the direct sum into two sums
and wrote $H_{T}$ for Tate cohomology. We conclude that $H_{c}^{0}%
(\mathcal{O}_{S},\mathbf{Z})=0$. Letting $\wp$ be the projection onto the
direct summand of finite places, we obtain the commutative diagram%
\[%
%%%%%%%%%%
%\adjustbox{max width=\textwidth}{
%\begin{tikzcd}
%	{\mathbf{Z}} && {\bigoplus_{v\in S\setminus S_{\infty}}H^{0}(F_{v},\mathbf
%{Z})\oplus\bigoplus_{v\in S_{\infty}}H_{T}^{0}(F_{v},\mathbf{Z})}
%&& {H^{1}_c(\mathcal{O}_S,\mathbf{Z})} \\
%	\\
%	{\mathbf{Z}} && {\bigoplus_{v\in S\setminus S_{\infty}}H^{0}(F_{v},\mathbf
%{Z})} && Q
%	\arrow[hook, from=1-1, to=1-3]
%	\arrow[equals, from=1-1, to=3-1]
%	\arrow[two heads, from=1-3, to=1-5]
%	\arrow["\wp", two heads, from=1-3, to=3-3]
%	\arrow[from=1-5, to=3-5]
%	\arrow["{\operatorname{diag}}"', hook, from=3-1, to=3-3]
%	\arrow[two heads, from=3-3, to=3-5]
%\end{tikzcd}
%}
%}}}%
%%%%%%%%%%%%%%%%%
{
\adjustbox{max width=\textwidth}{
\begin{tikzcd}
	{\mathbf{Z}} && {\bigoplus_{v\in S\setminus S_{\infty}}H^{0}(F_{v},\mathbf
{Z})\oplus\bigoplus_{v\in S_{\infty}}H_{T}^{0}(F_{v},\mathbf{Z})}
&& {H^{1}_c(\mathcal{O}_S,\mathbf{Z})} \\
	\\
	{\mathbf{Z}} && {\bigoplus_{v\in S\setminus S_{\infty}}H^{0}(F_{v},\mathbf
{Z})} && Q
	\arrow[hook, from=1-1, to=1-3]
	\arrow[equals, from=1-1, to=3-1]
	\arrow[two heads, from=1-3, to=1-5]
	\arrow["\wp", two heads, from=1-3, to=3-3]
	\arrow[from=1-5, to=3-5]
	\arrow["{\operatorname{diag}}"', hook, from=3-1, to=3-3]
	\arrow[two heads, from=3-3, to=3-5]
\end{tikzcd}
}
}%
%%%%%%%%%%%%%%%%%
\]
with exact rows, where $Q$ is a free abelian group since the lower left
horizontal arrow is the diagonal embedding of a copy of $\mathbf{Z}$ into
factors all of which satisfy $H^{0}(F_{v},\mathbf{Z})\cong\mathbf{Z}$ since
$v$ is a finite place. Moreover, $H_{T}^{0}(F_{v},\mathbf{Z})=0$ at complex
places and $H_{T}^{0}(F_{v},\mathbf{Z})$ is $2$-torsion at real places. The
snake lemma therefore shows that $H_{c}^{1}(\mathcal{O}_{S},\mathbf{Z})$ is an
extension of a free abelian group by a $2$-torsion group. We deduce that
$H_{c}^{1}(\mathcal{O}_{S},\mathbf{Z})$ has no $p$-torsion, as $p$ is odd. The
long exact sequence of multiplication with $p^{r}$ shows that%
\[
\cdots\overset{\cdot p^{r}}{\longrightarrow}\underset{=0}{H_{c}^{0}%
(\mathcal{O}_{S},\mathbf{Z})}\longrightarrow H_{c}^{0}(\mathcal{O}%
_{S},\mathbf{Z}/p^{r})\longrightarrow H_{c}^{1}(\mathcal{O}_{S},\mathbf{Z}%
)\overset{\cdot p^{r}}{\longrightarrow}H_{c}^{1}(\mathcal{O}_{S},\mathbf{Z})
\]
and since $H_{c}^{1}(\mathcal{O}_{S},\mathbf{Z})$ has no $p$-torsion, the
kernel of multiplication by $p^{r}$ must be trivial.
\end{proof}

The following observation is genuinely special to $\operatorname*{Spec}%
\mathbf{Z}\left[  \frac{1}{p}\right]  $. It works only since $p$ is odd and we
remove precisely one finite place.

\begin{lemma}
\label{lemma_special_vanishing}$H_{c}^{1}\left(  \mathbf{Z}\left[  \frac{1}%
{p}\right]  ,\mathbf{Z}/p^{r}(0)\right)  =0$.
\end{lemma}

\begin{proof}
(1) We use \cite[Ch. II, Prop. 2.1]{milne2006} for $U:=\mathbf{Z}\left[
\frac{1}{p}\right]  $ and $S:=\{p,\infty\}$. Following the computation in
\textit{loc. cit.} Remark 2.2. (a), we obtain the exact sequence%
\[
0\longrightarrow H^{2}\left(  \mathbf{Z}\left[  \tfrac{1}{p}\right]
,\mathbf{G}_{m}\right)  \longrightarrow\operatorname*{Br}(\mathbf{Q}%
_{p})\oplus\operatorname*{Br}(\mathbf{R})\overset{\Sigma}{\longrightarrow
}\mathbf{Q}/\mathbf{Z}\longrightarrow0
\]
and therefore $H^{2}(\mathbf{Z}\left[  \tfrac{1}{p}\right]  ,\mathbf{G}%
_{m})\cong\mathbf{Z}/2$, the Brauer group of the single real place. From the
Kummer sequence $\mathbf{Z}/p^{r}(1)\longrightarrow\mathbf{G}_{m}%
\overset{\cdot p^{r}}{\longrightarrow}\mathbf{G}_{m}$ we therefore obtain%
\[
\operatorname*{Pic}\left(  \mathbf{Z}\left[  \tfrac{1}{p}\right]  \right)
\longrightarrow H^{2}(\mathbf{Z}\left[  \tfrac{1}{p}\right]  ,\mathbf{Z}%
/p^{r}(1))\longrightarrow\mathbf{Z}/2\overset{\cdot p^{r}}{\longrightarrow
}\mathbf{Z}/2\text{,}%
\]
but since $\mathbf{Z}\left[  \frac{1}{p}\right]  $ is a principal ideal
domain, it has trivial class group and therefore $H^{2}(\mathbf{Z}\left[
\tfrac{1}{p}\right]  ,\mathbf{Z}/p^{r}(1))=0$. We write $q\colon
\operatorname*{Spec}\mathbf{Z}\left[  \tfrac{1}{p}\right]  \hookrightarrow
\operatorname*{Spec}\mathbf{Z}$ for the open immersion. From classical
Artin--Verdier duality we now obtain%
\[
H_{c}^{1}\left(  \mathbf{Z}\left[  \tfrac{1}{p}\right]  ,\mathbf{Z}%
/p^{r}(0)\right)  =H^{1}\left(  \mathbf{Z},q_{!}\mathbf{Z}/p^{r}(0)\right)
\cong H^{2}\left(  \mathbf{Z}\left[  \tfrac{1}{p}\right]  ,\mathbf{Z}%
/p^{r}(1)\right)  ^{\vee}=0\text{,}%
\]
which is what we needed to show for (1).
\end{proof}

Next, we make a series of observations which all stem from analyzing the
descent spectral sequence in settings of varying generality.

\begin{lemma}
\label{lemma_supp}For all $s>2d+3$ and $r\geq2$, we have $E_{r}^{s,t}%
(\mathcal{X})=0$.
\end{lemma}

\begin{proof}
Since $\pi\colon\mathcal{X}\rightarrow q$ is of relative dimension $d$ by
assumption (see \S \ref{sect_Setup}), the object $\pi_{\ast}\pi^{\ast
}\mathbf{Z}/p^{r}(-\frac{t}{2})$ is represented by a complex concentrated at
worst in degrees $[0,2d]$. As $p$ is odd, the full ring of integers
$\operatorname*{Spec}\mathcal{O}_{F}$ has $p$-cohomological dimension $3$, and
therefore any compactly supported \'{e}tale cohomology of a single torsion
sheaf is concentrated in degrees $[0,3]$. It follows that $H^{s}\left(
\mathcal{O}_{F},j_{!}\pi_{\ast}\pi^{\ast}\mathbf{Z}/p^{r}(-\frac{t}%
{2})\right)  =0$ as soon as $s>2d+3$. This proves the claim for the page
$r=2$, and the terms $E_{r}^{s,t}(\mathcal{X})$ then also must vanish on all
later pages.
\end{proof}

Hence, the spectral sequence is supported in the right half plane. It is only
supported in finitely many columns, and therefore it becomes stationary after
finitely many pages.

\begin{lemma}
\label{lem_OOSpecSeq}Our assumptions are as in \S \ref{sect_Setup} and $S$ is finite.

\begin{enumerate}
\item The spectral sequence naturally defines maps%
\begin{align}
&  H_{c}^{2d+3,-(2d+2)}(\mathcal{X},\mathbf{Z}/p^{r}(d+1))=E_{2}%
^{2d+3,-(2d+2)}(\mathcal{X})\label{lh07}\\
&  \qquad\qquad\qquad\underset{(\operatorname*{I})}{\twoheadrightarrow
}E_{\infty}^{2d+3,-(2d+2)}(\mathcal{X})\underset{(\operatorname*{II}%
)}{\hookrightarrow}\pi_{-1}\Sigma^{-1}K/p^{r}(\underline{{\mathsf{LC}}%
}_{\mathcal{O}_{S}})[\tau^{-1}]\text{,}\nonumber
\end{align}
arising from an edge map $(\operatorname*{I})$ and the lowest filtration step
inclusion of the filtration on the homotopy groups $(\operatorname*{II})$.

\item Suppose $d=0$. Then the filtration underlying the descent spectral
sequence provides us with the exact sequence%
\begin{equation}
H_{c}^{3}(\mathcal{O}_{S},\mathbf{Z}/p^{r}(1))\overset{e}{\hookrightarrow}%
\pi_{-1}\left(  \Sigma^{-1}K/p^{r}(\underline{{\mathsf{LC}}}_{\mathcal{O}_{S}%
})[\tau^{-1}]\right)  \twoheadrightarrow H_{c}^{1}(\mathcal{O}_{S}%
,\mathbf{Z}/p^{r}(0))\text{.} \label{lh06}%
\end{equation}

\item For $\mathcal{O}_{S}=\mathbf{Z}\left[  \frac{1}{p}\right]  $, the map
$e$ from Eq. \ref{lh06} specializes to an isomorphism%
\begin{equation}
\mathbf{Z}/p^{r}\cong H_{c}^{3}\left(  \mathbf{Z}\left[  \frac{1}{p}\right]
,\mathbf{Z}/p^{r}(1)\right)  \underset{e}{\cong}\pi_{-1}\left(  \Sigma
^{-1}K/p^{r}(\underline{{\mathsf{LC}}}_{\mathbf{Z}\left[  \frac{1}{p}\right]
})[\tau^{-1}]\right)  \text{.} \label{lh06b}%
\end{equation}

\end{enumerate}
\end{lemma}

\begin{proof}
Consider the $E_{2}$-page of the spectral sequence from Eq. \ref{lcih1}. We
may consider both $\mathcal{X}$ and $\operatorname*{Spec}\mathcal{O}_{S}$ (the
latter just amounts to letting $\mathcal{X}:=\operatorname*{Spec}%
\mathcal{O}_{S}$ so that the relative dimension is $d=0$). Consider the
special location $(2d+3,-(2d+2))$ on the second page:
\[
E_{2}^{2d+3,-(2d+2)}(\mathcal{X})=H_{c}^{2d+3}\left(  \mathcal{X}%
,\mathbf{Z}/p^{r}(d+1)\right)  \text{.}%
\]
The differential $d_{2}$, as well as the differentials originating from this
spot on all further pages, lead outside the support of the spectral sequence
by Lemma \ref{lemma_supp}, so when we form $\frac{\ker d_{r}}%
{\operatorname*{im}d_{r}}$ to compute each consecutive page, the kernel is
always the entire group. Solely the operation to form the quotient by
$\operatorname*{im}d_{r}$ might do something non-trivial to our group. Hence,
we obtain a chain of surjections%
\begin{equation}
E_{2}^{2d+3,-(2d+2)}(\mathcal{X})\twoheadrightarrow E_{3}^{2d+3,-(2d+2)}%
(\mathcal{X})\twoheadrightarrow\ldots\twoheadrightarrow E_{\infty
}^{2d+3,-(2d+2)}(\mathcal{X})\text{.} \label{lhj1}%
\end{equation}
This is the map $(\operatorname*{I})$ in Eq. \ref{lh07}. By our support
consideration, we have strong convergence and this sequence is stationary
after finitely many steps. Let us temporarily abbreviate%
\[
\mathbf{\Pi}_{-1}(\mathcal{X}):=\pi_{-(2d+3)+(2d+2)}(\Sigma^{-1}%
K/p^{r})(\underline{{\mathsf{LC}}}_{\mathcal{X}})[\tau^{-1}]\text{.}%
\]
There is a descending filtration on this group,
\begin{equation}
\ldots\subseteq F_{1}\mathbf{\Pi}_{-1}\subseteq F_{0}\mathbf{\Pi}_{-1}\text{,}
\label{lhj2}%
\end{equation}
along with a canonical identification%
\[
\frac{F_{2d+3}\mathbf{\Pi}_{-1}}{F_{2d+4}\mathbf{\Pi}_{-1}}\cong E_{\infty
}^{2d+3,-(2d+2)}(\mathcal{X})\text{.}%
\]
Since $F_{2d+4}\mathbf{\Pi}_{-1}=0$ for support reasons, this simplifies to%
\[
F_{2d+3}\mathbf{\Pi}_{-1}\cong E_{\infty}^{2d+3,-(2d+2)}(\mathcal{X})\text{.}%
\]
This is the lowest filtered part and the inclusion of this into $\mathbf{\Pi
}_{-1}$ is the map $(\operatorname*{II})$ of Eq. \ref{lh07}. This proves (1).
In the special case $\mathcal{X}:=\operatorname*{Spec}\mathcal{O}_{S}$ and
$d=0$ we can say more about the $E_{2}$-page. By Lemma \ref{lemma_supp} and
$d=0$ there are at worst four columns and the degree one diagonal is highlit
below:%
\begin{equation}%
%%%%%%%%%%
%\begin{tabular}
%[c]{rcccccc}
%$\vdots\,\,$ & \multicolumn{1}{|c}{} & $\vdots$ & $\vdots$ & $\vdots$ &
%$\vdots$ & \\
%$\mathsf{2}$ & \multicolumn{1}{|c}{$\quad0\quad$} & $H_{c}^{0}(\mathcal{O}%
%_{S},\mathbf{Z}/p^{r}(-1))$ & $H_{c}^{1}(\mathcal{O}_{S},\mathbf{Z}%
%/p^{r}(-1))$ & $H_{c}^{2}(\mathcal{O}_{S},\mathbf{Z}/p^{r}(-1))$ & $\quad
%H_{c}^{3}(\mathcal{O}_{S},\mathbf{Z}/p^{r}(-1))\quad$ & $\quad0\quad$\\
%$\mathsf{1}$ & \multicolumn{1}{|c}{$\quad0\quad$} & $0{\cellcolor{lg}}
%$ & $0$ & $0$ & $\quad0\quad$ & $\quad0\quad$\\
%$\mathsf{0}$ & \multicolumn{1}{|c}{$\quad0\quad$} & $H_{c}^{0}(\mathcal{O}%
%_{S},\mathbf{Z}/p^{r}(0))$ & $H_{c}^{1}(\mathcal{O}_{S},\mathbf{Z}%
%/p^{r}(0)) {\cellcolor{lg}} $ & $H_{c}^{2}(\mathcal{O}_{S},\mathbf{Z}%
%/p^{r}(0))$ & $\quad H_{c}^{3}(\mathcal{O}_{S},\mathbf{Z}/p^{r}(0))\quad
%$ & $\quad0\quad$\\
%$\mathsf{-1}$ & \multicolumn{1}{|c}{$\quad0\quad$}
%& $0$ & $0$ & $0 {\cellcolor{lg}} $ & $\quad0\quad$ & $\quad0\quad$\\
%$\mathsf{-2}$ & \multicolumn{1}{|c}{$\quad0\quad$} & $H_{c}^{0}(\mathcal
%{O}_{S},\mathbf{Z}/p^{r}(1))$ & $H_{c}^{1}(\mathcal{O}_{S},\mathbf{Z}%
%/p^{r}(1))$
%& $H_{c}^{2}(\mathcal{O}_{S},\mathbf{Z}/p^{r}(1))$ & $H_{c}^{3}(\mathcal
%{O}_{S},\mathbf{Z}/p^{r}(1)) {\cellcolor{lg}} $ & $\quad0\quad$\\
%$\vdots\,\,$ & \multicolumn{1}{|c}{} & $\vdots$ & $\vdots$ & $\vdots$ &
%$\vdots$ & \\\cline{2-7}
%& $\mathsf{-1}$ & $\mathsf{0}$ & $\mathsf{1}$ & $\mathsf{2}$ & $\mathsf{3}$ &
%$\mathsf{\cdots}$
%\end{tabular}
%}} }%
%%%%%%%%%%%%%%%%%
\adjustbox{max width=\textwidth}{
\begin{tabular}
[c]{rcccccc}
$\vdots\,\,$ & \multicolumn{1}{|c}{} & $\vdots$ & $\vdots$ & $\vdots$ &
$\vdots$ & \\
$\mathsf{2}$ & \multicolumn{1}{|c}{$\quad0\quad$} & $H_{c}^{0}(\mathcal{O}%
_{S},\mathbf{Z}/p^{r}(-1))$ & $H_{c}^{1}(\mathcal{O}_{S},\mathbf{Z}%
/p^{r}(-1))$ & $H_{c}^{2}(\mathcal{O}_{S},\mathbf{Z}/p^{r}(-1))$ & $\quad
H_{c}^{3}(\mathcal{O}_{S},\mathbf{Z}/p^{r}(-1))\quad$ & $\quad0\quad$\\
$\mathsf{1}$ & \multicolumn{1}{|c}{$\quad0\quad$} & $0{\cellcolor{lg}}
$ & $0$ & $0$ & $\quad0\quad$ & $\quad0\quad$\\
$\mathsf{0}$ & \multicolumn{1}{|c}{$\quad0\quad$} & $H_{c}^{0}(\mathcal{O}%
_{S},\mathbf{Z}/p^{r}(0))$ & $H_{c}^{1}(\mathcal{O}_{S},\mathbf{Z}%
/p^{r}(0)) {\cellcolor{lg}} $ & $H_{c}^{2}(\mathcal{O}_{S},\mathbf{Z}%
/p^{r}(0))$ & $\quad H_{c}^{3}(\mathcal{O}_{S},\mathbf{Z}/p^{r}(0))\quad
$ & $\quad0\quad$\\
$\mathsf{-1}$ & \multicolumn{1}{|c}{$\quad0\quad$}
& $0$ & $0$ & $0 {\cellcolor{lg}} $ & $\quad0\quad$ & $\quad0\quad$\\
$\mathsf{-2}$ & \multicolumn{1}{|c}{$\quad0\quad$} & $H_{c}^{0}(\mathcal
{O}_{S},\mathbf{Z}/p^{r}(1))$ & $H_{c}^{1}(\mathcal{O}_{S},\mathbf{Z}%
/p^{r}(1))$
& $H_{c}^{2}(\mathcal{O}_{S},\mathbf{Z}/p^{r}(1))$ & $H_{c}^{3}(\mathcal
{O}_{S},\mathbf{Z}/p^{r}(1)) {\cellcolor{lg}} $ & $\quad0\quad$\\
$\vdots\,\,$ & \multicolumn{1}{|c}{} & $\vdots$ & $\vdots$ & $\vdots$ &
$\vdots$ & \\\cline{2-7}
& $\mathsf{-1}$ & $\mathsf{0}$ & $\mathsf{1}$ & $\mathsf{2}$ & $\mathsf{3}$ &
$\mathsf{\cdots}$
\end{tabular}
}
%%%%%%%%%%%%%%%%%
\label{lfrio1}%
\end{equation}
We are only interested in this diagonal. No term on this diagonal can change
as we proceed to the $E_{3}$-page. The only term which could possibly change
as we go to the $E_{4}$-page is $E_{3}^{3,-2}(\mathcal{O}_{S})$ since there is
an incoming differential%
\[
d_{3}^{0,0}\colon E_{3}^{0,0}\longrightarrow E_{3}^{3,-2}\text{,}%
\]
i.e.,%
\[
d_{3}^{0,0}\colon H_{c}^{0}(\mathcal{O}_{S},\mathbf{Z}/p^{r}%
(0))\longrightarrow H_{c}^{3}(\mathcal{O}_{S},\mathbf{Z}/p^{r}(1))\text{,}%
\]
but the left term is zero by Lemma \ref{lemma_vanish_at_0_0}. The diagonal, as
highlit in Diagram \ref{lfrio1} hence agrees with the diagonal on the
$E_{\infty}$-page. This settles (2). Finally, in the special case
$\mathcal{O}_{S}=\mathbf{Z}\left[  \frac{1}{p}\right]  $, Lemma
\ref{lemma_special_vanishing} now shows that the filtration $F_{\bullet
}\mathbf{\Pi}_{-1}\left(  \mathbf{Z}\left[  \frac{1}{p}\right]  \right)  $
trivializes and has only one step:%
\begin{align*}
\mathbf{Z}/p^{r}\overset{\operatorname*{tr}}{\longleftarrow}H_{c}^{3}\left(
\mathbf{Z}\left[  \frac{1}{p}\right]  ,\mathbf{Z}/p^{r}(1)\right)
\cong\mathbf{\Pi}_{-1}\left(  \mathbf{Z}\left[  \frac{1}{p}\right]  \right)
&  =\pi_{-1}(\Sigma^{-1}K/p^{r})({\mathsf{LC}}_{\mathbf{Z}\left[  \frac{1}%
{p}\right]  })[\tau^{-1}]\\
&  \cong\pi_{0}(K/p^{r})({\mathsf{LC}}_{\mathbf{Z}\left[  \frac{1}{p}\right]
})[\tau^{-1}]\text{.}%
\end{align*}

This settles (3).
\end{proof}

\subsection{Open $!$-pushforward for rings of integers}

We remain in the setting of \S \ref{sect_Setup}, but work with two distinct
choices of finitely many places of $F$: Suppose $S\subseteq S^{\prime}$ are
finite sets of places such that all infinite places of $F$ are contained in
$S$ (and thus also $S^{\prime}$). Moreover, suppose $\frac{1}{p}\in
\mathcal{O}_{S}$, and thus also $\frac{1}{p}\in\mathcal{O}_{S^{\prime}}$. We
obtain an open immersion\footnote{we chose upper case $J$ as $j$ is already in
use}%
\begin{equation}
J\colon\operatorname*{Spec}\mathcal{O}_{S^{\prime}}\hookrightarrow
\operatorname*{Spec}\mathcal{O}_{S}\text{.} \label{lbipsw1}%
\end{equation}
The co-unit of the adjunction $J_{!}\dashv J^{!}$ is the natural
transformation $J_{!}J^{!}\rightarrow1$ and it induces a $!$-pushforward on
\'{e}tale cohomology. In this section we show that this map has a lift to
$K({\mathsf{LC}}_{\mathcal{O}_{S^{\prime}}})\rightarrow K({\mathsf{LC}%
}_{\mathcal{O}_{S}})$, even without \'{e}tale (hyper)sheafification or working
$K(1)$-locally.

\begin{lemma}
\label{lem_OO1}The solid arrows in the following diagram commute:%
\[%
%%%%%%%%%%
%\begin{tikzcd}
%	{\bigoplus\limits_{v\in{S'} \setminus S_{\infty}}\pi_{-1}\Sigma^{-1}%
%K(F_{v})[\tau^{-1}]} && {\pi_{-1}\Sigma^{-1}K({\mathsf{LC}}_{\mathcal{O}_{S'}%
%})[\tau^{-1}]} \\
%	\\
%	{\bigoplus\limits_{v\in{S} \setminus S_{\infty}}\pi_{-1}\Sigma^{-1}%
%K(F_{v})[\tau^{-1}]} && {\pi_{-1}\Sigma^{-1}K({\mathsf{LC}}_{\mathcal{O}%
%_S})[\tau^{-1}],}
%	\arrow["{{{{{\tilde{E}'}}}}}"{description}, from=1-1, to=1-3]
%	\arrow["{{{{{\tilde{P}}}}}}", shift left=3, dashed, from=1-1, to=3-1]
%	\arrow["B"{description}, from=1-3, to=3-3]
%	\arrow["{{{{{\tilde{I}}}}}}", shift left=3, from=3-1, to=1-1]
%	\arrow["{{{{{\tilde{E}}}}}}"{description}, from=3-1, to=3-3]
%\end{tikzcd}
%}}}%
%%%%%%%%%%%%%%%%%
{
\begin{tikzcd}
	{\bigoplus\limits_{v\in{S'} \setminus S_{\infty}}\pi_{-1}\Sigma^{-1}%
K(F_{v})[\tau^{-1}]} && {\pi_{-1}\Sigma^{-1}K({\mathsf{LC}}_{\mathcal{O}_{S'}%
})[\tau^{-1}]} \\
	\\
	{\bigoplus\limits_{v\in{S} \setminus S_{\infty}}\pi_{-1}\Sigma^{-1}%
K(F_{v})[\tau^{-1}]} && {\pi_{-1}\Sigma^{-1}K({\mathsf{LC}}_{\mathcal{O}%
_S})[\tau^{-1}],}
	\arrow["{{{{{\tilde{E}'}}}}}"{description}, from=1-1, to=1-3]
	\arrow["{{{{{\tilde{P}}}}}}", shift left=3, dashed, from=1-1, to=3-1]
	\arrow["B"{description}, from=1-3, to=3-3]
	\arrow["{{{{{\tilde{I}}}}}}", shift left=3, from=3-1, to=1-1]
	\arrow["{{{{{\tilde{E}}}}}}"{description}, from=3-1, to=3-3]
\end{tikzcd}
}%
%%%%%%%%%%%%%%%%%
\]
where $B$ comes from the exact functor forgetting the $\mathcal{O}_{S^{\prime
}}$-module structure in favour of the $\mathcal{O}_{S}$-module structure (and
keeping the locally compact topology unchanged\footnote{This obviously sends
real vector spaces to real vector spaces.}), and $\tilde{E}$, $\tilde
{E}^{\prime}$ send a finite-dimensional $F_{v}$-vector space to itself, now
regarded as a locally compact module.\footnote{All local fields $F_{v}$ are
finite extensions of $\mathbf{Q}_{p}$ or $\mathbf{R}$ and therefore
finite-dimensional vector spaces over these fields carry a canonical locally
compact topology.}

\begin{enumerate}
\item The map $\tilde{I}$ is the inclusion of direct summands.

\item The map $\tilde{P}$ is the reverse projection onto the direct summands
with $v\in S$.

\item Evidently, $\tilde{P}\circ\tilde{I}=\operatorname*{id}$.
\end{enumerate}
\end{lemma}

\begin{proof}
The following diagram of exact categories and exact functors\footnote{Note
that any $F_{v}$-linear map between finite-dimensional $F_{v}$-vector spaces
is automatically continuous with respect to the topology.} is evidently
commutative%
\[%
%%%%%%%%%%
%\begin{tikzcd}
%	{\bigoplus\limits_{v\in{S'}\setminus S_{\infty}}\mathsf{Vect}_{fd}(F_{v})}
%&& {{\mathsf{LC}}_{\mathcal{O}_{S'}}} \\
%	\\
%	{\bigoplus\limits_{v\in{S}\setminus S_{\infty}}\mathsf{Vect}_{fd}(F_{v})}
%&& {{\mathsf{LC}}_{\mathcal{O}_S}.}
%	\arrow["{{{{{{\tilde{E}'}}}}}}"{description}, from=1-1, to=1-3]
%	\arrow["B"{description}, from=1-3, to=3-3]
%	\arrow["{{{{{{\tilde{I}}}}}}}", from=3-1, to=1-1]
%	\arrow["{{{{{{\tilde{E}}}}}}}"{description}, from=3-1, to=3-3]
%\end{tikzcd}
%}}}%
%%%%%%%%%%%%%%%%%
{
\begin{tikzcd}
	{\bigoplus\limits_{v\in{S'}\setminus S_{\infty}}\mathsf{Vect}_{fd}(F_{v})}
&& {{\mathsf{LC}}_{\mathcal{O}_{S'}}} \\
	\\
	{\bigoplus\limits_{v\in{S}\setminus S_{\infty}}\mathsf{Vect}_{fd}(F_{v})}
&& {{\mathsf{LC}}_{\mathcal{O}_S}.}
	\arrow["{{{{{{\tilde{E}'}}}}}}"{description}, from=1-1, to=1-3]
	\arrow["B"{description}, from=1-3, to=3-3]
	\arrow["{{{{{{\tilde{I}}}}}}}", from=3-1, to=1-1]
	\arrow["{{{{{{\tilde{E}}}}}}}"{description}, from=3-1, to=3-3]
\end{tikzcd}
}%
%%%%%%%%%%%%%%%%%
\]
Applying $K$-theory, inverting the Bott element $\tau$, shifts and taking
$\pi_{-1}$, all preserve this commutativity.
\end{proof}

\begin{lemma}
\label{lem_OO2}The solid arrows in the following diagram commute%
\[%
%%%%%%%%%%
%\begin{tikzcd}
%	{\bigoplus\limits_{v\in{S'}}\operatorname{Br}(F_{v})_{\circ}} && {H_{c}%
%^{3}(\mathcal{O}_{S'},\mathbf{Z}/p^r(1))} \\
%	\\
%	{\bigoplus\limits_{v\in{S}}\operatorname{Br}(F_{v})_{\circ}} && {H_{c}%
%^{3}(\mathcal{O}_{S},\mathbf{Z}/p^r(1)),}
%	\arrow["{{{{E'}}}}"{description}, two heads, from=1-1, to=1-3]
%	\arrow["P", shift left=3, dashed, from=1-1, to=3-1]
%	\arrow["W"{description}, equals, from=1-3, to=3-3]
%	\arrow["I", shift left=3, from=3-1, to=1-1]
%	\arrow["E"{description}, two heads, from=3-1, to=3-3]
%\end{tikzcd}
%}}}%
%%%%%%%%%%%%%%%%%
{
\begin{tikzcd}
	{\bigoplus\limits_{v\in{S'}}\operatorname{Br}(F_{v})_{\circ}} && {H_{c}%
^{3}(\mathcal{O}_{S'},\mathbf{Z}/p^r(1))} \\
	\\
	{\bigoplus\limits_{v\in{S}}\operatorname{Br}(F_{v})_{\circ}} && {H_{c}%
^{3}(\mathcal{O}_{S},\mathbf{Z}/p^r(1)),}
	\arrow["{{{{E'}}}}"{description}, two heads, from=1-1, to=1-3]
	\arrow["P", shift left=3, dashed, from=1-1, to=3-1]
	\arrow["W"{description}, equals, from=1-3, to=3-3]
	\arrow["I", shift left=3, from=3-1, to=1-1]
	\arrow["E"{description}, two heads, from=3-1, to=3-3]
\end{tikzcd}
}%
%%%%%%%%%%%%%%%%%
\]
and the horizontal arrows are surjective, where we write $\operatorname*{Br}%
(F_{v})_{\circ}\subseteq\operatorname*{Br}(F_{v})$ for the subgroup of
elements which $E$ (or equivalently $E^{\prime}$) maps to elements $y\in
H_{c}^{3}(\mathcal{O}_{S^{\prime}},\mathbf{G}_{m})$ such that $p^{r}y=0$. We
will provide more details in the proof how $E$ and $E^{\prime}$ are set up.

\begin{enumerate}
\item The map $W$ is induced from $J_{!}J^{!}\rightarrow1$ for the open
immersion in Eq. \ref{lbipsw1}.

\item The map $I$ is the inclusion of direct summands.

\item The map $P$ is the reverse projection onto the direct summands with
$v\in S$.

\item Evidently, $P\circ I=\operatorname*{id}$.
\end{enumerate}
\end{lemma}

\begin{proof}
\textit{(Step 1)} First, we need to set up the diagram at all. The Kummer
sequence induces%
\begin{equation}
H_{c}^{2}(\mathcal{O}_{S},\mathbf{G}_{m})\longrightarrow H_{c}^{3}%
(\mathcal{O}_{S},\mathbf{Z}/p^{r}(1))\longrightarrow H_{c}^{3}(\mathcal{O}%
_{S},\mathbf{G}_{m})\overset{\cdot p^{r}}{\longrightarrow}H_{c}^{3}%
(\mathcal{O}_{S},\mathbf{G}_{m}) \label{lhibso1}%
\end{equation}
and by \cite[Ch. II, Prop. 2.6]{milne2006} $H_{c}^{2}(\mathcal{O}%
_{S},\mathbf{G}_{m})=0$. Thus, $H_{c}^{3}(\mathcal{O}_{S},\mathbf{Z}%
/p^{r}(1))$ can be identified with the subgroup of $p^{r}$-annihilated
elements inside $H_{c}^{3}(\mathcal{O}_{S},\mathbf{G}_{m})$. By the first
lines of the proof of \cite[Ch. II, Prop. 2.6]{milne2006}, there is an exact
sequence%
\begin{equation}%
%%%%%%%%%%
%%%%%%%%%%%%%%%%%
{\displaystyle\bigoplus\limits_{v\in S}}
%%%%%%%%%%%%%%%%%
\operatorname*{Br}(F_{v})\overset{L}{\longrightarrow}H_{c}^{3}(\mathcal{O}%
_{S},\mathbf{G}_{m})\longrightarrow H^{3}(\mathcal{O}_{S},\mathbf{G}%
_{m})\longrightarrow0\text{,} \label{lhibso3}%
\end{equation}
but $H^{3}(\mathcal{O}_{S},\mathbf{G}_{m})=0$ by \cite[Ch. II, Rmk. 2.2
(a)]{milne2006} and the fact that $\frac{1}{p}\in\mathcal{O}_{S}$ implies that
$S$ contains at least one finite place. Thus, the direct sum of Brauer groups
surjects onto $H_{c}^{3}(\mathcal{O}_{S},\mathbf{G}_{m})$. As
$\operatorname*{Br}(F_{v})_{\circ}\subseteq\operatorname*{Br}(F_{v})$ denotes
the subgroup of the Brauer group mapping to elements annihilated by $p^{r}$ in
$H_{c}^{3}(\mathcal{O}_{S},\mathbf{G}_{m})$, we therefore tautologically
obtain the following commutative diagram with surjective downward arrows
\[%
%%%%%%%%%%
%\begin{tikzcd}
%	{\bigoplus\limits_{v\in{S}}\operatorname{Br}(F_{v})_{\circ}} && {\bigoplus
%\limits_{v\in{S}}\operatorname{Br}(F_{v})} \\
%	\\
%	{H_{c}^{3}(\mathcal{O}_{S},\mathbf{Z}/p^r(1))} && {H_{c}^{3}(\mathcal{O}%
%_{S},\mathbf{G}_m).}
%	\arrow[hook, from=1-1, to=1-3]
%	\arrow["E"{description}, two heads, from=1-1, to=3-1]
%	\arrow["L"{description}, two heads, from=1-3, to=3-3]
%	\arrow[hook, from=3-1, to=3-3]
%\end{tikzcd}
%}}}%
%%%%%%%%%%%%%%%%%
{
\begin{tikzcd}
	{\bigoplus\limits_{v\in{S}}\operatorname{Br}(F_{v})_{\circ}} && {\bigoplus
\limits_{v\in{S}}\operatorname{Br}(F_{v})} \\
	\\
	{H_{c}^{3}(\mathcal{O}_{S},\mathbf{Z}/p^r(1))} && {H_{c}^{3}(\mathcal{O}%
_{S},\mathbf{G}_m).}
	\arrow[hook, from=1-1, to=1-3]
	\arrow["E"{description}, two heads, from=1-1, to=3-1]
	\arrow["L"{description}, two heads, from=1-3, to=3-3]
	\arrow[hook, from=3-1, to=3-3]
\end{tikzcd}
}%
%%%%%%%%%%%%%%%%%
\]
This entire argument works just as well for $S^{\prime}$, and the
corresponding left downward arrow will be called $E^{\prime}$ and this settles
the construction of $E^{\prime}$. We note that our definition of
$\operatorname*{Br}(F_{v})_{\circ}$ does not depend on using $E$ or
$E^{\prime}$ since these maps do the same on each summand and only differ by
the set of places we sum over.\newline\textit{(Step 2) }The diagram below on
the left commutes%
\[%
%%%%%%%%%%
%\begin{tikzcd}
%	{H_{c}^{3}(\mathcal{O}_{S'},\mathbf{G}_m)} && {\mathbf{Q}/\mathbf{Z}} \\
%	\\
%	{H_{c}^{3}(\mathcal{O}_{S},\mathbf{G}_m)} && {\mathbf{Q}/\mathbf{Z}}
%	\arrow["\sim", from=1-1, to=1-3]
%	\arrow[from=1-1, to=3-1]
%	\arrow[equals, from=1-3, to=3-3]
%	\arrow["\sim"', from=3-1, to=3-3]
%\end{tikzcd}
%}}}%
%%%%%%%%%%%%%%%%%
{
\begin{tikzcd}
	{H_{c}^{3}(\mathcal{O}_{S'},\mathbf{G}_m)} && {\mathbf{Q}/\mathbf{Z}} \\
	\\
	{H_{c}^{3}(\mathcal{O}_{S},\mathbf{G}_m)} && {\mathbf{Q}/\mathbf{Z}}
	\arrow["\sim", from=1-1, to=1-3]
	\arrow[from=1-1, to=3-1]
	\arrow[equals, from=1-3, to=3-3]
	\arrow["\sim"', from=3-1, to=3-3]
\end{tikzcd}
}%
%%%%%%%%%%%%%%%%%
\qquad%
%%%%%%%%%%
%\begin{tikzcd}
%	{H_{c}^{3}(\mathcal{O}_{S'},\mathbf{Z}/p^r(1))} && {\mathbf{Z}/p^r\mathbf{Z}}
%\\
%	\\
%	{H_{c}^{3}(\mathcal{O}_{S},\mathbf{Z}/p^r(1))} && {\mathbf{Z}/p^r\mathbf{Z}}
%	\arrow["\sim", from=1-1, to=1-3]
%	\arrow["{J_! J^! \rightarrow1}"', from=1-1, to=3-1]
%	\arrow[equals, from=1-3, to=3-3]
%	\arrow["\sim"', from=3-1, to=3-3]
%\end{tikzcd}
%}}}%
%%%%%%%%%%%%%%%%%
{
\begin{tikzcd}
	{H_{c}^{3}(\mathcal{O}_{S'},\mathbf{Z}/p^r(1))} && {\mathbf{Z}/p^r\mathbf{Z}}
\\
	\\
	{H_{c}^{3}(\mathcal{O}_{S},\mathbf{Z}/p^r(1))} && {\mathbf{Z}/p^r\mathbf{Z}}
	\arrow["\sim", from=1-1, to=1-3]
	\arrow["{J_! J^! \rightarrow1}"', from=1-1, to=3-1]
	\arrow[equals, from=1-3, to=3-3]
	\arrow["\sim"', from=3-1, to=3-3]
\end{tikzcd}
}%
%%%%%%%%%%%%%%%%%
\]
by \cite[Ch. II, beginning of \S 3, (a)]{milne2006}. From Eq. \ref{lhibso1} it
follows that the right diagram commutes as well. Now take $W$ to be the map
induced by $(J_{!}J^{!}\rightarrow1)$ on the compactly supported cohomology
groups. One can be more explicit, although we do not need to know this: For a
finite place, and using the isomorphism $\operatorname*{Br}(F_{v}%
)\cong\mathbf{Q}/\mathbf{Z}$, this is exactly the subgroup $\left.  \frac
{1}{p^{r}}\mathbf{Z}\right/  \mathbf{Z}$ and for infinite places
$\operatorname*{Br}(\mathbf{R})\cong\left.  \frac{1}{2}\mathbf{Z}\right/
\mathbf{Z}$ or $\operatorname*{Br}(\mathbf{C})=0$, so this subgroup is
trivial. Thus, if we choose to unravel all terms, the commutativity of the
square boils down to the trivially commutative square%
\[%
%%%%%%%%%%
%\begin{tikzcd}
%	{\bigoplus\limits_{v\in{S'}\setminus S_{\infty}}{\left.  \frac{1}{p^{r}%
%}\mathbf{Z}\right/  \mathbf{Z}}} && {\left.  \frac{1}{p^{r}}\mathbf{Z}%
%\right/  \mathbf{Z}} \\
%	\\
%	{\bigoplus\limits_{v\in{S}\setminus S_{\infty}}{\left.  \frac{1}{p^{r}%
%}\mathbf{Z}\right/  \mathbf{Z}}} && {\left.  \frac{1}{p^{r}}\mathbf{Z}%
%\right/  \mathbf{Z}}
%	\arrow["{+}", two heads, from=1-1, to=1-3]
%	\arrow["P", shift left=3, from=1-1, to=3-1]
%	\arrow[equals, from=1-3, to=3-3]
%	\arrow["I", shift left=3, dashed, from=3-1, to=1-1]
%	\arrow["{+}"', two heads, from=3-1, to=3-3]
%\end{tikzcd}
%}}}%
%%%%%%%%%%%%%%%%%
{
\begin{tikzcd}
	{\bigoplus\limits_{v\in{S'}\setminus S_{\infty}}{\left.  \frac{1}{p^{r}%
}\mathbf{Z}\right/  \mathbf{Z}}} && {\left.  \frac{1}{p^{r}}\mathbf{Z}%
\right/  \mathbf{Z}} \\
	\\
	{\bigoplus\limits_{v\in{S}\setminus S_{\infty}}{\left.  \frac{1}{p^{r}%
}\mathbf{Z}\right/  \mathbf{Z}}} && {\left.  \frac{1}{p^{r}}\mathbf{Z}%
\right/  \mathbf{Z}}
	\arrow["{+}", two heads, from=1-1, to=1-3]
	\arrow["P", shift left=3, from=1-1, to=3-1]
	\arrow[equals, from=1-3, to=3-3]
	\arrow["I", shift left=3, dashed, from=3-1, to=1-1]
	\arrow["{+}"', two heads, from=3-1, to=3-3]
\end{tikzcd}
}%
%%%%%%%%%%%%%%%%%
\]
However, for the other proofs it will be better to remember how these groups
arise in \'{e}tale cohomology.
\end{proof}

\begin{lemma}
\label{lem_OO3}The square%
\begin{equation}%
%%%%%%%%%%
%\begin{tikzcd}
%	{\bigoplus\limits_{v \in S} \operatorname{Br}(F_{v})_{\circ}} && {H_{c}%
%^{3}(\mathcal{O}_S,\mathbf{Z}/p^r(1))} \\
%	\\
%	{\pi_{-1}\bigoplus\limits_{v \in S \setminus S_{\infty}}\Sigma^{-1}
%K/p^r(F_{v})[\tau^{-1}]} && {\pi_{-1}\Sigma^{-1} K/p^r({\mathsf{LC}}%
%_{\mathcal{O}_{S}})[\tau^{-1}],}
%	\arrow["E"{description}, from=1-1, to=1-3]
%	\arrow["M"{description}, hook, from=1-1, to=3-1]
%	\arrow["G"{description}, hook, from=1-3, to=3-3]
%	\arrow["{{{{\tilde{E}}}}}"{description}, from=3-1, to=3-3]
%\end{tikzcd}
%}} }%
%%%%%%%%%%%%%%%%%
{
\begin{tikzcd}
	{\bigoplus\limits_{v \in S} \operatorname{Br}(F_{v})_{\circ}} && {H_{c}%
^{3}(\mathcal{O}_S,\mathbf{Z}/p^r(1))} \\
	\\
	{\pi_{-1}\bigoplus\limits_{v \in S \setminus S_{\infty}}\Sigma^{-1}
K/p^r(F_{v})[\tau^{-1}]} && {\pi_{-1}\Sigma^{-1} K/p^r({\mathsf{LC}}%
_{\mathcal{O}_{S}})[\tau^{-1}],}
	\arrow["E"{description}, from=1-1, to=1-3]
	\arrow["M"{description}, hook, from=1-1, to=3-1]
	\arrow["G"{description}, hook, from=1-3, to=3-3]
	\arrow["{{{{\tilde{E}}}}}"{description}, from=3-1, to=3-3]
\end{tikzcd}
}
%%%%%%%%%%%%%%%%%
\label{lct4}%
\end{equation}
commutes (and analogously for $S^{\prime}$), where $\tilde{E}$ is induced from
the functor sending an $F_{v}$-vector space to itself, regarded as a locally
compact module (or equivalently the map $\beta$ of Eq. \ref{l_d1}, or yet
equivalently the same map $\tilde{E}$ as in the statement of Lemma
\ref{lem_OO1}).
\end{lemma}

\begin{remark}
There is no mismatch between the summands being indexed over $S$ resp. over
$S\setminus S_{\infty}$ in Diagram \ref{lct4}, as the $\operatorname*{Br}%
(F_{v})_{\circ}$ are trivial at infinite places.
\end{remark}

\begin{proof}
The equivalence of Proposition \ref{prop_Equiv1} (for $\mathcal{X}%
:=\operatorname*{Spec}\mathcal{O}_{S}$) extends to an identification of fiber
sequences%
\begin{equation}%
%%%%%%%%%%
%\begin{tikzcd}
%	{\Sigma^{-1} K/p^r(\mathcal{O}_{S})[\tau^{-1}]} && {\bigoplus\limits_{v \in
%S \setminus{S_{\infty}}}\Sigma^{-1} K/p^r(F_{v})[\tau^{-1}]} && {j_{!}\pi
%_{*}\pi^{*}\mathcal{K}(\mathcal{O}_{S})} \\
%	\\
%	{\Sigma^{-1} K/p^r(\mathcal{O}_{S})[\tau^{-1}]} && {\bigoplus\limits_{v \in
%S  \setminus{S_{\infty}}}\Sigma^{-1} K/p^r(F_{v})[\tau^{-1}]} && {\Sigma^{-1}
%K/p^r({\mathsf{LC}}_{\mathcal{O}_{S}})[\tau^{-1}]}
%	\arrow[from=1-1, to=1-3]
%	\arrow[equals, from=1-1, to=3-1]
%	\arrow["E"{description}, from=1-3, to=1-5]
%	\arrow[equals, from=1-3, to=3-3]
%	\arrow[equals, from=1-5, to=3-5]
%	\arrow[from=3-1, to=3-3]
%	\arrow["{{\tilde{E}}}"{description}, from=3-3, to=3-5]
%\end{tikzcd}
%}} }%
%%%%%%%%%%%%%%%%%
\adjustbox{max width=\textwidth}{
\begin{tikzcd}
	{\Sigma^{-1} K/p^r(\mathcal{O}_{S})[\tau^{-1}]} && {\bigoplus\limits_{v \in
S \setminus{S_{\infty}}}\Sigma^{-1} K/p^r(F_{v})[\tau^{-1}]} && {j_{!}\pi
_{*}\pi^{*}\mathcal{K}(\mathcal{O}_{S})} \\
	\\
	{\Sigma^{-1} K/p^r(\mathcal{O}_{S})[\tau^{-1}]} && {\bigoplus\limits_{v \in
S  \setminus{S_{\infty}}}\Sigma^{-1} K/p^r(F_{v})[\tau^{-1}]} && {\Sigma^{-1}
K/p^r({\mathsf{LC}}_{\mathcal{O}_{S}})[\tau^{-1}]}
	\arrow[from=1-1, to=1-3]
	\arrow[equals, from=1-1, to=3-1]
	\arrow["E"{description}, from=1-3, to=1-5]
	\arrow[equals, from=1-3, to=3-3]
	\arrow[equals, from=1-5, to=3-5]
	\arrow[from=3-1, to=3-3]
	\arrow["{{\tilde{E}}}"{description}, from=3-3, to=3-5]
\end{tikzcd}
}
%%%%%%%%%%%%%%%%%
\label{lwit1}%
\end{equation}
as is seen from the proof of \textit{loc. cit.} (it amounts to the statement
that the fiber sequence in the left column of Diagram \ref{vchdiag3} can be
identified with the fiber for Eq. \ref{lch2}). Once we evaluate $\pi_{-1}$,
this yields the commutative diagram%
\begin{equation}%
%%%%%%%%%%
%\begin{tikzcd}
%	{\bigoplus\limits_{v \in S} H^{2}(F_{v},\mathbf{Z}/p^r(1))} && {H_{c}%
%^{3}(\mathcal{O}_S,\mathbf{Z}/p^r(1))} \\
%	\\
%	{\pi_{-1} \bigoplus\limits_{v \in S \setminus S_{\infty}}\Sigma^{-1}
%K/p^r(F_{v})[\tau^{-1}]} && {\pi_{-1} j_{!}\pi_{*}\pi^{*}\mathcal{K}%
%(\mathcal{O}_{S})} \\
%	{\pi_{-1}\bigoplus\limits_{v \in S \setminus S_{\infty}}\Sigma^{-1}
%K/p^r(F_{v})[\tau^{-1}]} && {\pi_{-1}\Sigma^{-1} K/p^r({\mathsf{LC}}%
%_{\mathcal{O}_{S}})[\tau^{-1}],}
%	\arrow["e"{description}, from=1-1, to=1-3]
%	\arrow[hook, from=1-1, to=3-1]
%	\arrow[hook, from=1-3, to=3-3]
%	\arrow[from=3-1, to=3-3]
%	\arrow[equals, from=3-1, to=4-1]
%	\arrow[equals, from=3-3, to=4-3]
%	\arrow["{{{\tilde{E}}}}"{description}, from=4-1, to=4-3]
%\end{tikzcd}
%}} }%
%%%%%%%%%%%%%%%%%
{
\begin{tikzcd}
	{\bigoplus\limits_{v \in S} H^{2}(F_{v},\mathbf{Z}/p^r(1))} && {H_{c}%
^{3}(\mathcal{O}_S,\mathbf{Z}/p^r(1))} \\
	\\
	{\pi_{-1} \bigoplus\limits_{v \in S \setminus S_{\infty}}\Sigma^{-1}
K/p^r(F_{v})[\tau^{-1}]} && {\pi_{-1} j_{!}\pi_{*}\pi^{*}\mathcal{K}%
(\mathcal{O}_{S})} \\
	{\pi_{-1}\bigoplus\limits_{v \in S \setminus S_{\infty}}\Sigma^{-1}
K/p^r(F_{v})[\tau^{-1}]} && {\pi_{-1}\Sigma^{-1} K/p^r({\mathsf{LC}}%
_{\mathcal{O}_{S}})[\tau^{-1}],}
	\arrow["e"{description}, from=1-1, to=1-3]
	\arrow[hook, from=1-1, to=3-1]
	\arrow[hook, from=1-3, to=3-3]
	\arrow[from=3-1, to=3-3]
	\arrow[equals, from=3-1, to=4-1]
	\arrow[equals, from=3-3, to=4-3]
	\arrow["{{{\tilde{E}}}}"{description}, from=4-1, to=4-3]
\end{tikzcd}
}
%%%%%%%%%%%%%%%%%
\label{lwit2}%
\end{equation}
where the bottom square comes from the identification of fiber sequences and
the top square from repeating the descent spectral sequence computation of
Lemma \ref{lem_OOSpecSeq} for the local fields $F_{v}$ and using that a map
induces a compatible map of spectral sequences (respecting the underlying
filtration on the homotopy groups $\pi_{\bullet}$). For the local fields at
finite places $v\in S\setminus S_{\infty}$ we have cohomological dimension
$2$, so the spectral sequence collapses earlier, has $E_{2}$-page%
\[%
%%%%%%%%%%
%\begin{tabular}
%[c]{rccccc}$\vdots\,\,$ & \multicolumn{1}{|c}{} & $\vdots$ & $\vdots
%$ & $\vdots$ & \\
%$\mathsf{2}$ & \multicolumn{1}{|c}{$\quad0\quad$} & $H^{0}(F_{v},\mathbf
%{Z}/p^{k}(-1))$ & $H^{1}(F_{v},\mathbf{Z}/p^{k}(-1))$ & $H^{2}(F_{v}%
%,\mathbf{Z}/p^{k}(-1))$ & $\quad0\quad$\\
%$\mathsf{1}$ & \multicolumn{1}{|c}{$\quad0\quad$} & $0$ & $0$ & $0$ &
%$\quad0\quad$\\
%$\mathsf{0}$ & \multicolumn{1}{|c}{$\quad0\quad$} & $H^{0}(F_{v},\mathbf
%{Z}/p^{k}(0)){\cellcolor{lg}} $ & $H^{1}(F_{v},\mathbf{Z}/p^{k}(0))$ & $H^{2}%
%(F_{v},\mathbf{Z}/p^{k}(0))$ &
%$\quad0\quad$\\
%$\mathsf{-1}$ & \multicolumn{1}{|c}{$\quad0\quad$} & $0$ & $0 {\cellcolor{lg}}
%$ & $0$ & $\quad0\quad$\\
%$\mathsf{-2}$ & \multicolumn{1}{|c}{$\quad0\quad$} & $H^{0}(F_{v},\mathbf
%{Z}/p^{k}(1))$ & $H^{1}(F_{v},\mathbf{Z}/p^{k}(1))$ & $H^{2}(F_{v},\mathbf
%{Z}/p^{k}(1)){\cellcolor{lg}} $ & $\quad0\quad$\\
%$\vdots\,\,$ & \multicolumn{1}{|c}{} & $\vdots$ & $\vdots$ & $\vdots$ &
%\\\cline{2-6}
%& $\mathsf{-1}$ & $\mathsf{0}$ & $\mathsf{1}$ & $\mathsf{2}$ & $\mathsf
%{\cdots}$
%\end{tabular}
%}}}%
%%%%%%%%%%%%%%%%%
\adjustbox{max width=\textwidth}{
\begin{tabular}
[c]{rccccc}$\vdots\,\,$ & \multicolumn{1}{|c}{} & $\vdots$ & $\vdots
$ & $\vdots$ & \\
$\mathsf{2}$ & \multicolumn{1}{|c}{$\quad0\quad$} & $H^{0}(F_{v},\mathbf
{Z}/p^{k}(-1))$ & $H^{1}(F_{v},\mathbf{Z}/p^{k}(-1))$ & $H^{2}(F_{v}%
,\mathbf{Z}/p^{k}(-1))$ & $\quad0\quad$\\
$\mathsf{1}$ & \multicolumn{1}{|c}{$\quad0\quad$} & $0$ & $0$ & $0$ &
$\quad0\quad$\\
$\mathsf{0}$ & \multicolumn{1}{|c}{$\quad0\quad$} & $H^{0}(F_{v},\mathbf
{Z}/p^{k}(0)){\cellcolor{lg}} $ & $H^{1}(F_{v},\mathbf{Z}/p^{k}(0))$ & $H^{2}%
(F_{v},\mathbf{Z}/p^{k}(0))$ &
$\quad0\quad$\\
$\mathsf{-1}$ & \multicolumn{1}{|c}{$\quad0\quad$} & $0$ & $0 {\cellcolor{lg}}
$ & $0$ & $\quad0\quad$\\
$\mathsf{-2}$ & \multicolumn{1}{|c}{$\quad0\quad$} & $H^{0}(F_{v},\mathbf
{Z}/p^{k}(1))$ & $H^{1}(F_{v},\mathbf{Z}/p^{k}(1))$ & $H^{2}(F_{v},\mathbf
{Z}/p^{k}(1)){\cellcolor{lg}} $ & $\quad0\quad$\\
$\vdots\,\,$ & \multicolumn{1}{|c}{} & $\vdots$ & $\vdots$ & $\vdots$ &
\\\cline{2-6}
& $\mathsf{-1}$ & $\mathsf{0}$ & $\mathsf{1}$ & $\mathsf{2}$ & $\mathsf
{\cdots}$
\end{tabular}
}%
%%%%%%%%%%%%%%%%%
\]
and thanks to the shift $\Sigma^{-1}$ we get the map denoted by $e$ in Diagram
\ref{lwit2}. It is compatible to the map $L$ of Eq. \ref{lhibso3} that we had
used from Milne \cite{milne2006} in the proof of Lemma \ref{lem_OO2}. Now use
the identification of $H^{2}(F_{v},\mathbf{G}_{m})$ with Brauer groups to show
that both formulations match up. For local fields at infinite places, $v\in
S_{\infty}$, we had already discussed that $\operatorname*{Br}(F_{v})_{\circ
}=0$, so we do not need to discuss them here.
\end{proof}

\begin{lemma}
\label{lem_OOCube}Suppose $S$ is finite. The cube%
\begin{equation}%
%%%%%%%%%%
%\begin{tikzcd}
%	{\bigoplus\limits_{v\in{S'}}\operatorname{Br}(F_{v})_{\circ}} &&& {\bigoplus
%\limits_{v\in{S'} \setminus S_{\infty}}\pi_{-1}\Sigma^{-1}K(F_{v})[\tau^{-1}]}
%\\
%	& {H_{c}^{3}(\mathcal{O}_{S'},\mathbf{Z}/p^r(1))} &&& {\pi_{-1}\Sigma
%^{-1}K({\mathsf{LC}}_{\mathcal{O}_{S'}})[\tau^{-1}]} \\
%	\\
%	{\bigoplus\limits_{v\in{S}}\operatorname{Br}(F_{v})_{\circ}} &&& {\bigoplus
%\limits_{v\in{S} \setminus S_{\infty}}\pi_{-1}\Sigma^{-1}K(F_{v})[\tau^{-1}]}
%\\
%	& {H_{c}^{3}(\mathcal{O}_{S},\mathbf{Z}/p^r(1))} &&& {\pi_{-1}\Sigma
%^{-1}K({\mathsf{LC}}_{\mathcal{O}_S})[\tau^{-1}]}
%	\arrow["{{M'}}"{description, pos=0.7}, hook, from=1-1, to=1-4]
%	\arrow["{{E'}}"{description}, two heads, from=1-1, to=2-2]
%	\arrow["P", shift left=3, dashed, from=1-1, to=4-1]
%	\arrow["{{\tilde{E}'}}"{description}, from=1-4, to=2-5]
%	\arrow["{{\tilde{P}}}", shift left=3, dashed, from=1-4, to=4-4]
%	\arrow["{{G'}}"{description, pos=0.3}, hook, from=2-2, to=2-5]
%	\arrow["W"{description}, equals, from=2-2, to=5-2]
%	\arrow["B"{description}, from=2-5, to=5-5]
%	\arrow["I", shift left=3, from=4-1, to=1-1]
%	\arrow["M"{description, pos=0.7}, hook, from=4-1, to=4-4]
%	\arrow["E"{description}, two heads, from=4-1, to=5-2]
%	\arrow["{{\tilde{I}}}", shift left=3, from=4-4, to=1-4]
%	\arrow["{{\tilde{E}}}"{description}, from=4-4, to=5-5]
%	\arrow["G"{description, pos=0.3}, hook, from=5-2, to=5-5]
%\end{tikzcd}
%}} }%
%%%%%%%%%%%%%%%%%
\adjustbox{max width=\textwidth}{
\begin{tikzcd}
	{\bigoplus\limits_{v\in{S'}}\operatorname{Br}(F_{v})_{\circ}} &&& {\bigoplus
\limits_{v\in{S'} \setminus S_{\infty}}\pi_{-1}\Sigma^{-1}K(F_{v})[\tau^{-1}]}
\\
	& {H_{c}^{3}(\mathcal{O}_{S'},\mathbf{Z}/p^r(1))} &&& {\pi_{-1}\Sigma
^{-1}K({\mathsf{LC}}_{\mathcal{O}_{S'}})[\tau^{-1}]} \\
	\\
	{\bigoplus\limits_{v\in{S}}\operatorname{Br}(F_{v})_{\circ}} &&& {\bigoplus
\limits_{v\in{S} \setminus S_{\infty}}\pi_{-1}\Sigma^{-1}K(F_{v})[\tau^{-1}]}
\\
	& {H_{c}^{3}(\mathcal{O}_{S},\mathbf{Z}/p^r(1))} &&& {\pi_{-1}\Sigma
^{-1}K({\mathsf{LC}}_{\mathcal{O}_S})[\tau^{-1}]}
	\arrow["{{M'}}"{description, pos=0.7}, hook, from=1-1, to=1-4]
	\arrow["{{E'}}"{description}, two heads, from=1-1, to=2-2]
	\arrow["P", shift left=3, dashed, from=1-1, to=4-1]
	\arrow["{{\tilde{E}'}}"{description}, from=1-4, to=2-5]
	\arrow["{{\tilde{P}}}", shift left=3, dashed, from=1-4, to=4-4]
	\arrow["{{G'}}"{description, pos=0.3}, hook, from=2-2, to=2-5]
	\arrow["W"{description}, equals, from=2-2, to=5-2]
	\arrow["B"{description}, from=2-5, to=5-5]
	\arrow["I", shift left=3, from=4-1, to=1-1]
	\arrow["M"{description, pos=0.7}, hook, from=4-1, to=4-4]
	\arrow["E"{description}, two heads, from=4-1, to=5-2]
	\arrow["{{\tilde{I}}}", shift left=3, from=4-4, to=1-4]
	\arrow["{{\tilde{E}}}"{description}, from=4-4, to=5-5]
	\arrow["G"{description, pos=0.3}, hook, from=5-2, to=5-5]
\end{tikzcd}
}
%%%%%%%%%%%%%%%%%
\label{lhrdtu5}%
\end{equation}
commutes, where the right face is the square of Lemma \ref{lem_OO1} and the
left face the square of Lemma \ref{lem_OO2}. The map $W$ agrees with the map
induced from $J_{!}J^{!}\rightarrow1$ for $J$ as in Eq. \ref{lbipsw1}.
\end{lemma}

\begin{proof}
We first construct the cube step by step: (1) The back face commutes since
$M,M^{\prime}$ are induced from summand-wise maps. The top and bottom face
stem from the compatible map between descent spectral sequences, Lemma
\ref{lem_OO3}. The left face commutes by Lemma \ref{lem_OO2}, and the right
face by Lemma \ref{lem_OO2}. (2) It remains to settle that the front face
commutes. To this end, let $x\in H_{c}^{3}(\mathcal{O}_{S^{\prime}}%
,\mathbf{Z}/p^{r}(1))$ be an arbitrary element. We wish to show that%
\[
BG^{\prime}(x)=GW(x)\text{.}%
\]
We proceed as follows: Read the isomorphism $W$ in Diagram \ref{lhrdtu5} as a
downward map. Consider $Wx\in H_{c}^{3}(\mathcal{O}_{S},\mathbf{Z}/p^{r}(1))$.
Since $E$ is surjective (Lemma \ref{lem_OO2}), we may pick some
$x_{\operatorname*{pre}}\in\bigoplus_{v\in S}\operatorname*{Br}(F_{v})_{\circ
}$ such that%
\[
Wx=E(x_{\operatorname*{pre}})\text{.}%
\]
Now, since $PI=\operatorname*{id}$,%
\begin{equation}
Wx=E(x_{\operatorname*{pre}})=EPI(x_{\operatorname*{pre}})=WE^{\prime
}I(x_{\operatorname*{pre}}) \label{lhrdu4}%
\end{equation}
as the restriction of $E^{\prime}$ to the summands with $v\in S$ agrees with
the map $E$ (i.e., for this special element $x_{\operatorname*{pre}}$ the left
face behaves almost as if the face also commutes with respect to the dashed
arrow $P$). As $W$ is an isomorphism, apply $W^{-1}$ to get%
\[
x=E^{\prime}I(x_{\operatorname*{pre}})\text{.}%
\]
Hence,%
\begin{align*}
BG^{\prime}(x)  &  =BG^{\prime}(E^{\prime}I(x_{\operatorname*{pre}}))\\
&  =B(\tilde{E}^{\prime}M^{\prime})I(x_{\operatorname*{pre}})\text{ (top
face)}\\
&  =B\tilde{E}^{\prime}(\tilde{I}M)(x_{\operatorname*{pre}})\text{ (back
face)}\\
&  =\tilde{E}\tilde{P}(\tilde{I}M)(x_{\operatorname*{pre}})\text{ (right
face)}\\
&  =\tilde{E}(\tilde{P}\tilde{I})M(x_{\operatorname*{pre}})\text{ (}\tilde
{P}\tilde{I}=\operatorname*{id}\text{)}\\
&  =\tilde{E}M(x_{\operatorname*{pre}})\\
&  =GE(x_{\operatorname*{pre}})\text{ (bottom face)}\\
&  =GW(x)\text{ (Eq. \ref{lhrdu4}),}%
\end{align*}
and this finishes the proof.
\end{proof}

\subsection{Duals and lifting the spectral trace morphism}

The trace map of Artin--Verdier duality for the scheme $\mathcal{X}$,%
\[
H_{c}^{2d+3}(\mathcal{X},\mathbf{Z}/p^{r}(d+1))\cong\mathbf{Z}/p^{r}\text{,}%
\]
appears as a map originating from $E_{2}^{2d+3,-(2d+2)}$ in the descent
spectral sequence for $K(\underline{{\mathsf{LC}}}_{\mathcal{X}})[\tau^{-1}]$.
In this section we shall construct a lift of this map to a map of spectra%
\[
K(\underline{{\mathsf{LC}}}_{\mathcal{X}})[\tau^{-1}]\longrightarrow
I_{\mathbf{Z}_{p}}\mathbb{S}_{\widehat{p}}\text{.}%
\]

We write $\mathbb{S}$ for the sphere spectrum. If $A$ denotes an abelian
group, write $\mathbb{S}A$ for the corresponding Moore spectrum. In particular
$\mathbb{S}\mathbf{Z}\cong\mathbb{S}$ is just the sphere spectrum again and
$\mathbb{S}\mathbf{Z}/p\cong\mathbb{S}/p\cong\operatorname*{cofib}%
(\mathbb{S}\overset{\cdot p}{\longrightarrow}\mathbb{S})$ are
equivalences\footnote{and moreover agrees with what other authors would call
the \textquotedblleft$\operatorname{mod}p$ Moore spectrum\textquotedblright,
and perhaps denote by $M_{p}$.}. The Moore spectrum of $\mathbf{Q}%
_{p}/\mathbf{Z}_{p}$ can also be realized through%
\[
\mathbb{S}\mathbf{Q}_{p}/\mathbf{Z}_{p}\cong\operatorname*{colim}%
\mathbb{S}\left(  \left.  \frac{1}{p^{r}}\mathbf{Z}\right/  \mathbf{Z}\right)
\]
under inclusion. Define $\mathcal{N}:=\Sigma^{-1}\mathbb{S}\mathbf{Q}%
_{p}/\mathbf{Z}_{p}$ in $\mathsf{Sp}$. The functors%
\[
X\longmapsto X\otimes_{\mathsf{Sp}}\mathcal{N}\qquad\text{and}\qquad
X\longmapsto\operatorname*{Hom}\nolimits_{\mathsf{Sp}}(\mathcal{N},X)
\]
induce an equivalence of the stable $\infty$-categories of $p$-complete
spectra and $p$-torsion spectra.

We write $\mathsf{K}$ for the $K(1)$-local stable homotopy category for the
prime $p$, considered as a subcategory of local objects inside $\mathsf{Sp}$.
We briefly recall that we shall denote the spectrum which represents the
cohomological functor%
\begin{align*}
\mathsf{K}  &  \longrightarrow\mathsf{Ab}\\
Z  &  \longmapsto\operatorname*{Hom}\nolimits_{\mathsf{Ab}}(\pi_{0}%
(Z\otimes_{\mathsf{Sp}}\mathcal{N}),\mathbf{Q}_{p}/\mathbf{Z}_{p})
\end{align*}
by $I_{\mathbf{Q}_{p}/\mathbf{Z}_{p}}\mathbb{S}_{\widehat{p}}$. Since $Z$ is
$p$-complete, $Z\otimes\mathcal{N}$ is a $p$-torsion spectrum so that
$\operatorname*{Hom}\nolimits_{\mathsf{Ab}}(\pi_{0}(Z\otimes\mathcal{N}%
),\mathbf{Q}_{p}/\mathbf{Z}_{p})$ is a derived $p$-complete abelian group. We
call%
\[
I_{\mathbf{Q}_{p}/\mathbf{Z}_{p}}X:=\operatorname*{Hom}\nolimits_{\mathsf{K}%
}(X,I_{\mathbf{Q}_{p}/\mathbf{Z}_{p}}\mathbb{S}_{\widehat{p}})
\]
the $p$\emph{-complete Brown--Comenetz dual}, following \cite[\S 8]%
{MR2327028}, and evidently this is not introducing an inconsistency into our
notation. Hence, by its very definition, we have a natural isomorphism%
\[
\operatorname*{Hom}\nolimits_{\mathsf{K}}(Z,I_{\mathbf{Q}_{p}/\mathbf{Z}_{p}%
}\mathbb{S}_{\widehat{p}})\cong\operatorname*{Hom}\nolimits_{\mathsf{Ab}}%
(\pi_{0}(Z\otimes\mathcal{N}),\mathbf{Q}_{p}/\mathbf{Z}_{p})\text{.}%
\]

\begin{definition}
[{\cite[\S 2]{MR2946825}, \cite{MR2327028}}]\label{def_padicAndersonDual}We
then call%
\begin{equation}
I_{\mathbf{Z}_{p}}X:=\Sigma^{-1}I_{\mathbf{Q}_{p}/\mathbf{Z}_{p}}X
\label{lch4}%
\end{equation}
the $p$\emph{-adic Anderson dual}\footnote{these names come from the ordinary
non-completed homotopy category, where the latter has a different definition
and where the two duals are connected by a fiber sequence instead of a mere
shift \cite{MR3901158}. For the present setting, the choice which dual to take
is entirely guided by what shift leads to a more typographically pleasant
formulation.}.
\end{definition}

Now Eq. \ref{lh06b} pins down a unique map%
\[
\pi_{0}\left(  K\otimes\left.  \left.  \frac{1}{p^{r}}\mathbb{S}\right/
\mathbb{S}\right.  \right)  ({\mathsf{LC}}_{\mathbf{Z}\left[  \frac{1}%
{p}\right]  })[\tau^{-1}]\longrightarrow\left.  \frac{1}{p^{r}}\mathbf{Z}%
\right/  \mathbf{Z}%
\]
where, if the reader does not like our somewhat unusual notation of a
$p$-power divided sphere, we note that we just mean%
\[
\left.  \left.  \frac{1}{p^{r}}\mathbb{S}\right/  \mathbb{S}\right.
:=\mathbb{S}/p^{r}\mathbb{S}\text{.}%
\]
and we may assemble these maps in an inductive system along increasing $r$ (if
you write it as $\mathbb{S}/p^{r}\mathbb{S}$, this means the map
$\mathbb{S}/p^{r}\mathbb{S}\overset{\cdot p^{s}}{\rightarrow}\mathbb{S}%
/p^{r+s}\mathbb{S}$, but for $\mathbf{Z}$, this just means inclusion $\frac
{1}{p^{r}}\mathbf{Z}\subseteq\frac{1}{p^{r+s}}\mathbf{Z}$). In the colimit, we
obtain%
\[
\pi_{0}\left(  K\otimes\left.  \mathbb{S}\mathbf{Q}_{p}/\mathbf{Z}_{p}\right.
\right)  ({\mathsf{LC}}_{\mathbf{Z}\left[  \frac{1}{p}\right]  })[\tau
^{-1}]\longrightarrow\mathbf{Q}_{p}/\mathbf{Z}_{p}\text{,}%
\]
or%
\[
\pi_{0}\left(  K({\mathsf{LC}}_{\mathbf{Z}\left[  \frac{1}{p}\right]  }%
)[\tau^{-1}]\otimes\Sigma\mathcal{N}\right)  \longrightarrow\mathbf{Q}%
_{p}/\mathbf{Z}_{p}\text{.}%
\]
Recalling our conventions around the $p$-adic Anderson dual, this uniquely
pins down a map%
\[
\Sigma K({\mathsf{LC}}_{\mathbf{Z}\left[  \frac{1}{p}\right]  })[\tau
^{-1}]\longrightarrow I_{\mathbf{Q}_{p}/\mathbf{Z}_{p}}\mathbb{S}_{\widehat
{p}}\text{,}%
\]
or rephrased along Eq. \ref{lch4},%
\begin{equation}
\operatorname*{tr}\nolimits_{\heartsuit}\colon K({\mathsf{LC}}_{\mathbf{Z}%
\left[  \frac{1}{p}\right]  })[\tau^{-1}]\longrightarrow I_{\mathbf{Z}_{p}%
}\mathbb{S}_{\widehat{p}}\text{.} \label{lciub1}%
\end{equation}
Our running assumption that $\frac{1}{p}\in\mathcal{O}_{S}$ implies that there
is an inclusion of rings $\mathcal{O}_{F}\left[  \frac{1}{p}\right]
\subseteq\mathcal{O}_{S}$, which (as $S$ is finite) geometrically becomes an
open immersion.

\begin{theorem}
\label{thmF}Consider the quasi-finite map%
\begin{equation}
u\colon\operatorname*{Spec}\mathcal{O}_{S}\hookrightarrow\operatorname*{Spec}%
\mathcal{O}_{F}\left[  \frac{1}{p}\right]  \overset{b}{\rightarrow
}\operatorname*{Spec}\mathbf{Z}\left[  \frac{1}{p}\right]  \text{.}
\label{lhuv7}%
\end{equation}
Then the co-unit $u_{!}u^{!}\rightarrow1$ induces a commutative diagram%
\[%
%%%%%%%%%%
%\begin{tikzcd}
%	{H^3_c \left( \mathcal{O}_S , \mathbf{Z}/p^r(1) \right)} && {\pi_{-1}%
%\Sigma^{-1}K/p^r(\underline{{\mathsf{LC}}}_{\mathcal{O}_S})[{\tau}^{-1}]} \\
%	\\
%	{H^3_c \left( {\mathbf{Z}\left[ \frac{1}{p} \right]} , \mathbf{Z}%
%/p^r(1) \right)} && {\pi_{-1}\Sigma^{-1}K/p^r(\underline{{\mathsf{LC}}%
%}_{{\mathbf{Z}\left[ \frac{1}{p} \right]}})[{\tau}^{-1}],}
%	\arrow[hook, from=1-1, to=1-3]
%	\arrow["{{{{u_! u^! \rightarrow1}}}}"', from=1-1, to=3-1]
%	\arrow["{b_{*} \circ B}", from=1-3, to=3-3]
%	\arrow["\sim"', from=3-1, to=3-3]
%\end{tikzcd}
%}}}%
%%%%%%%%%%%%%%%%%
{
\begin{tikzcd}
	{H^3_c \left( \mathcal{O}_S , \mathbf{Z}/p^r(1) \right)} && {\pi_{-1}%
\Sigma^{-1}K/p^r(\underline{{\mathsf{LC}}}_{\mathcal{O}_S})[{\tau}^{-1}]} \\
	\\
	{H^3_c \left( {\mathbf{Z}\left[ \frac{1}{p} \right]} , \mathbf{Z}%
/p^r(1) \right)} && {\pi_{-1}\Sigma^{-1}K/p^r(\underline{{\mathsf{LC}}%
}_{{\mathbf{Z}\left[ \frac{1}{p} \right]}})[{\tau}^{-1}],}
	\arrow[hook, from=1-1, to=1-3]
	\arrow["{{{{u_! u^! \rightarrow1}}}}"', from=1-1, to=3-1]
	\arrow["{b_{*} \circ B}", from=1-3, to=3-3]
	\arrow["\sim"', from=3-1, to=3-3]
\end{tikzcd}
}%
%%%%%%%%%%%%%%%%%
\]
where $B$ is the functor which forgets the $\mathcal{O}_{S}$-module structure
on a locally compact abelian group in favour of the $\mathcal{O}_{F}\left[
\frac{1}{p}\right]  $-module structure. The lower horizontal arrow is an isomorphism.
\end{theorem}

\begin{proof}
This is a special case of the commutativity of the front face in Lemma
\ref{lem_OOCube} for the open immersion $\operatorname*{Spec}\mathcal{O}%
_{S}\hookrightarrow\operatorname*{Spec}\mathcal{O}_{F}\left[  \frac{1}%
{p}\right]  $, and a special case of Lemma \ref{lemma_pi_proper1} for the
finite (proper) morphism $b$.
\end{proof}

Denote the composition%
\begin{equation}
\mathcal{X}\overset{\pi}{\longrightarrow}\operatorname*{Spec}\mathcal{O}%
_{S}\overset{u}{\rightarrow}\operatorname*{Spec}\mathbf{Z}[\frac{1}{p}]
\label{lhuv9}%
\end{equation}
by $\rho$.

\begin{definition}
We define the $!$\emph{-pushforward} $\rho_{!}\colon\underline{{\mathsf{LC}}%
}_{\mathcal{X}}\longrightarrow{\mathsf{LC}}_{\mathbf{Z}\left[  \frac{1}%
{p}\right]  }$ as the composition%
\[
\rho_{!}\colon\underline{{\mathsf{LC}}}_{\mathcal{X}}\overset{\pi_{\ast}%
}{\longrightarrow}{\mathsf{LC}}_{\mathcal{O}_{S}}\overset{b_{\ast}\circ
B}{\longrightarrow}{\mathsf{LC}}_{\mathbf{Z}\left[  \frac{1}{p}\right]
}\text{,}%
\]
where $\pi_{\ast}$ is the pushforward of locally compact modules as defined in
Eq. \ref{l_h0a}, and $b_{\ast}\circ B$ is as in Theorem \ref{thmF}.
\end{definition}

Theorem \ref{thmF} justifies calling this a $!$-pushforward. One can show that
this definition does not depend on the choice of the \textquotedblleft base
field\textquotedblright\ $F$ and $\mathcal{O}_{S}$ in Eq. \ref{lhuv9}, but we
omit the details.

\begin{definition}
\label{def_GeneralTraceMap}Suppose $S$ is finite and $\pi$ is proper. We
define%
\[
\operatorname*{tr}\nolimits_{\mathcal{X}}\colon K(\underline{{\mathsf{LC}}%
}_{\mathcal{X}})\overset{\rho_{!}}{\longrightarrow}K({\mathsf{LC}}%
_{\mathbf{Z}\left[  \frac{1}{p}\right]  })\longrightarrow K({\mathsf{LC}%
}_{\mathbf{Z}\left[  \frac{1}{p}\right]  })[\tau^{-1}]\overset
{\operatorname*{tr}\nolimits_{\heartsuit}}{\longrightarrow}I_{\mathbf{Z}_{p}%
}\mathbb{S}\text{.}%
\]

\end{definition}

There is a module structure%
\begin{equation}
\operatorname*{Perf}\mathcal{X}\times\underline{{\mathsf{LC}}}_{\mathcal{X}%
}\longrightarrow\underline{{\mathsf{LC}}}_{\mathcal{X}} \label{lh04}%
\end{equation}
which reduces to the ordinary symmetric monoidal structure on perfect
complexes:%
\[%
%%%%%%%%%%
%\begin{tikzcd}
%	{\operatorname{Perf}{\mathcal{X} } \times{\underline{{\mathsf{LC}}}%
%_{\mathcal{X}}}} &&& {{\underline{{\mathsf{LC}}}_{\mathcal{X}}}} \\
%	{\operatorname{Perf}{\mathcal{X} } \times\left( \operatorname{Perf}%
%{\mathcal{X} } \times\operatorname{D}_{\infty}^{b} {{{\mathsf{LC}}}%
%_{\mathcal{O}_S}} \right)} \\
%	{\left( \operatorname{Perf}{\mathcal{X} } \times\operatorname{Perf}%
%{\mathcal{X} } \right) \times\operatorname{D}_{\infty}^{b} {{{\mathsf{LC}}%
%}_{\mathcal{O}_S}}} &&& { \operatorname{Perf}{\mathcal{X} }   \times
%\operatorname{D}_{\infty}^{b} {{{\mathsf{LC}}}_{\mathcal{O}_S}}}
%	\arrow[from=1-1, to=1-4]
%	\arrow["{1 \times\mathrm{Def.\thinspace} \ref{def_LCA_X}}%
%", equals, from=1-1, to=2-1]
%	\arrow[equals, from=1-4, to=3-4]
%	\arrow[equals, from=2-1, to=3-1]
%	\arrow["{\times_{\mathrm{Perf}}}"', from=3-1, to=3-4]
%\end{tikzcd}
%}}}%
%%%%%%%%%%%%%%%%%
{
\begin{tikzcd}
	{\operatorname{Perf}{\mathcal{X} } \times{\underline{{\mathsf{LC}}}%
_{\mathcal{X}}}} &&& {{\underline{{\mathsf{LC}}}_{\mathcal{X}}}} \\
	{\operatorname{Perf}{\mathcal{X} } \times\left( \operatorname{Perf}%
{\mathcal{X} } \times\operatorname{D}_{\infty}^{b} {{{\mathsf{LC}}}%
_{\mathcal{O}_S}} \right)} \\
	{\left( \operatorname{Perf}{\mathcal{X} } \times\operatorname{Perf}%
{\mathcal{X} } \right) \times\operatorname{D}_{\infty}^{b} {{{\mathsf{LC}}%
}_{\mathcal{O}_S}}} &&& { \operatorname{Perf}{\mathcal{X} }   \times
\operatorname{D}_{\infty}^{b} {{{\mathsf{LC}}}_{\mathcal{O}_S}}}
	\arrow[from=1-1, to=1-4]
	\arrow["{1 \times\mathrm{Def.\thinspace} \ref{def_LCA_X}}%
", equals, from=1-1, to=2-1]
	\arrow[equals, from=1-4, to=3-4]
	\arrow[equals, from=2-1, to=3-1]
	\arrow["{\times_{\mathrm{Perf}}}"', from=3-1, to=3-4]
\end{tikzcd}
}%
%%%%%%%%%%%%%%%%%
\]
and the very definition of $\underline{{\mathsf{LC}}}_{\mathcal{X}}$.\ It
renders $K(\underline{{\mathsf{LC}}}_{\mathcal{X}})$ a $K(\mathcal{X}%
)$-module. As the underlying bi-functor on perfect complexes is compatible
with the filtration underlying the descent spectral sequence, we get a
multiplicative structure%
\[%
\begin{tabular}
[c]{lll}%
$E_{r}^{i,j}(\mathcal{X})\otimes E_{r}^{i^{\prime},j^{\prime}}(\underline
{{\mathsf{LC}}}_{\mathcal{X}})\longrightarrow E_{r}^{i+i^{\prime},j+j^{\prime
}}(\underline{{\mathsf{LC}}}_{\mathcal{X}})$ & $\qquad$ & on the pages\\
$F_{i}\otimes F_{j}\longrightarrow F_{i+j}$ & $\qquad$ & on the filtration
\end{tabular}
\ \ \ \ \ \ \ \
\]
satsifying the Leibniz rule for products on the $r$-th page%
\begin{equation}
d_{r}(x\smile y)=d_{r}x\smile y+(-1)^{i+j}x\smile d_{r}y\text{.} \label{lw4}%
\end{equation}
We recall that a pairing (of spectra or abelian groups) $\mu\colon X\otimes
Y\rightarrow Z$ is called \emph{left perfect} if the Hom-tensor adjunction
$\operatorname*{Hom}(X\otimes Y,Z)\cong\operatorname*{Hom}%
(X,\operatorname*{Hom}(Y,Z))$ sends $\mu$ to an equivalence $X\overset{\sim
}{\rightarrow}\operatorname*{Hom}(Y,Z)$. Correspondingly, the pairing $\mu$ is
called \emph{right perfect} if the reversed pairing $\mu^{\operatorname*{op}%
}\colon Y\otimes X\rightarrow Z$ is a left perfect pairing. A \emph{perfect}
pairing is a pairing that is both left and right perfect.

\begin{theorem}
\label{thmA_inside}Let $p$ be an odd prime, $S$ is finite and $\frac{1}{p}%
\in\mathcal{O}_{S}$. We assume $\pi$ is proper. Suppose $F$ is a number field
without real places and let $S$ be a finite set of finite places such that
$\frac{1}{p}\in\mathcal{O}_{S}$. Then%
\begin{equation}
K(\mathcal{X})[\tau^{-1}]\otimes K(\underline{{\mathsf{LC}}}_{\mathcal{X}%
})[\tau^{-1}]\longrightarrow K(\underline{{\mathsf{LC}}}_{\mathcal{X}}%
)[\tau^{-1}]\overset{\operatorname*{tr}}{\longrightarrow}I_{\mathbf{Z}_{p}%
}\mathbb{S}_{\widehat{p}} \label{lh2}%
\end{equation}
is a perfect pairing, where

\begin{itemize}
\item the first map stems from the module structure set up in Eq. \ref{lh04}, and

\item the map $\operatorname*{tr}$ is the one of Definition
\ref{def_GeneralTraceMap}.
\end{itemize}
\end{theorem}

The proof follows the technique of Blumberg--Mandell \cite{MR4121155}, i.e.,
we bootstrap the perfect pairing from classical Artin--Verdier duality for
arithmetic schemes on the $E_{2}$-page.

\begin{proof}
In this proof, we write

\begin{itemize}
\item $E_{r}^{i^{\prime},j^{\prime}}(\underline{{\mathsf{LC}}}_{\mathcal{X}})$
for the spectral sequence converging to%
\[
\pi_{-\bullet-\bullet}\Sigma^{-1}K/p^{r}(\underline{{\mathsf{LC}}%
}_{\mathcal{O}_{S}})[\tau^{-1}]
\]
that we had considered in Lemma \ref{lem_OOSpecSeq} and,

\item $E_{r}^{i,j}(\mathcal{X})$ for the descent spectral sequence of
$K(\mathcal{X})[\tau^{-1}]$.
\end{itemize}

Both exist, are convergent, and the first arrow in Eq. \ref{lh2} induces a map
between them because Prop. \ref{prop_Equiv1} had shown that instead of
$K/p^{r}(\underline{{\mathsf{LC}}}_{\mathcal{O}_{S}})$ we could also work with
$j_{!}\pi_{\ast}{{\pi}^{\ast}\mathcal{K}}$, so in either case we deal with the
descent spectral sequence of an \'{e}tale hypersheaf, evaluated to get its
global sections on $\mathcal{X}$. The pairing of $K(\mathcal{X})$-module
spectra%
\[
K(\mathcal{X})[\tau^{-1}]\otimes K(\underline{{\mathsf{LC}}}_{\mathcal{X}%
})[\tau^{-1}]\longrightarrow K(\underline{{\mathsf{LC}}}_{\mathcal{X}}%
)[\tau^{-1}]
\]
induces a multiplicative structure on the descent spectral sequence. This
entails having induced pairings on the cycle and coboundary groups,%
\begin{align*}
Z_{r}^{i}(\mathcal{X})\otimes Z_{r}^{i^{\prime}}(\underline{{\mathsf{LC}}%
}_{\mathcal{X}})  &  \longrightarrow Z_{r}^{i+i^{\prime}}(\underline
{{\mathsf{LC}}}_{\mathcal{X}})\text{,}\\
B_{r}^{i}(\mathcal{X})\otimes Z_{r}^{i^{\prime}}(\underline{{\mathsf{LC}}%
}_{\mathcal{X}})  &  \longrightarrow B_{r}^{i+i^{\prime}}(\underline
{{\mathsf{LC}}}_{\mathcal{X}})\text{,}\\
Z_{r}^{i}(\mathcal{X})\otimes B_{r}^{i^{\prime}}(\underline{{\mathsf{LC}}%
}_{\mathcal{X}})  &  \longrightarrow B_{r}^{i+i^{\prime}}(\underline
{{\mathsf{LC}}}_{\mathcal{X}})\text{,}%
\end{align*}
as well as on all pages (including the case $r=\infty$)%
\begin{equation}
\bigoplus_{\substack{i+i^{\prime}=m\\j+j^{\prime}=n}}E_{r}^{i,j}%
(\mathcal{X})\otimes E_{r}^{i^{\prime},j^{\prime}}(\underline{{\mathsf{LC}}%
}_{\mathcal{X}})\longrightarrow E_{r}^{m,n}(\underline{{\mathsf{LC}}%
}_{\mathcal{X}}) \label{lwin5}%
\end{equation}
satisfying the Leibniz rule of Eq. \ref{lw4}. On the $E_{\infty}$-page, the
entries correspond to the associated graded pieces of the respective
filtrations of the homotopy groups. Recall the following general principle: If
the induced pairing is a perfect pairing of abelian groups on the $E_{r}$-page
(for some $r\geq2$), then so is the pairing on the next page, and therefore
also on the $E_{\infty}$-page. We will use this criterion for $r=2$: The trace
map%
\[
K(\underline{{\mathsf{LC}}}_{\mathcal{X}})[\tau^{-1}]\overset
{\operatorname*{tr}}{\longrightarrow}I_{\mathbf{Z}_{p}}\mathbb{S}%
\]
corresponds, by construction, to the evaluation depicted below, originating
from $E_{2}^{2d+3,-(2d+2)}(\underline{{\mathsf{LC}}}_{\mathcal{X}})$ and going
around the full right semi-circle to $H_{c}^{3}(\mathbf{Z}\left[  \frac{1}%
{p}\right]  ,\mathbf{Z}/p^{r}(1))$ on the lower left.%
\[%
%%%%%%%%%%
%\begin{tikzcd}
%	& {E_{2}^{2d+3,-(2d+2)}({\underline{{\mathsf{LC}}}_{\mathcal{X} }})}
%& {E_{\infty}^{2d+3,-(2d+2)}({\underline{{\mathsf{LC}}}_{\mathcal{X} }})}
%& {F_{2d+3}\mathbf{\Pi}_{-1}(\mathcal{X})} \\
%	\\
%	& {H_{c}^{2d+3}\left( \mathcal{X},\mathbf{Z}/p^{r}(d+1)\right)}
%&&& {\mathbf{\Pi}_{-1}(\mathcal{X})} \\
%	{E_{2}^{3,-2}\left({\underline{{\mathsf{LC}}}_{\mathcal{O}_{S} }}\right)}
%&& {E_{\infty}^{3,-2}\left({\underline{{\mathsf{LC}}}_{\mathcal{O}_{S} }%
%}\right)} & {F_{3}\mathbf{\Pi}_{-1}\left({\mathcal{O}_S}\right)} \\
%	\\
%	{H_{c}^{3}\left({\mathcal{O}_S},\mathbf{Z}/p^{r}(1)\right)} & {\mathbf
%{Z}/p^r} &&& {\mathbf{\Pi}_{-1}\left({\mathcal{O}_S}\right)} \\
%	\\
%	& {H_{c}^{3}\left({{\mathbf{Z}\left[ \frac{1}{p} \right]}},\mathbf{Z}%
%/p^{r}(1)\right)} && {\mathbf{\Pi}_{-1}\left({\mathbf{Z}\left[ \frac{1}{p}
%	\right]}\right)}
%	\arrow[two heads, from=1-2, to=1-3]
%	\arrow[equals, from=1-2, to=3-2]
%	\arrow["{{{{{{\pi}_{!} {\pi}^{!}\rightarrow1}}}}}"', from=1-2, to=4-1]
%	\arrow[equals, from=1-3, to=1-4]
%	\arrow[hook, from=1-4, to=3-5]
%	\arrow["{{{{{\pi_{*}}}}}}"', from=1-4, to=4-4]
%	\arrow["{{{{{{\pi}_{!} {\pi}^{!}\rightarrow1}}}}}", from=3-2, to=6-1]
%	\arrow[equals, from=3-2, to=6-2]
%	\arrow["{{{{{\pi_*}}}}}", from=3-5, to=6-5]
%	\arrow[two heads, from=4-1, to=4-3]
%	\arrow[equals, from=4-1, to=6-1]
%	\arrow[equals, from=4-3, to=4-4]
%	\arrow[hook, from=4-4, to=6-5]
%	\arrow[equals, from=6-1, to=8-2]
%	\arrow["B", from=6-5, to=8-4]
%	\arrow["{{{{\operatorname{tr}}}}}", equals, from=8-2, to=6-2]
%	\arrow[equals, from=8-4, to=8-2]
%\end{tikzcd}
%}}}%
%%%%%%%%%%%%%%%%%
\adjustbox{max width=\textwidth}{
\begin{tikzcd}
	& {E_{2}^{2d+3,-(2d+2)}({\underline{{\mathsf{LC}}}_{\mathcal{X} }})}
& {E_{\infty}^{2d+3,-(2d+2)}({\underline{{\mathsf{LC}}}_{\mathcal{X} }})}
& {F_{2d+3}\mathbf{\Pi}_{-1}(\mathcal{X})} \\
	\\
	& {H_{c}^{2d+3}\left( \mathcal{X},\mathbf{Z}/p^{r}(d+1)\right)}
&&& {\mathbf{\Pi}_{-1}(\mathcal{X})} \\
	{E_{2}^{3,-2}\left({\underline{{\mathsf{LC}}}_{\mathcal{O}_{S} }}\right)}
&& {E_{\infty}^{3,-2}\left({\underline{{\mathsf{LC}}}_{\mathcal{O}_{S} }%
}\right)} & {F_{3}\mathbf{\Pi}_{-1}\left({\mathcal{O}_S}\right)} \\
	\\
	{H_{c}^{3}\left({\mathcal{O}_S},\mathbf{Z}/p^{r}(1)\right)} & {\mathbf
{Z}/p^r} &&& {\mathbf{\Pi}_{-1}\left({\mathcal{O}_S}\right)} \\
	\\
	& {H_{c}^{3}\left({{\mathbf{Z}\left[ \frac{1}{p} \right]}},\mathbf{Z}%
/p^{r}(1)\right)} && {\mathbf{\Pi}_{-1}\left({\mathbf{Z}\left[ \frac{1}{p}
	\right]}\right)}
	\arrow[two heads, from=1-2, to=1-3]
	\arrow[equals, from=1-2, to=3-2]
	\arrow["{{{{{{\pi}_{!} {\pi}^{!}\rightarrow1}}}}}"', from=1-2, to=4-1]
	\arrow[equals, from=1-3, to=1-4]
	\arrow[hook, from=1-4, to=3-5]
	\arrow["{{{{{\pi_{*}}}}}}"', from=1-4, to=4-4]
	\arrow["{{{{{{\pi}_{!} {\pi}^{!}\rightarrow1}}}}}", from=3-2, to=6-1]
	\arrow[equals, from=3-2, to=6-2]
	\arrow["{{{{{\pi_*}}}}}", from=3-5, to=6-5]
	\arrow[two heads, from=4-1, to=4-3]
	\arrow[equals, from=4-1, to=6-1]
	\arrow[equals, from=4-3, to=4-4]
	\arrow[hook, from=4-4, to=6-5]
	\arrow[equals, from=6-1, to=8-2]
	\arrow["B", from=6-5, to=8-4]
	\arrow["{{{{\operatorname{tr}}}}}", equals, from=8-2, to=6-2]
	\arrow[equals, from=8-4, to=8-2]
\end{tikzcd}
}%
%%%%%%%%%%%%%%%%%
\]
This map agrees (by definition!) with the map of Definition
\ref{def_GeneralTraceMap}. Going around the diagram on the left side instead,
we see that%
\[
E_{2}^{2d+3,-(2d+2)}(\underline{{\mathsf{LC}}}_{\mathcal{X}})\longrightarrow
E_{2}^{3,-2}(\underline{{\mathsf{LC}}}_{\mathcal{O}_{S}})\longrightarrow
H_{c}^{3}(\mathbf{Z}\left[  \frac{1}{p}\right]  ,\mathbf{Z}/p^{r}%
(1))\cong\mathbf{Z}/p^{r}%
\]
agrees with the usual trace map from Artin--Verdier duality for arithmetic
schemes. Thus, the pairing on the $E_{2}$-page reads%
\[
\underset{\text{\textsf{(A)}}}{E_{2}^{i,j}(\mathcal{X})}\otimes\underset
{\text{\textsf{(B)}}}{E_{2}^{i^{\prime},j^{\prime}}(\underline{{\mathsf{LC}}%
}_{\mathcal{O}_{S}})}\longrightarrow\underset{\text{\textsf{(C)}}}%
{E_{2}^{i+i^{\prime},j+j^{\prime}}(\underline{{\mathsf{LC}}}_{\mathcal{O}_{S}%
})}%
\]
and \textsf{(A)} unravels to be $H^{i}(\mathcal{X},\mathbf{Z}/p^{r}(-\tfrac
{j}{2}))$ thanks to the ordinary descent spectral sequence for $\mathcal{X}$,
\textsf{(B)} unravels to be $H_{c}^{i^{\prime}}(\mathcal{X},\mathbf{Z}%
/p^{r}(-\tfrac{j^{\prime}}{2}))$ by identifying the spectrum with compactly
supported \'{e}tale cohomology (Prop. \ref{prop_Equiv1}) and the descent
spectral sequence, and \textsf{(C)} unravels to be $H_{c}^{i+i^{\prime}%
}(\mathcal{X},\mathbf{Z}/p^{r}(-\tfrac{j+j^{\prime}}{2}))$ by the same reason.
Moreover, the pairing agrees with the cup product on \'{e}tale cohomology as
we can identify $K/p^{r}(\underline{{\mathsf{LC}}}_{\mathcal{O}_{S}}%
)[\tau^{-1}]$ with a fiber just as in Blumberg--Mandell, Cor. \ref{cor2} and
\cite[Thm. 3.1 and 3.2]{MR4121155} (\textit{loc. cit.} also shows that one can
phrase the pairing just as well as a Yoneda $\operatorname*{Ext}$-pairing).
However, classical Artin--Verdier duality guarantees that the cup product
pairing
\[
\underset{\text{\textsf{(A)}}}{H^{s}(\mathcal{X},\mathcal{F})}\otimes
\underset{\text{\textsf{(B)}}}{H_{c}^{2d+3-s}(\mathcal{X},\mathcal{F}^{\ast}%
)}\longrightarrow\underset{\text{\textsf{(C)}}}{H_{c}^{2d+3}(\mathcal{X}%
,\mathbf{Z}/p^{r}(d+1))}\cong\mathbf{Z}/p^{r}%
\]
is a perfect pairing of finite abelian groups (\cite[Ch. II, \S 7, Cor.
7.7]{milne2006} or \cite[Theorem 1.19]{MR1327282}, or the even broader
strength of \cite[Theorem 4.6]{MR3867292}). There is a subtlety hiding here:
All these sources define $H_{c}^{\bullet}$ as to include certain Tate
cohomology summands at the infinite places $S_{\infty}$, but we have never
introduced them. This is fine for the reasons explained in Rmk.
\ref{rmk_CompactlySupportedEtaleCohomology}.\footnote{Had we introduced them,
we would actually be in trouble, as the descent spectral sequence always
outputs ordinary \'{e}tale cohomology and not Tate cohomology.} We deduce that
the pairing of Eq. \ref{lwin5} is perfect on the $E_{\infty}$-page. Since our
filtration of the homotopy groups is finite in each case (since the support of
the spectral sequence is horizontally bounded), this inductively shows that
the pairing of homotopy groups%
\[
\pi_{n}K/p^{r}(\mathcal{X})[\tau^{-1}]\otimes\pi_{-n}K/p^{r}(\underline
{{\mathsf{LC}}}_{\mathcal{X}})[\tau^{-1}]\longrightarrow\mathbf{Z}/p^{r}%
\]
is perfect. For $p$-complete spectra, a map $f\colon X\rightarrow Z$ is an
equivalence iff $f/p\colon X/p\rightarrow Z/p$ is an equivalence, so this is
sufficient to show that the pairing in Eq. \ref{lh2} is perfect.
\end{proof}

We are ready to prove Theorem \ref{introThmA}.

\begin{proof}
[Proof of Theorem \ref{introThmA}]The claim follows from Theorem
\ref{thmA_inside}, and the fact that inverting the Bott elements agrees with
$K(1)$-localization \cite{MR826102}, \cite[\S 4 Appendix, p. 193]{MR764579}.
\end{proof}

Let us briefly explain how classical class field theory for $\mathcal{X}%
=\operatorname*{Spec}\mathcal{O}_{S}$ can be recovered from this.

\begin{definition}
[Artin map]\label{def_ArtinMap}Call the composition of the solid arrows%
\begin{equation}%
%%%%%%%%%%
%\begin{tikzcd}
%	{\pi_{1}K(\underline{{\mathsf{LC}}}_{\mathcal{X}})} && {\underset
%{r}{\underleftarrow{\lim}} \thinspace\pi_{1}^{\operatorname*{ab} }(\mathcal
%{X})/p^{r}} \\
%	{\pi_{1}L_{K(1)}K(\underline{{\mathsf{LC}}}_{\mathcal{X}})} && {\pi
%_{1}I_{\mathbf{Z}_{p}}L_{K(1)}K({\mathcal{X}})}
%	\arrow[dashed, from=1-1, to=1-3]
%	\arrow[from=1-1, to=2-1]
%	\arrow["w", hook, from=1-3, to=2-3]
%	\arrow["{{\operatorname{D}}}", from=2-1, to=2-3]
%\end{tikzcd}
%}} }%
%%%%%%%%%%%%%%%%%
{
\begin{tikzcd}
	{\pi_{1}K(\underline{{\mathsf{LC}}}_{\mathcal{X}})} && {\underset
{r}{\underleftarrow{\lim}} \thinspace\pi_{1}^{\operatorname*{ab} }(\mathcal
{X})/p^{r}} \\
	{\pi_{1}L_{K(1)}K(\underline{{\mathsf{LC}}}_{\mathcal{X}})} && {\pi
_{1}I_{\mathbf{Z}_{p}}L_{K(1)}K({\mathcal{X}})}
	\arrow[dashed, from=1-1, to=1-3]
	\arrow[from=1-1, to=2-1]
	\arrow["w", hook, from=1-3, to=2-3]
	\arrow["{{\operatorname{D}}}", from=2-1, to=2-3]
\end{tikzcd}
}
%%%%%%%%%%%%%%%%%
\label{lDiag7}%
\end{equation}
the $K$\emph{-theoretic Artin map} of the arithmetic scheme $\mathcal{X}$,
where the

\begin{enumerate}
\item left downward arrow is the $K(1)$-localization, the

\item lower horizontal arrow is the duality of Theorem \ref{introThmA},

\item the arrow $w$ is the $p$-adic Anderson dual of the suitable edge map of
the descent spectral sequence for $L_{K(1)}K(\mathcal{X})$.
\end{enumerate}
\end{definition}

The solid arrows make sense for arbitrary $\pi\colon\mathcal{X}\rightarrow
\operatorname*{Spec}\mathcal{O}_{S}$ as in \S \ref{sect_Setup}. But in general
$w$ is just an injection. However, for $\mathcal{X=O}_{S}$, the map $w$
becomes an isomorphism, so the dashed arrow in Diagram \ref{lDiag7} exists,
and one recovers the classical Artin map (except for $p=2$, as we had
generally excluded this prime in our treatment).

\begin{theorem}
\label{thm_W}If $\mathcal{X}=\mathcal{O}_{S}$, the map $w$ is an isomorphism.
The objects and arrows in the diagram in Definition \ref{def_ArtinMap} unravel
to be%
\[
\frac{\underset{v\in S\;}{%
%%%%%%%%%%
%%%%%%%%%%%%%%%%%
{\textstyle\prod^{\prime}}
%%%%%%%%%%%%%%%%%
}F_{v}^{\times}}{F^{\times}\cdot\prod_{w\notin S}\mathcal{O}_{w}^{\times}%
}\cong\pi_{1}K({\mathsf{LCA}}_{\mathcal{O}_{S}})\longrightarrow\pi_{1}%
L_{K(1)}K({\mathsf{LC}}_{\mathcal{O}_{S}})\overset{\operatorname*{D}%
}{\longrightarrow}\widehat{\pi_{1}^{\operatorname*{ab}}(\mathcal{O}_{S}%
)}^{(p)}%
\]
and this map agrees with the pro-$p$ completed output of the ordinary Artin
reciprocity map with ramification restricted to $S$. This means that each
$x\in F_{v}^{\times}$ is sent to the local reciprocity map of the local field
$F_{v}$,%
\[
F_{v}^{\times}\longrightarrow\operatorname*{Gal}(F_{v}^{\operatorname*{ab}%
}/F_{v})\text{,}%
\]
and upon identifying the latter Galois group with a decomposition group over
the place $v$, it embeds (canonically) into the abelianized \'{e}tale
fundamental group.
\end{theorem}

The first three arrows in Definition \ref{def_ArtinMap} are just the
specialization of maps of spectra to $\pi_{1}$. There is no need to restrict
to $\pi_{1}$, and then the map%
\[
K({\mathsf{LC}}_{\mathcal{X}})\longrightarrow L_{K(1)}K({\mathsf{LC}%
}_{\mathcal{X}})\overset{\operatorname*{D}}{\longrightarrow}I_{\mathbf{Z}_{p}%
}L_{K(1)}K(\mathcal{X})\text{,}%
\]
specialized to $\mathcal{X=O}_{S}$ and precomposed with $\mathsf{LCA}%
_{\mathcal{O}_{S}}\rightarrow\mathsf{LC}_{\mathcal{O}_{S}}$, corresponds to
the adjoint of Clausen's Artin map of \cite{clausen}\footnote{Loc. cit. this
is phrased in terms of\textit{ Selmer }$K$\textit{-homology}, which is a
(slightly different from Anderson) dual to Selmer $K$-theory
(\cite{clausen,MR4110725,MR4296353}), but as we work exclusively in
characteristic zero and with $p$ odd, all of the subtle differences
evaporate.}. For higher-dimensional $\mathcal{X}$ this morphism does
\textit{not} readily output the reciprocity map of higher-dimensional class
field theory on $\pi_{1}$, or any other $\pi_{n}$.\footnote{For the reason
that the correct objects are found in motivic cohomology and not $K$-theory,
so they are hidden in subquotients of the weight filtration. For
$\mathcal{X=O}_{S}$ the filtration is short enough to conceal this mismatch.}

\begin{proof}
First, we unravel the non-$K(1)$-local left side. Define the outside
$S$-unramified id\`{e}le group $J_{F,S}:=\left.
%%%%%%%%%%
%%%%%%%%%%%%%%%%%
{\textstyle\prod\nolimits_{v\in S}}
%%%%%%%%%%%%%%%%%
F_{v}^{\times}\right/  \mathcal{O}_{S}^{\times}$ with $\mathcal{O}_{S}%
^{\times}$ being embedded diagonally, as in \cite[Ch. I, \S 4]{milne2006} or
\cite[\S 15.5]{MR4174395}. Consider the long exact sequence of \cite[Theorem
4.30]{klca1}, but this time without going $K(1)$-local. We find the
commutative diagram%
\[%
%%%%%%%%%%
%\begin{tikzcd}
%	&& {\frac{\underset{v\in S\;}{\prod} F_{v}^{\times}}{ \mathcal{O}_{S}%
%^{\times}}} && {\frac{\underset{v\in S\;}{\prod^{\prime} } F_{v}^{\times}%
%}{F^{\times} \cdot\underset{w\notin S\;}{\prod}\mathcal{O}_{w}^{\times}}}
%&& {\operatorname{Cl}(\mathcal{O}_{S})} \\
%	\\
%	{K_{1}(\mathcal{O}_{S})} && {\underset{v\in S\;}{\prod}F_{v}^{\times}}
%&& {K_{1}(\mathsf{LCA}_{\mathcal{O}_{S}})} && {K_{0}(\mathcal{O}_{S})}
%&& {\underset{v\in S\;}{\prod}\mathbf{Z},}
%	\arrow[hook, from=1-3, to=1-5]
%	\arrow[two heads, from=1-5, to=1-7]
%	\arrow[from=3-1, to=3-3]
%	\arrow[from=3-3, to=1-3]
%	\arrow[from=3-3, to=3-5]
%	\arrow[from=3-5, to=1-5]
%	\arrow[from=3-5, to=3-7]
%	\arrow[two heads, from=3-7, to=1-7]
%	\arrow["\gamma", from=3-7, to=3-9]
%\end{tikzcd}
%}}}%
%%%%%%%%%%%%%%%%%
\adjustbox{max width=\textwidth}{
\begin{tikzcd}
	&& {\frac{\underset{v\in S\;}{\prod} F_{v}^{\times}}{ \mathcal{O}_{S}%
^{\times}}} && {\frac{\underset{v\in S\;}{\prod^{\prime} } F_{v}^{\times}%
}{F^{\times} \cdot\underset{w\notin S\;}{\prod}\mathcal{O}_{w}^{\times}}}
&& {\operatorname{Cl}(\mathcal{O}_{S})} \\
	\\
	{K_{1}(\mathcal{O}_{S})} && {\underset{v\in S\;}{\prod}F_{v}^{\times}}
&& {K_{1}(\mathsf{LCA}_{\mathcal{O}_{S}})} && {K_{0}(\mathcal{O}_{S})}
&& {\underset{v\in S\;}{\prod}\mathbf{Z},}
	\arrow[hook, from=1-3, to=1-5]
	\arrow[two heads, from=1-5, to=1-7]
	\arrow[from=3-1, to=3-3]
	\arrow[from=3-3, to=1-3]
	\arrow[from=3-3, to=3-5]
	\arrow[from=3-5, to=1-5]
	\arrow[from=3-5, to=3-7]
	\arrow[two heads, from=3-7, to=1-7]
	\arrow["\gamma", from=3-7, to=3-9]
\end{tikzcd}
}%
%%%%%%%%%%%%%%%%%
\]
where the upper row identifies with the exact sequence of \cite[Ch. I, Lemma
4.1]{milne2006} or \cite[Lemma 15.39]{MR4174395}. Since $K_{0}(\mathcal{O}%
_{S})\cong\mathbf{Z}\oplus\operatorname*{Cl}(\mathcal{O}_{S})$ and $\gamma
\mid_{\mathbf{Z}}$ is clearly injective on the rank summand, $\gamma
\mid_{\operatorname*{Cl}(\mathcal{O}_{S})}$ must be the zero map (it maps a
finite group to a torsion-free group). Moreover, $K_{1}(\mathcal{O}_{S}%
)\cong\mathcal{O}_{S}^{\times}$, so we obtain the isomorphism%
\[
K_{1}(\mathsf{LCA}_{\mathcal{O}_{S}})\cong\frac{\underset{v\in S\;}{%
%%%%%%%%%%
%%%%%%%%%%%%%%%%%
{\textstyle\prod^{\prime}}
%%%%%%%%%%%%%%%%%
}F_{v}^{\times}}{F^{\times}\cdot\prod_{w\notin S}\mathcal{O}_{w}^{\times}%
}\text{.}%
\]
In order to unravel $I_{\mathbf{Z}_{p}}L_{K(1)}K(\mathcal{X})$, first use the
descent spectral sequence for $L_{K(1)}K(\mathcal{X})$, we leave the details
to the reader, and then we may access the dual via%
\begin{align*}
\pi_{n}I_{\mathbf{Z}_{p}}X  &  \cong\pi_{n}(\Sigma^{-1}I_{\mathbf{Q}%
_{p}/\mathbf{Z}_{p}}X)\\
&  \cong\pi_{0}\operatorname*{Hom}\nolimits_{\mathsf{Sp}}(\Sigma^{n}%
\mathbb{S}_{\widehat{p}},\Sigma^{-1}I_{\mathbf{Q}_{p}/\mathbf{Z}_{p}}X)\\
&  \cong\pi_{0}\operatorname*{Hom}\nolimits_{\mathsf{Sp}}(\Sigma
^{n+1}\mathbb{S}_{\widehat{p}},I_{\mathbf{Q}_{p}/\mathbf{Z}_{p}}X)\\
&  \cong\pi_{0}\operatorname*{Hom}\nolimits_{\mathsf{Sp}}(\Sigma
^{n+1}\mathbb{S}_{\widehat{p}},\operatorname*{Hom}\nolimits_{\mathsf{Sp}%
}(X,I_{\mathbf{Q}_{p}/\mathbf{Z}_{p}}\mathbb{S}_{\widehat{p}}))\\
&  \cong\pi_{0}\operatorname*{Hom}\nolimits_{\mathsf{Sp}}(\Sigma
^{n+1}\mathbb{S}_{\widehat{p}}\otimes_{\mathsf{K}}X,I_{\mathbf{Q}%
_{p}/\mathbf{Z}_{p}}\mathbb{S}_{\widehat{p}})\text{\qquad(}\sharp\text{)}\\
&  \cong\pi_{0}\operatorname*{Hom}\nolimits_{\mathsf{Sp}}(\Sigma
^{n+1}X,I_{\mathbf{Q}_{p}/\mathbf{Z}_{p}}\mathbb{S}_{\widehat{p}})\\
&  \cong\operatorname*{Hom}\nolimits_{\mathsf{Ab}}(\pi_{0}\left(  \Sigma
^{n+1}X\otimes_{\mathsf{Sp}}\mathcal{N}\right)  ,\mathbf{Q}_{p}/\mathbf{Z}%
_{p})\\
&  \cong\pi_{0}\left(  \Sigma^{n+1}X\otimes_{\mathsf{Sp}}\Sigma^{-1}%
\mathbb{S}\mathbf{Q}_{p}/\mathbf{Z}_{p}\right)  ^{\ast}\\
&  \cong\pi_{0}\left(  \Sigma^{n}X\otimes_{\mathsf{Sp}}\mathbb{S}%
\mathbf{Q}_{p}/\mathbf{Z}_{p}\right)  ^{\ast}\\
&  \cong\pi_{-n}\left(  X\otimes_{\mathsf{Sp}}\mathbb{S}\mathbf{Q}%
_{p}/\mathbf{Z}_{p}\right)  ^{\ast}\text{,}%
\end{align*}
where $\operatorname*{Hom}\nolimits_{\mathsf{Sp}}$ denotes the intrinsic
mapping spectrum and ($\sharp$): $\operatorname*{Hom}\nolimits_{\mathsf{Sp}%
}(X,I_{\mathbf{Q}_{p}/\mathbf{Z}_{p}}\mathbb{S}_{\widehat{p}})\cong
\operatorname*{Hom}\nolimits_{\mathsf{K}}(X,I_{\mathbf{Q}_{p}/\mathbf{Z}_{p}%
}\mathbb{S}_{\widehat{p}})$ as the arguments are $K(1)$-local, and we use the
Hom-tensor adjunction intrinsic to $\mathsf{K}$ (for this $\mathbb{S}%
_{\widehat{p}}$ is the tensor unit). Hence,%
\[
\pi_{1}\left(  I_{\mathbf{Z}_{p}}K(\mathcal{X})[\tau^{-1}]\right)  \cong
\pi_{-1}\left(  K(\mathcal{X})[\tau^{-1}]\otimes_{\mathsf{Sp}}\mathbb{S}%
\mathbf{Q}_{p}/\mathbf{Z}_{p}\right)  ^{\ast}\longrightarrow E_{\infty}%
^{1,0}(\mathbf{Q}_{p}/\mathbf{Z}_{p})^{\ast}%
\]
If $\mathcal{X}=\mathcal{O}_{S}$ so that the relative dimension $d$ is zero
and the spectral sequence supported at worst in columns $[0,3]$, the
filtration trivializes to one single step, i.e., in this special case%
\[
\pi_{1}\left(  I_{\mathbf{Z}_{p}}K(\mathcal{O}_{S})[\tau^{-1}]\right)  \cong
E_{\infty}^{1,0}(\mathbf{Q}_{p}/\mathbf{Z}_{p})^{\ast}\cong E_{2}%
^{1,0}(\mathbf{Q}_{p}/\mathbf{Z}_{p}))^{\ast}\cong\underset{r}{\underleftarrow
{\lim}}\,\pi_{1}^{\operatorname*{ab}}(\mathcal{O}_{S})/p^{r}\text{,}%
\]
the $p$-power completion of the abelianized \'{e}tale fundamental group. As
$S$ was assumed to contain only finitely many places, $\pi_{1}%
^{\operatorname*{ab}}(\mathcal{O}_{S})$ is topologically finitely generated
\cite[Ch. X, \S 2]{MR2392026}, and therefore%
\[
\pi_{1}\left(  I_{\mathbf{Z}_{p}}K(\mathcal{O}_{S})[\tau^{-1}]\right)
\cong\widehat{\pi_{1}^{\operatorname*{ab}}(\mathcal{O}_{S})}^{(p)}\text{,}%
\]
the pro-$p$ completion. The identification of the map with the Artin map then
follows from the identification of the underlying map from Artin--Verdier
duality with the reciprocity map, as explained for example in \cite[\S 0,
0.1-0.3]{MR1045856}.
\end{proof}

\section{Definition via co-descent}

In this section a fixed scheme $X$, integral separated and of finite type over
$\operatorname*{Spec}F$, will be the main object of interest. Occasionally, we
shall pick an integral model $\pi\colon\mathcal{X}\rightarrow
\operatorname*{Spec}\mathcal{O}_{S}$ such that $X\cong\mathcal{X}%
\times_{\mathcal{O}_{S}}\operatorname*{Spec}F$. In this section, we will often
be able to get by with less restrictive assumptions than in
\S \ref{sect_Setup}, so we will always state precisely what is needed for what.

\subsection{Definition}

Let $R$ be a Noetherian commutative ring. We use Definition \ref{def_LCA_OS}
to set up the quasi-abelian exact category $\mathsf{LCA}_{R}$. Suppose $f\in
R$ is a nonzerodivisor, then we have the localization%
\[
R\longrightarrow R\left[  \frac{1}{f}\right]
\]
and a corresponding corestriction exact functor%
\begin{equation}
\mathsf{LCA}_{R[f^{-1}]}\longrightarrow\mathsf{LCA}_{R}\text{,} \label{lcca1}%
\end{equation}
sending $X\mapsto X$, merely forgetting the $R[f^{-1}]$-module structure in
favour of the $R$-module structure. Just as any exact functor would, this
determines a unique map on $K$-theory spectra%
\[
K(\mathsf{LCA}_{R[f^{-1}]})\longrightarrow K(\mathsf{LCA}_{R})
\]
and we will call it the \emph{corestriction map}. Recall that a co-sheaf
$\mathcal{G}$ with values in $\mathsf{Sp}$ is a sheaf with values in
$\mathsf{Sp}^{\operatorname*{op}}$. By a standard construction, a cosheaf for
the Zariski topology can be constructed by defining it on a basis of the
topology. We may take the distinguished affine opens in a scheme.

\begin{definition}
\label{def_lcax}For any affine open $U\subseteq X$ define $K{\mathsf{LC}%
}_{X,\operatorname*{pre}}^{\operatorname*{naive}}(U)$ to be%
\[
K{\mathsf{LC}}_{X,\operatorname*{pre}}^{\operatorname*{naive}}(U):=K\left(
\mathsf{LCA}_{\mathcal{O}_{X}(U)}\right)  \text{.}%
\]
Then define $K{\mathsf{LC}}_{X}^{\operatorname*{naive}}(U)$ for distinguished
affine opens $U$ by Zariski co-descent from $K{\mathsf{LC}}%
_{X,\operatorname*{pre}}^{\operatorname*{naive}}(U)$ on distinguished affine
opens $U$ and along the co-restriction functors of Eq. \ref{lcca1}.
\end{definition}

The existence of the\ Zariski co-sheaf $K{\mathsf{LC}}_{X}%
^{\operatorname*{naive}}$ follows from the generalities of co-sheafification.
Note that there is no reason why $K{\mathsf{LC}}_{X}^{\operatorname*{naive}%
}(U)=K{\mathsf{LC}}_{X,\operatorname*{pre}}^{\operatorname*{naive}}(U)$ would
have to hold, even if $U$ is a distinguished affine open.\footnote{Think of
the example that once we impose Zariski descent to the prestack of finitely
generated \textit{free} $\mathcal{O}_{X}$-modules, we end up with
\textit{locally} \textit{free} $\mathcal{O}_{X}$-modules, so even on an affine
open the sections need not agree with the original presheaf. We expect a
similar behaviour here, but with respect to local compactness instead of
freeness. However, it would clearly be linguistically unhealthy to call this
\textit{\textquotedblleft local local compactness\textquotedblright.}}

In the next section we will isolate a special type of affine open, where the
computation of $K{\mathsf{LC}}_{X,\operatorname*{pre}}^{\operatorname*{naive}%
}(U)$ is easy.

\subsection{Rings spanned by units}

\begin{definition}
\label{def_RingSpannedByUnits}A ring $R$ is \emph{spanned by units} if for
every $x\in R$ there exists a (possibly empty) finite list of units
$u_{1},\ldots,u_{r}\in R^{\times}$ such that $x=u_{1}+\cdots+u_{r}$.
\end{definition}

The remark about a possibly empty finite list of units is not really
necessary. One may always write $0=1+(-1)$. Note that this even works in the
zero ring because then $0$ is a unit itself.

\begin{example}
The integers $\mathbf{Z}$ are spanned by units. A polynomial ring $R[T]$ over
a domain $R$ is never spanned by units since $R[T]^{\times}=R^{\times}$.
\end{example}

\begin{remark}
The property in Definition \ref{def_RingSpannedByUnits} has been studied in
the literature under various names. There has been quite some work towards
finding examples, e.g., \cite{MR2675727,MR1645560,MR2827171,MR2881334}, in
particular Henriksen's result \cite{MR349745} that every ring is Morita
equivalent to a (non-commutative) one which is spanned by units. However,
these investigations largely focus on rings which are quite remote from
coordinate rings of affine varieties.
\end{remark}

We begin with obvious facts.

\begin{lemma}
Let $F$ be a field and $R$ an $F$-algebra. Then $R$ is spanned by units if and
only if every $x\in R$ is a (possibly empty) finite $F$-linear combination of
units in $R$, i.e., $x=\sum a_{i}u_{i}$ with $a_{i}\in F$ and $u_{i}\in
R^{\times}$.
\end{lemma}

\begin{lemma}
\label{lemma_OneCharacterizationOfBeingSpannedByUnits}A ring $R$ is spanned by
units if and only if there exist surjective ring homomorphisms $\mathbf{Z}%
[T_{1}^{\pm1},\ldots,T_{c(i)}^{\pm1}]\twoheadrightarrow R_{i}$ for a poset of
subrings $(R_{i})_{i\in I}$ and integers $c(i)\geq1$ ($I$ any indexing poset)
such that $R=\bigcup R_{i}$.
\end{lemma}

\begin{proof}
The image of each $T_{j}^{\pm}$ must be a unit, so we can expand $\sum
n_{t}x_{t}$ with $n_{t}\in\mathbf{Z}$ and $x_{t}\in R^{\times}$ by writing
$n_{t}=1+1+\ldots+1$ and replace $x_{t}\in R^{\times}$ by $-x_{t}\in
R^{\times}$ if $n_{t}<0$. Conversely, let $I$ be the family of all finite
subsets $\{x_{1},\ldots,x_{\ell}\}$ of $R$, let $u_{1},\ldots,u_{r}$ be enough
units to write each $x_{j}$ as a sub-sum of them, then take $c(i)=r$ and send
$T_{i}$ to $u_{i}$, and $R_{i}$ the subring generated by the image.
\end{proof}

\begin{lemma}
\label{lem_propertiespreservedforringsspannedbyunits}Suppose $R$ is a domain
spanned by units.

\begin{enumerate}
\item If $S\subseteq R$ is a multiplicatively closed subset, then the
localization $R_{S}$ is spanned by units.

\item If $I\subseteq R$ is an ideal, then $R/I$ is spanned by units.
\end{enumerate}
\end{lemma}

\begin{proof}
(1) Every $x\in R_{S}$ can be written as $\frac{r}{s}$ with $r\in R$ and $s\in
S$, and writing $r=\sum u_{i}$ with $u_{i}\in R^{\times}$, we get $\frac{r}%
{s}=\sum\frac{u_{i}}{s}$, but $u_{i}/s\in R_{S}^{\times}$. (2) Follows
directly from Lemma \ref{lemma_OneCharacterizationOfBeingSpannedByUnits}.
Alternatively: Write a preimage of $\overline{r}\in R/I$ as $r=\sum u_{i}$ and
use that the image of a unit in $R/I$ must be a unit again.
\end{proof}

\begin{lemma}
\label{lemma_refine_to_be_spanned_by_units}Let $F$ be a field and $X/F$ a
finite type integral affine scheme, and $x\in X$ a closed point. Then there
exists an affine dense open neighbourhood $U\subseteq X$ of $x$ such that
$\operatorname*{Spec}\mathcal{O}_{X}(U)$ is spanned by units.
\end{lemma}

\begin{proof}
Write $X=\operatorname*{Spec}R$ and let $\mathfrak{m}\subset R$ be the maximal
ideal corresponding to the closed point $x$. Then $R$ is a finitely generated
$F$-algebra, so there exists some $n\geq0$ and a surjective ring homomorphism
$\varphi\colon F[T_{1},\ldots,T_{n}]\twoheadrightarrow R$. We may assume
without loss of generality that each $T_{i}$ maps under the composition%
\[
F[T_{1},\ldots,T_{n}]\overset{\varphi}{\twoheadrightarrow}R\twoheadrightarrow
R/\mathfrak{m}%
\]
to a non-zero element (otherwise: $F$-algebra maps from a polynomial ring are
uniquely determined by choosing the value of each variable $T_{i}$; now change
$\varphi$ such that it maps $T_{i}$ to $\varphi(T_{i})+1$ instead. This cannot
also be zero since in each field, in particularly in $R/\mathfrak{m}$, we have
$0\neq1$. Also, $\varphi$ is still surjective onto $R$). As $R$ is a domain,
the images $\varphi(T_{i})$ in $R$ are nonzerodivisors, so the principal
(a.k.a. distinguished) opens $D(\varphi(T_{i}))\subseteq X$ each are dense.
Hence, the finite intersection of them, $U:=\bigcap_{i}D(\varphi(T_{i}))$, is
also a dense open, and also affine. By the universal property of localization,
$\varphi$ now factors through $F[T_{1}^{\pm1},\ldots,T_{n}^{\pm1}%
]\twoheadrightarrow\mathcal{O}_{X}(U)$ (every fraction $\frac{r}{s}$ now has
both $s$ a unit and $r$ a sum of units, as in the proof of Lemma
\ref{lem_propertiespreservedforringsspannedbyunits}). Thus, $\mathcal{O}%
_{X}(U)$ is spanned by units. We have $x\in U$ by construction.
\end{proof}

\begin{lemma}
\label{lem3}Let $F$ be a field and $X/F$ a finite type separated integral
scheme. Then $X$ admits a finite open cover $X=\bigcup U_{i}$ such that each
$U_{i}$ is affine and $\mathcal{O}_{X}(U_{i})$ is spanned by units.
\end{lemma}

\begin{proof}
First assume $X$ is affine. Pick an open neighbourhood for each closed point
in $X$ as obtained through Lemma \ref{lemma_refine_to_be_spanned_by_units}. By
quasi-compactness of $X$, finitely many of these suffice. If $X$ is not
affine, by the same argument, it can be covered by finitely many affines to
which we may apply the previous refinement argument.
\end{proof}

Let $F$ be a number field and $R$ a finitely generated $F$-algebra. The
forgetful functor%
\begin{equation}
\mathsf{LCA}_{R}\longrightarrow\mathsf{LCA}_{F} \label{lcof1}%
\end{equation}
which forgets the $R$-module structure in favour of the $F$-module structure,
is faithful, exact and reflects exactness.

We call objects in $\mathsf{LCA}_{R}$ \emph{adelic}, \emph{quasi-adelic}, etc.
if they have these properties when we regard them as an object of
$\mathsf{LCA}_{F}$, cf. \cite{kthyartin}.

\begin{definition}
\label{def_PropF}We say that a finitely generated $F$-algebra $R$ has
\emph{Property }$(\mathbf{F})$ if the following properties all hold:

\begin{enumerate}
\item Every object in $\mathsf{LCA}_{R}$ has a canonical $3$-step ascending
admissible\footnote{A filtration is \emph{admissible} if each filtered step
$\mathcal{F}_{i}X\hookrightarrow\mathcal{F}_{j}X$ (for $i\leq j$) is an
admissible monic. Assuming this, the associated graded exists.} filtration%
\[
0=\mathcal{F}_{-1}X\hookrightarrow\mathcal{F}_{0}X\hookrightarrow
\mathcal{F}_{1}X\hookrightarrow\mathcal{F}_{2}X=X
\]
with associated graded pieces $K:=\operatorname{gr}_{\mathcal{F}}(X)_{0}$
compact, $A:=\operatorname{gr}_{\mathcal{F}}(X)_{1}$ adelic, and
$D:=\operatorname{gr}_{\mathcal{F}}(X)_{2}$ discrete.

\item All morphisms $f\colon X\rightarrow Y$ in $\mathsf{LCA}_{R}$ respect the
filtration of $(1)$ in the sense that $f(\mathcal{F}_{i}X)\subseteq
\mathcal{F}_{i}Y$.

\item Each step of the filtration splits (not necessarily canonically), so
that each object admits a (possibly non-canonical) direct sum decomposition%
\begin{equation}
X\simeq K\oplus A\oplus D \label{lwwmix1}%
\end{equation}
with $K,A,D\in\mathsf{LCA}_{R}$ as in (1).
\end{enumerate}
\end{definition}

\begin{remark}
\label{rmk_MaxQASubobjectIsF1}Suppose Property $(\mathbf{F})$ holds. Then
$\mathcal{F}_{1}X$ is the (unique) maximal quasi-adelic subobject of $X$. The
filtration is modelled after a (less restrictive) filtration in the work of
Hoffmann and Spitzweck \cite[Prop. 2.2]{MR2329311}.
\end{remark}

\begin{lemma}
\label{lemma_LiftDecompToAffineInTorus}Let $F$ be a number field. Suppose $R$
is a finitely generated $F$-algebra which is spanned by units. Then $R$ has
Property $(\mathbf{F})$.
\end{lemma}

\begin{proof}
As $X\in\mathsf{LCA}_{R}$, we may in particular treat it as an object of
$\mathsf{LCA}_{F}$ via the functor of Eq. \ref{lcof1}. According to \cite[Thm.
3.10]{kthyartin} there is a canonical $3$-step filtration on $X$ which
(non-canonically) splits $X$ as a direct sum having the properties as in our
claim%
\begin{equation}
X\simeq K\oplus A\oplus D\text{,} \label{lttj1}%
\end{equation}
but only in the category $\mathsf{LCA}_{F}$, i.e., without taking the
$R$-module structure into account. As $X\in\mathsf{LCA}_{R}$, there is an
additional action (through continuous $F$-linear maps) by all elements of
$r\in R$. If $r\in R^{\times}$ is a unit, this action must be by a
homeomorphism, it preserves \textit{(a)} the canonical filtration and
\textit{(b)} the direct sum decomposition in Eq. \ref{lttj1}. Next, note that
if the action by $r_{1},r_{2}\in R$ has these properties \textit{(a), (b)},
then so does their sum $r_{1}+r_{2}$. Hence, since $R$ is spanned by units,
\textit{any} $r\in R$ satisfies \textit{(a), (b)}. We deduce that the action
of any $r$ splits into separate actions $r\mid_{Y}\colon Y\rightarrow Y$ for
each $Y\in\{K,A,D\}$, and thus Eq. \ref{lttj1} can be promoted to a filtration
and direct sum splitting in the category $\mathsf{LCA}_{R}$. Our claim follows.
\end{proof}

\begin{example}
The $n$-torus $(\mathbf{G}_{m})^{n}:=\operatorname*{Spec}F[T_{1}^{\pm}%
,\ldots,T_{n}^{\pm}]$ over a number field is the affine scheme attached to a
ring spanned by units, namely $F^{\times}\cup\{T_{1},\ldots,T_{n}\}$ and
therefore the $F$-algebra $F[T_{1}^{\pm},\ldots,T_{n}^{\pm}]$ satisfies
Property $\left(  \mathbf{F}\right)  $.
\end{example}

\begin{example}
The rings $\mathbf{Q}[T]$ and $\mathbf{Q}[T]/(T^{2})$ do \emph{not} satisfy
Property $(\mathbf{F})$. We can take $X:=\mathbf{Q}\oplus\mathbf{R}$ as an
object in $\mathsf{LCA}_{\mathbf{Q}}$ and define $T(a,b):=(0,a)$. As
$\mathbf{Q}$ is discrete, $T$ is a continuous nilpotent endomorphism of $X$
and renders $X$ an object of $\mathsf{LCA}_{\mathbf{Q}[T]}$ (resp.
$\mathsf{LCA}_{\mathbf{Q}[T]/(T^{2})}$). The only direct sum splitting
possibly satisfying the requirements needed for Property $(\mathbf{F})$ is
$X\simeq\mathbf{Q}\oplus\mathbf{R}$, but this splitting in $\mathsf{LCA}%
_{\mathbf{Q}}$ cannot be promoted to $\mathsf{LCA}_{R}$ for either ring.
\end{example}

\subsection{A vanishing result for $\operatorname*{Ext}^{1}$-classes}

We now develop some homological tools in order to compute $K{\mathsf{LC}%
}_{X,\operatorname*{pre}}^{\operatorname*{naive}}(U)$ on affine opens such
that $\mathcal{O}_{X}(U)$ has Property $(\mathbf{F})$.

\subsubsection{Support and injective hulls\label{section_ReminderOnSupport}}

As $R$ is Noetherian, the finitely generated $R$-modules $\mathsf{Mod}_{R,fg}$
form an abelian category. We recall that if $M^{\prime}\hookrightarrow
M\twoheadrightarrow M^{\prime\prime}$ is an exact sequence of (not necessarily
finitely generated) $R$-modules, then%
\begin{equation}
\operatorname*{supp}M=\operatorname*{supp}M^{\prime}\cup\operatorname*{supp}%
M^{\prime\prime}\text{.} \label{lhh1}%
\end{equation}
As before, $\mathsf{Mod}_{R}$ is the abelian category of all $R$-modules.
Write $\mathsf{Mod}_{R,0}\subset\mathsf{Mod}_{R,fg}$ for the category of
finitely generated $R$-modules of at most zero-dimensional support. As seen
from Eq. \ref{lhh1}, this is a Serre subcategory in the abelian category
$\mathsf{Mod}_{R,fg}$, and thus abelian itself. Moreover, if $E_{R}(M)$
denotes an injective hull of $M$, then $E_{R_{S}}(M_{S})\cong E_{R}(M)_{S}$
holds for any multiplicatively closed subset $S\subset R$, \cite[Lemma
3.2.5]{MR1251956}. Thus,%
\begin{equation}
\operatorname*{supp}E_{R}(M)\subseteq\operatorname*{supp}M\text{.}
\label{lhh2}%
\end{equation}
Of course, the injective hull is usually not finitely generated.

The following observation is both entirely elementary, yet the key obstruction
why the $K$-theory of $\mathsf{LCA}_{R}$ will only see phenomena with
zero-dimensional support. We will employ it for $k=F$ or completions $F_{v}$.

\begin{lemma}
\label{lemma_ModHasZeroDimSupportIffFiniteDimensional}Let $k$ be a field and
let $R$ be a finitely generated $k$-algebra. A finitely generated $R$-module
$M$ satisfies $\dim\operatorname*{supp}M=0$ if and only if it is
finite-dimensional as an $k$-vector space.
\end{lemma}

\begin{proof}
Clear.
\end{proof}

\subsubsection{Structure of adelic
objects\label{sect_StructAdelicObjectsIntegralModel}}

\subsubsection{Restricted product of categories}

Suppose $(\mathsf{C}_{v})_{v\in I}$ are exact categories for some indexing set
$I$. Then $\prod_{v\in I}\mathsf{C}_{v}$ is an exact category, where the
objects are vectors $X=(X_{v})_{v\in I}$ such that $X_{v}\in\mathsf{C}_{v}$,
morphisms are taken entrywise, and a kernel-cokernel
sequence\footnote{\cite[\S 2, for kernel-cokernel pairs]{MR2606234}} is called
exact if it is exact entrywise in each $\mathsf{C}_{v}$. Assume we are given
exact functors $\xi_{v}\colon\mathsf{C}_{v}\rightarrow\mathsf{D}_{v}$ with
$(\mathsf{C}_{v})_{v\in I},(\mathsf{D}_{v})_{v\in I}$ for the same indexing
set $I$. For any subset $I^{\prime}\subseteq I$, we define $\mathsf{J}%
^{(I^{\prime})}:=\prod_{v\in I\setminus I^{\prime}}\mathsf{C}_{v}\times
\prod_{v\in I^{\prime}}\mathsf{D}_{v}$ and whenever $I^{\prime}\subseteq
I^{\prime\prime}$, there is an exact functor $\xi_{I^{\prime},I^{\prime\prime
}}\colon\mathsf{J}^{(I^{\prime})}\rightarrow\mathsf{J}^{(I^{\prime\prime})}$,
defined by applying the functor $\xi_{v}$ entrywise for all $v\in
I^{\prime\prime}\setminus I^{\prime}$, and as the identity elsewhere. We
define the \emph{restricted product} of exact categories as%
\begin{equation}
\underset{v\in I\;}{\prod\nolimits^{\prime}}(\mathsf{D}_{v}:\mathsf{C}%
_{v}):=\left.  \underset{I^{\prime}\subseteq I\text{, }\#I^{\prime}<\infty
}{\operatorname*{colim}}\left.  \mathsf{J}^{(I^{\prime})}\right.  \right.
\text{,} \label{l_RestrictedProductAsColimitJFormula}%
\end{equation}
where the $\xi_{v}$ are implicit, and the notation is (although perhaps deemed
unnaturally ordered) inspired from the traditional notation in topological
group theory.

\subsubsection{Comparison to adelic objects}

Recall our conventions from \S \ref{sect_Setup}, especially that we had set
$\mathcal{O}_{v}=F_{v}$ if $v$ is an infinite place. The full subcategory of
adelic objects in $\mathsf{LCA}_{F}$ is denoted by $\mathsf{LCA}_{F,ad}$.

\begin{proposition}
[{\cite[\S 3]{kthyartin}}]\label{prop_IdentifyAdelicBlocks}There is an
$F$-linear\footnote{see Rmk. \ref{rmk_FLinearityOfAdeles} for the $F$-linear
structure of the restricted product category on the left, which is a little
subtle.} exact equivalence of exact categories%
\begin{equation}
\Psi\colon\underset{v\in S\;}{\prod\nolimits^{\prime}}(\mathsf{Proj}%
_{F_{v},fg}:\mathsf{Proj}_{\mathcal{O}_{v},fg})\overset{\sim}{\rightarrow
}\mathsf{LCA}_{F,ad}\text{.} \label{lmx1}%
\end{equation}
The functor $\Psi$ sends an array of modules $(M_{v})_{v\in S}$ to itself, but
equipped with the real topology for $v\in S_{\infty}$, and with the adic
topology for $v\in S\setminus S_{\infty}$.
\end{proposition}

\begin{remark}
[$F$-linearity of the adelic objects]\label{rmk_FLinearityOfAdeles}The
category $\mathsf{LCA}_{F,ad}$ is $F$-linear, but the categories
$\mathsf{Proj}_{\mathcal{O}_{v},fg}$ only admit an $\mathcal{O}_{v}$-linear
structure. Nonetheless, the restricted product $\prod\nolimits^{\prime
}(\mathsf{Proj}_{F_{v},fg}:\mathsf{Proj}_{\mathcal{O}_{v},fg})$ carries a
canonical $F$-linear structure by the following observation: Any $\alpha\in F$
can be written as $\frac{\alpha_{0}}{n}$ with $\alpha_{0}\in\mathcal{O}_{F}$
and $n\geq1$ some integer. As there are only finitely many primes dividing
$n$, take $I^{\prime}$ in Eq. \ref{l_RestrictedProductAsColimitJFormula} big
enough to contain all primes over these (these are still only finitely many),
and then $\frac{\alpha_{0}}{n}$ is invertible in the $\mathcal{O}_{v}$-linear
structure of all factors of $\mathsf{J}^{(I^{\prime})}$ loc. cit.
\end{remark}

This result can easily be lifted to the present setting.

\begin{lemma}
\label{lemma_IdentifyAdelicBlocksVersion2}Pick a finitely generated
$\mathcal{O}_{F}$-flat\footnote{= torsion-free (as $\mathcal{O}_{F}$ is
Dedekind)} $\mathcal{O}_{F}$-algebra $\mathcal{R}$ such that $\mathcal{R}%
\otimes_{\mathcal{O}_{F}}F\cong R$ (an \textquotedblleft integral
model\textquotedblright). For any place $v$ of $F$, define%
\begin{equation}
\mathcal{R}_{v}:=\mathcal{R}\otimes_{\mathcal{O}_{F}}\mathcal{O}_{v}%
\qquad\text{and}\qquad R_{v}:=R\otimes_{F}F_{v}\text{.} \label{lsoc1}%
\end{equation}
Write $\mathsf{Mod}_{\mathcal{R}_{v},0}^{R_{v}}\subseteq\mathsf{Mod}%
_{\mathcal{R}_{v}}$ for the exact category of $\mathcal{O}_{F}$-flat finitely
generated $\mathcal{R}_{v}$-modules $M$ such that $M\otimes_{\mathcal{O}_{v}%
}F_{v}$ has zero-dimensional support.\footnote{Without imposing the
$\mathcal{O}_{F}$-flatness condition, this would be an abelian category since
the functor $M\mapsto M$ $\otimes_{\mathcal{O}_{v}}F_{v}$ is exact. Now the
full subcategory of those modules which are additionally $\mathcal{O}_{F}%
$-flat is extension-closed in this abelian category, and therefore is a fully
exact subcategory. This is the exact structure in question.} There is an exact
equivalence of exact categories%
\begin{equation}
\Psi\colon\underset{v\in S\;}{\prod\nolimits^{\prime}}(\mathsf{Mod}_{R_{v}%
,0}:\mathsf{Mod}_{\mathcal{R}_{v},0}^{R_{v}})\overset{\sim}{\rightarrow
}\mathsf{LCA}_{R,ad}\text{.} \label{lwups4}%
\end{equation}
Upon forgetting the $R$-module structure in favour of the $F$-vector space
structure, $\Psi$ reduces to the exact equivalence of Prop.
\ref{prop_IdentifyAdelicBlocks}.
\end{lemma}

\begin{proof}
We use Prop. \ref{prop_IdentifyAdelicBlocks} in $\mathsf{LCA}_{F}$. Then we
note that the $R$-module structure on objects in $\mathsf{LCA}_{R,ad}$, by
transport of structure along $\Psi$, equips the objects in $\underset{v\in
S\;}{\prod\nolimits^{\prime}}(\mathsf{Proj}_{F_{v},fg}:\mathsf{Proj}%
_{\mathcal{O}_{v},fg})$ of Eq. \ref{lmx1} with an $R$-module structure. Thus,
instead of $\mathsf{Proj}_{\mathcal{O}_{v},fg}$ (resp. $\mathsf{Proj}%
_{F_{v},fg}$) we get the category of $\mathcal{R}_{v}$-modules (resp. $R_{v}%
$-modules) which are finite rank free over $\mathcal{O}_{v}$ (resp. $F_{v}$).
In either case, this forces the modules to be torsion-free over $\mathcal{O}%
_{F}$ (equivalently: as an abelian group) and therefore to be $\mathcal{O}%
_{F}$-flat. By Lemma \ref{lemma_ModHasZeroDimSupportIffFiniteDimensional} the
condition to have finite rank over $F_{v}$ is equivalent to asking the module
to have zero-dimensional support over $R_{v}$. Analogously, being a finitely
generated torsion-free $\mathcal{O}_{v}$-module, $\dim_{F_{v}}(M\otimes
_{\mathcal{O}_{v}}F_{v})<\infty$ forces $M$ to be free of finite rank. Hence,
we arrive at the categories $\mathsf{Mod}_{\mathcal{R}_{v},0}^{R_{v}}$ (resp.
$\mathsf{Mod}_{R_{v},0}$) in our claim.
\end{proof}

\begin{remark}
\label{rmk_SimilarOrders}If $\mathcal{R},\mathcal{R}^{\prime}$ are two
integral models (as in Lemma \ref{lemma_IdentifyAdelicBlocksVersion2}), we
have $\mathcal{R}_{v}\simeq\mathcal{R}_{v}^{\prime}$ for all but finitely many
places $v$, giving a direct argument why the left side in Eq. \ref{lwups4} is
independent of the choice of the integral model.
\end{remark}

\begin{example}
\label{example_Adeles}Define the \emph{ad\`{e}les} $\mathbb{A}_{S}$ as the
restricted product%
\begin{equation}
\mathbb{A}_{S}:=\underset{v\in S\;}{\prod\nolimits^{\prime}}F_{v}\text{,}
\label{lcva1}%
\end{equation}
with the usual restricted product topology. Then $\mathbb{A}_{S}%
\in\mathsf{LCA}_{F}$. This object can be promoted to an object of
$\mathsf{LCA}_{R}$, for $R$ an $F$-algebra, by tensoring $\mathbb{A}%
_{S}\otimes_{F}M$ for $M$ an $R$-module which is finite-dimensional as an
$F$-vector space. If $\dim_{F}M=\infty$, the tensor product $\mathbb{A}%
_{S}\otimes_{F}M$ fails to be locally compact.
\end{example}

\subsubsection{Constructing splittings}

In the entire section we assume that $R$ is a finitely generated algebra over
a number field $F$ and that $R$ has Property $(\mathbf{F})$.

\begin{lemma}
\label{lemma_YonedaTrick1}Suppose $A\in\mathsf{LCA}_{R}$ is adelic. Then for
every exact sequence $A\hookrightarrow X\twoheadrightarrow X^{\prime\prime}$
in $\mathsf{LCA}_{R}$ there exists an admissible monic $a\colon
A\hookrightarrow\widehat{A}$ with $\widehat{A}$ adelic such that in the
pushout%
\begin{equation}%
%%%%%%%%%%
%A \ar@{^{(}->}[r] \ar@{^{(}->}[d]_{a} \ar@{}[dr] |{\operatorname{pushout}}
%& X \ar@{->>}[r] \ar[d] & X^{\prime} \ar@{=}[d] \\
%\widehat{A} \ar@{^{(}->}[r] & \widehat{A} \cup_{A} X \ar@{->>}[r] & X^{\prime}
%}} }%
%%%%%%%%%%%%%%%%%
\xymatrix{
A \ar@{^{(}->}[r] \ar@{^{(}->}[d]_{a} \ar@{}[dr] |{\operatorname{pushout}}
& X \ar@{->>}[r] \ar[d] & X^{\prime} \ar@{=}[d] \\
\widehat{A} \ar@{^{(}->}[r] & \widehat{A} \cup_{A} X \ar@{->>}[r] & X^{\prime}
}
%%%%%%%%%%%%%%%%%
\label{lwups2}%
\end{equation}
the lower row splits.
\end{lemma}

Some readers may wish to interpret this lemma as an effaceable functor
property for $\operatorname*{Ext}^{1}$, but the concept of effaceability
appears to be disappearing from the literature.

\begin{proof}
Using Property $\mathbf{(F)}$ for the arrow $A\rightarrow X$, we get a diagram%
\[%
%%%%%%%%%%
%& \tilde{A} \oplus\tilde{K} \ar@{^{(}->}[d] \\
%A \ar[dr]_{0} \ar@{.>}[ur]_{h} \ar@{^{(}->}[r] & X \ar@{->>}[d] \\
%& \tilde{D},
%}}}%
%%%%%%%%%%%%%%%%%
\xymatrix{
& \tilde{A} \oplus\tilde{K} \ar@{^{(}->}[d] \\
A \ar[dr]_{0} \ar@{.>}[ur]_{h} \ar@{^{(}->}[r] & X \ar@{->>}[d] \\
& \tilde{D},
}%
%%%%%%%%%%%%%%%%%
\]
where the downward diagonal arrow must be zero by the filtration $\mathcal{F}%
$. Thus, the upward dotted diagonal arrow exists and by \cite[Prop.
7.6]{MR2606234} it must be an admissible monic. Hence, we get $h\colon
A\hookrightarrow\tilde{A}\oplus\tilde{K}$. We may consider $A\cap\tilde{K}$ in
$\mathsf{LCA}_{F}$, and it is a compact $F$-vector space contained in an
adelic object. This is only possible if $A\cap\tilde{K}=0$. Thus, we get an
admissible monic $f\colon A\hookrightarrow\tilde{A}$ and we learn that the
monic $A\hookrightarrow X$ actually comes from an admissible subobject of the
summand $\tilde{A}$ alone. Now employ the exact equivalence of categories
provided by Lemma \ref{lemma_IdentifyAdelicBlocksVersion2}. The admissible
monic $f$ gets identified with a vector $(f_{v})$ of admissible monics%
\[
f_{v}\colon A_{v}\hookrightarrow\tilde{A}_{v}\qquad\qquad\text{(}v\text{ a
place of }F\text{)}%
\]
in $\mathsf{Mod}_{\mathcal{R}_{v},0}$ (resp. $\mathsf{Mod}_{R_{v},0}$, as in
Eq. \ref{lsoc1}). Now we deal with each place $v$ individually:\ Pick an
injective envelope $A_{v}\hookrightarrow E_{\mathcal{R}_{v}}(A_{v})$. As
$E_{\mathcal{R}_{v}}(A_{v})$ is an injective module (similarly for $R_{v}$),
we can find a lift, giving us the solid arrows in the following diagram
(ignore the dotted arrows and ${\widehat{A}}_{v}$ for the moment!)%
\begin{equation}%
%%%%%%%%%%
%A_{v} \ar@{^{(}->}[rr]^{f_{v}} \ar@{^{(}->}[ddr] \ar@{^{(}..>}[dr]_{a}
%& & \tilde{A}_{v}
%\ar[ddl] \ar@{..>}[dl]^{\tilde{a}} \\
%& {\widehat{A}}_{v} \ar@{^{(}..>}[d] & \\
%& E_{\mathcal{R}_{v}}(A_{v}).
%}}} }%
%%%%%%%%%%%%%%%%%
\Scale[0.9]{\xymatrix{
A_{v} \ar@{^{(}->}[rr]^{f_{v}} \ar@{^{(}->}[ddr] \ar@{^{(}..>}[dr]_{a}
& & \tilde{A}_{v}
\ar[ddl] \ar@{..>}[dl]^{\tilde{a}} \\
& {\widehat{A}}_{v} \ar@{^{(}..>}[d] & \\
& E_{\mathcal{R}_{v}}(A_{v}).
}}
%%%%%%%%%%%%%%%%%
\label{lmiau1}%
\end{equation}
Since both $A_{v}$ and $\tilde{A}_{v}$ are finitely generated $\mathcal{R}%
_{v}$-modules (resp. $R_{v}$), their respective images in $E_{\mathcal{R}_{v}%
}(A_{v})$ (resp. $E_{\mathcal{R}_{v}}(A_{v})\otimes_{\mathcal{R}_{v}}R_{v}$)
are also finitely generated, so we find a finitely generated $\mathcal{R}_{v}%
$-submodule ${\widehat{A}}_{v}$ such that both arrows factor through it,
providing us with the dotted arrows $a,\tilde{a}$ in Diagram \ref{lmiau1}. As
$E_{\mathcal{R}_{v}}(A_{v})$ has zero-dimensional support after base change to
$R_{v}$ (by Eq. \ref{lhh2}) and ${\widehat{A}}_{v}$ is a submodule, we must
have $\dim\operatorname*{supp}_{R_{v}}{\widehat{A}}_{v}=0$ by Eq. \ref{lhh1},
and therefore, again after base change to $R_{v}$, it is finite rank over over
$F_{v}$ by\ Lemma \ref{lemma_ModHasZeroDimSupportIffFiniteDimensional}. Hence,
the vector $({\widehat{A}}_{v})_{v}$ defines an object of the restricted
product category, and running the reverse direction of Lemma
\ref{lemma_IdentifyAdelicBlocksVersion2}, we obtain the factorization
(depicted below on the left) in $\mathsf{LCA}_{R}$. Taking the pushout along
$A\hookrightarrow\widehat{A}$, we get the diagram below on the right.%
\begin{equation}%
%%%%%%%%%%
%A \ar@{^{(}->}[rr]^{f} \ar@{^{(}->}[dr]_{a} && \tilde{A} \ar@{->}%
%[dl]^{\tilde{a}} \\
%& {\widehat{A}}
%}}}%
%%%%%%%%%%%%%%%%%
\xymatrix{
A \ar@{^{(}->}[rr]^{f} \ar@{^{(}->}[dr]_{a} && \tilde{A} \ar@{->}%
[dl]^{\tilde{a}} \\
& {\widehat{A}}
}%
%%%%%%%%%%%%%%%%%
\qquad\qquad%
%%%%%%%%%%
%A \ar@{^{(}->}[r] \ar@{^{(}->}[d]_{a} \ar@{}[dr] |{\operatorname{pushout}}
%& \tilde{K} \oplus\tilde{D} \oplus\tilde{A} \ar@{->>}[r] \ar[d] & X^{\prime}
%\ar@{=}[d] \\
%\widehat{A} \ar@{^{(}->}[r] & \tilde{K} \oplus\tilde{D} \oplus(\tilde{A}
%\cup_{A} \widehat{A}) \ar@/^1pc/[l]^{b} \ar@{->>}[r] & X^{\prime}.
%}} }%
%%%%%%%%%%%%%%%%%
\xymatrix{
A \ar@{^{(}->}[r] \ar@{^{(}->}[d]_{a} \ar@{}[dr] |{\operatorname{pushout}}
& \tilde{K} \oplus\tilde{D} \oplus\tilde{A} \ar@{->>}[r] \ar[d] & X^{\prime}
\ar@{=}[d] \\
\widehat{A} \ar@{^{(}->}[r] & \tilde{K} \oplus\tilde{D} \oplus(\tilde{A}
\cup_{A} \widehat{A}) \ar@/^1pc/[l]^{b} \ar@{->>}[r] & X^{\prime}.
}
%%%%%%%%%%%%%%%%%
\label{lwups1}%
\end{equation}
Defining a morphism $b$ out of $\tilde{A}\cup_{A}\widehat{A}$ is (by the
universal property of pushouts) equivalent to giving morphisms $\tilde
{A}\rightarrow\widehat{A}$, $\widehat{A}\rightarrow\widehat{A}$ which agree on
$A$. We take $(\tilde{a},\operatorname*{id}_{\widehat{A}})$, and the agreement
is just the left diagram in Eq. \ref{lwups1}. By inspection, $b$ is a left
splitting of the exact sequence in the bottom row on the right in Eq.
\ref{lwups1} (as $\widehat{A}$ is a subobject of the summand $\tilde{A}%
\cup_{A}\widehat{A}$ alone, it is sufficient to set up a splitting defined on
this summand). Forgetting the summand decomposition, we obtain Eq.
\ref{lwups2}.
\end{proof}

\begin{corollary}
\label{cor_YonedaTrick2}Suppose $A,X\in\mathsf{LCA}_{R}$ and assume $A$ is
adelic. Then for every $\phi\in\operatorname*{Ext}\nolimits^{1}(X,A)$ there
exists an admissible monic $a\colon A\hookrightarrow\widehat{A}$ with
$\widehat{A}$ adelic such that the Yoneda product%
\[
\operatorname*{Hom}(A,\widehat{A})\otimes\operatorname*{Ext}\nolimits^{1}%
(X,A)\longrightarrow\operatorname*{Ext}\nolimits^{1}(X,\widehat{A})
\]
sends $a\otimes\phi$ to zero.
\end{corollary}

\begin{proof}
We work with Yoneda's presentation of $\operatorname*{Ext}$%
-groups\footnote{This is valid in all exact categories and does not depend on
the availability of enough injectives or projectives ($\mathsf{LCA}_{R}$ has
neither; \cite[Prop. 8.1]{MR4028830} or \cite[Thm. 4.15]{noncomclassgroup}
classify injectives and projectives for certain $R$).}. The class $[\phi]$
represents an exact sequence $A\hookrightarrow U\twoheadrightarrow X$ in
$\mathsf{LCA}_{R}$. The Yoneda product with an arbitrary morphism $a\colon
A\rightarrow\widehat{A}$ then amounts to taking the pushout of this exact
sequence by $a$. As in any exact category, the pushout of an admissible monic
along an arbitrary morphism is again an admissible monic. Now take $a$ as
supplied by Lemma \ref{lemma_YonedaTrick1} and use that \textit{split} exact
sequences are trivial in $\operatorname*{Ext}\nolimits^{1}$.
\end{proof}

\subsection{The functor of the `discrete part'}

In the entire section we assume that $R$ is a finitely generated algebra over
a number field $F$ and that $R$ has Property $(\mathbf{F})$.

We consider%
\[
J\colon\mathsf{LCA}_{R}\longrightarrow\mathsf{Mod}_{R}\text{,}\qquad
X\mapsto\operatorname{gr}_{F}(X)_{2}=X/\mathcal{F}_{1}X\text{,}%
\]
i.e., the functor which sends each $X$ to its discrete filtered piece. As the
filtration is canonical, this is a well-defined functor.

\begin{example}
This functor fails to be exact. Even for $R:=F$ itself (just a single closed
point), the ad\`{e}le sequence $F\hookrightarrow\mathbb{A}_{S}%
\twoheadrightarrow\mathbb{A}_{S}/F$ gets sent to $F\rightarrow0\rightarrow0$.
\end{example}

A slight variant of $J$ turns out to be exact.

\begin{lemma}
\label{lemma_JOverlineIsExact}The composite functor $\overline{J}%
\colon\mathsf{LCA}_{R}\overset{J}{\rightarrow}\mathsf{Mod}_{R}\rightarrow
\mathsf{Mod}_{R}/\mathsf{Mod}_{R,0}$ is an exact functor of exact categories.
\end{lemma}

\begin{proof}
Suppose $X^{\prime}\overset{i}{\hookrightarrow}X\twoheadrightarrow
X^{\prime\prime}$ is an exact sequence in $\mathsf{LCA}_{R}$. Using the
filtration of Def. \ref{def_PropF} for both $X^{\prime}$ and $X$, we obtain
the diagram%
\begin{equation}%
%%%%%%%%%%
%{\mathcal{F}}_{1} X^{\prime} \ar@{..>}[d] \ar@{^{(}->}[r] & X^{\prime}
%\ar@{^{(}->}[d]_{i}
%\ar@{->>}[r] & J(X^{\prime}) \ar@{..>}[d]^{j} \\
%{\mathcal{F}}_{1} X           \ar@{^{(}->}[r] & X           \ar@
%{->>}[r] & J(X), \\
%}} }%
%%%%%%%%%%%%%%%%%
\xymatrix{
{\mathcal{F}}_{1} X^{\prime} \ar@{..>}[d] \ar@{^{(}->}[r] & X^{\prime}
\ar@{^{(}->}[d]_{i}
\ar@{->>}[r] & J(X^{\prime}) \ar@{..>}[d]^{j} \\
{\mathcal{F}}_{1} X           \ar@{^{(}->}[r] & X           \ar@
{->>}[r] & J(X), \\
}
%%%%%%%%%%%%%%%%%
\label{lcux1}%
\end{equation}
where the dotted arrows come from the naturality of the filtration. To prepare
the following arguments, we recall that $\mathsf{LCA}_{R}$ is weakly
idempotent complete (this is implied by being quasi-abelian, so all kernels
exist, in particular those of idempotents). By \cite[Prop. 7.6, dualized]%
{MR2606234} the left dotted arrow in Eq. \ref{lcux1} must be an admissible
monic. We apply the snake lemma, in the version of \cite[Cor. 8.13]%
{MR2606234}. All downward arrows in Eq. \ref{lcux1} are admissible morphisms
(for the left two because they are admissible monics, and for the right one
since it is between discrete modules, i.e., lives in a full subcategory which
is abelian). We obtain the exact sequence%
\begin{equation}
0\rightarrow\ker(j)\rightarrow\mathcal{F}_{1}X/\mathcal{F}_{1}X^{\prime
}\overset{t}{\rightarrow}X^{\prime\prime}\rightarrow\operatorname*{coker}%
(j)\rightarrow0 \label{lcux2}%
\end{equation}
in the category $\mathsf{LCA}_{R}$. As $\mathcal{F}_{1}X/\mathcal{F}%
_{1}X^{\prime}$ is quasi-adelic (by\ Rmk. \ref{rmk_MaxQASubobjectIsF1} both
$\mathcal{F}_{1}X$ and $\mathcal{F}_{1}X^{\prime}$ are, and therefore so is
their quotient by \cite[Lem. 2.18]{kthyartin}). As $\ker(j)$ is simultaneously
the kernel of the arrow $t$ in the above sequence, $\ker(j)$ is a direct sum
of something quasi-adelic and a \textit{finite-dimensional} $F$-vector space
(\cite[Thm. 2.19]{kthyartin} and unravel the definition of $\mathsf{LCA}%
_{F,qab+fd}$ in the notation \textit{loc. cit.}). However, since $\ker(j)$ is
discrete, $\ker(j)$ is a finite-dimensional $F$-vector space itself. Now two
steps remain: \textit{(Step 1)} Eq. \ref{lcux2} induces%
\begin{equation}
(\mathcal{F}_{1}X/\mathcal{F}_{1}X^{\prime})/\ker(j)\hookrightarrow
X^{\prime\prime}\twoheadrightarrow\operatorname*{coker}(j)\text{.}
\label{lcux2b}%
\end{equation}
Since $j$ is a morphism between discrete objects, $\operatorname*{coker}(j)$
is discrete. And the left term is, as a quotient of something quasi-adelic,
itself quasi-adelic\footnote{If $q\colon A\twoheadrightarrow B$ is an epic and
$\mathcal{F}_{1}A=A$, then $q(A)\subseteq\mathcal{F}_{1}B$ must be all of $B$,
so $\mathcal{F}_{1}B=B$, i.e., $B$ is quasi-adelic.}. Thus, since the
filtration of Def. \ref{def_PropF} is canonical, we see that Eq. \ref{lcux2b}
agrees with $\mathcal{F}_{1}X^{\prime\prime}\hookrightarrow X^{\prime\prime
}\twoheadrightarrow J(X^{\prime\prime})$, i.e., $\operatorname*{coker}%
(j)=J(X^{\prime\prime})$. \textit{(Step 2)} As $\ker(j)$ is, by construction,
an $R$-module and simultaneously finite-dimensional as an $F$-vector space, it
can have at most zero-dimensional support as a discrete $R$-module, i.e.,
$\ker(j)\in\mathsf{Mod}_{R,0}$ (Lemma
\ref{lemma_ModHasZeroDimSupportIffFiniteDimensional}). Read in the quotient
exact category $\mathsf{Mod}_{R}/\mathsf{Mod}_{R,0}$, the right column of the
snake lemma diagram of Eq. \ref{lcux1} therefore becomes $\overline
{J}(X^{\prime})\overset{j}{\hookrightarrow}\overline{J}(X)\twoheadrightarrow
\overline{J}(X^{\prime\prime})$. This finishes the proof.
\end{proof}

\begin{remark}
The proof also shows that the functor $J$ is right exact. However, we have no
need for this property.
\end{remark}

\begin{definition}
Let $\mathsf{M}\subseteq\mathsf{LCA}_{R}$ be the full subcategory of objects
which admit an isomorphism\footnote{Equivalently, one can demand that this
holds for \textit{any} such isomorphism. This is seen by the properties of the
filtration $\mathcal{F}$.} $M\simeq K\oplus A\oplus D$ (as in Eq.
\ref{lwwmix1}) such that $D\in\mathsf{Mod}_{R,0}$.
\end{definition}

\begin{lemma}
\label{lemma_toolbox}Suppose $\mathsf{A}$ is an abelian category and
$\mathsf{A}^{\prime}\subseteq\mathsf{A}$ a (weak) Serre subcategory. Recall
that $\mathsf{D}_{\mathsf{A}^{\prime}}^{b}(\mathsf{A})\subseteq\mathsf{D}%
^{b}(\mathsf{A})$ denotes the full subcategory of bounded complexes in
$\mathsf{A}$ whose homology lies in $\mathsf{A}^{\prime}$ (as in \cite[Def.
13.2.7]{MR2182076}).

\begin{enumerate}
\item The inclusion $\mathsf{D}^{b}(\mathsf{Mod}_{R,fg})\overset{\sim
}{\rightarrow}\mathsf{D}_{\mathsf{Mod}_{R,fg}}^{b}(\mathsf{Mod}_{R})$ is an
equivalence of triangulated categories.

\item The inclusion $\mathsf{D}^{b}(\mathsf{Mod}_{R,0})\overset{\sim
}{\rightarrow}\mathsf{D}_{\mathsf{Mod}_{R,0}}^{b}(\mathsf{Mod}_{R,fg})$ is an
equivalence of triangulated categories.

\item The inclusion functor $\mathsf{D}^{b}(\mathsf{M})\rightarrow
\mathsf{D}^{b}(\mathsf{LCA}_{R})$ is fully faithful.
\end{enumerate}
\end{lemma}

\begin{proof}
(1) We shall employ the dual version of \cite[Thm. 13.2.8]{MR2182076}: Since
$R$ is Noetherian by assumption, $\mathsf{Mod}_{R,fg}\subseteq\mathsf{Mod}%
_{R}$ is a Serre subcategory. Hence, given an epic $e\colon
X\twoheadrightarrow C$ with $X\in\mathsf{Mod}_{R}$ and $C\in\mathsf{Mod}%
_{R,fg}$, we can pick a finitely generated free $R$-module $C^{\prime}%
:=R^{n}\in\mathsf{Mod}_{R,fg}$ surjecting onto $C$, and being projective, a
lift $f^{\prime}$ exists, as depicted below on the left:%
\begin{equation}%
%%%%%%%%%%
%\begin{tikzcd}
%	{C'} \\
%	\\
%	X && {C.}
%	\arrow["{f'}"', dashed, from=1-1, to=3-1]
%	\arrow["f", two heads, from=1-1, to=3-3]
%	\arrow["e"', two heads, from=3-1, to=3-3]
%\end{tikzcd}
%}}}%
%%%%%%%%%%%%%%%%%
{
\begin{tikzcd}
	{C'} \\
	\\
	X && {C.}
	\arrow["{f'}"', dashed, from=1-1, to=3-1]
	\arrow["f", two heads, from=1-1, to=3-3]
	\arrow["e"', two heads, from=3-1, to=3-3]
\end{tikzcd}
}%
%%%%%%%%%%%%%%%%%
\qquad%
%%%%%%%%%%
%\begin{tikzcd}
%	C && X \\
%	\\
%	&& {E_R(C).}
%	\arrow["m", hook, from=1-1, to=1-3]
%	\arrow["f"', from=1-1, to=3-3]
%	\arrow["{f'}", dashed, from=1-3, to=3-3]
%\end{tikzcd}
%}} }%
%%%%%%%%%%%%%%%%%
{
\begin{tikzcd}
	C && X \\
	\\
	&& {E_R(C).}
	\arrow["m", hook, from=1-1, to=1-3]
	\arrow["f"', from=1-1, to=3-3]
	\arrow["{f'}", dashed, from=1-3, to=3-3]
\end{tikzcd}
}
%%%%%%%%%%%%%%%%%
\label{lg6}%
\end{equation}
(2) Here we employ \cite[Thm. 13.2.8]{MR2182076} directly: $\mathsf{Mod}%
_{R,0}\subseteq\mathsf{Mod}_{R,fg}$ is a Serre subcategory (see the discussion
in \S \ref{section_ReminderOnSupport}). Given a monic $m\colon
C\hookrightarrow X$ with $C\in\mathsf{Mod}_{R,0}$ and $X\in\mathsf{Mod}%
_{R,fg}$, an injective hull $E_{R}(C)\in\mathsf{Mod}_{R}$ (usually not
finitely generated!) is injective and we obtain the lift $f^{\prime}$ in Eq.
\ref{lg6} on the right. Since both $C$ and $X$ are finitely generated, their
joint images inside $E_{R}(C)$ will be finitely generated, so we may replace
$E_{R}(C)$ by a submodule from $\mathsf{Mod}_{R,fg}$, and by the discussion in
\S \ref{section_ReminderOnSupport}, $E_{R}(C)$ has zero-dimensional support by
Eq. \ref{lhh2}, and so must any of its submodules. Thus, we may even replace
it by a submodule from $\mathsf{Mod}_{R,0}$.\newline(3) Since $\mathsf{LCA}%
_{R}$ is not an abelian category, we need to use a slightly different toolbox.
By Lemma \ref{lemma_MIsFullyExactSubcat} $\mathsf{M}$ is a fully exact
subcategory of $\mathsf{LCA}_{R}$. We may therefore employ Criterion
\textit{C2} via \cite[\S 12, Thm. 12.1 (b)]{MR1421815} to ensure that the
induced functor $\mathsf{D}^{b}(\mathsf{M})\rightarrow\mathsf{D}%
^{b}(\mathsf{LCA}_{R})$ is fully faithful. Thus, we need to show that for all
exact sequences $M^{\prime}\hookrightarrow X\twoheadrightarrow X^{\prime
\prime}$ in $\mathsf{LCA}_{R}$ with $M^{\prime}\in\mathsf{M}$ there exists a
commutative diagram%
\begin{equation}%
%%%%%%%%%%
%M^{\prime} \ar@{^{(}->}[r]^{u} \ar@{=}[d] & X \ar@{->>}[r] \ar[d] & X^{\prime
%\prime} \ar[d] \\
%M^{\prime} \ar@{^{(}->}[r] & M \ar@{->>}[r] & M^{\prime\prime} \\
%}} }%
%%%%%%%%%%%%%%%%%
\xymatrix{
M^{\prime} \ar@{^{(}->}[r]^{u} \ar@{=}[d] & X \ar@{->>}[r] \ar[d] & X^{\prime
\prime} \ar[d] \\
M^{\prime} \ar@{^{(}->}[r] & M \ar@{->>}[r] & M^{\prime\prime} \\
}
%%%%%%%%%%%%%%%%%
\label{lwups3}%
\end{equation}
such that the lower row is an exact sequence in $\mathsf{M}$. We proceed as
follows: Write $M^{\prime}\simeq K^{\prime}\oplus A^{\prime}\oplus D^{\prime}$
with $D^{\prime}\in\mathsf{Mod}_{R,0}$. As $D^{\prime}$ is a
finite-dimensional $F$-vector space (Lemma
\ref{lemma_ModHasZeroDimSupportIffFiniteDimensional}), we can form the exact
sequence (for $S$ being the set of all places of $F$)%
\begin{equation}
D^{\prime}\overset{\phi}{\hookrightarrow}D^{\prime}\otimes_{F}\mathbb{A}%
_{S}\twoheadrightarrow Q\text{,} \label{lsivi1a}%
\end{equation}
where $D^{\prime}$ carries (of course) the discrete topology, but we equip the
middle term with the ad\`{e}le topology. Only because $\dim_{F}D^{\prime
}<\infty$, this is indeed an adelic object in $\mathsf{LCA}_{R}$ and the
inclusion $\phi$ is a closed continuous map (Example \ref{example_Adeles}).
The object $Q$ is just defined as the quotient (as usual, this a compact and
connected LCA group, topologically). Naturally, all three objects inherit the
$R$-module structure induced from $D^{\prime}$. We abbreviate $A:=D^{\prime
}\otimes_{F}\mathbb{A}_{S}$. From the top row in Eq. \ref{lwups3} we get the
composite admissible monic $i\colon D^{\prime}\hookrightarrow M^{\prime
}\hookrightarrow X$. We wish to lift the map $\phi$ along $i$ as in%
\begin{equation}%
%%%%%%%%%%
%D^{\prime} \ar@{^{(}->}[rr]^{i } \ar@{^{(}->}[dr]_{\phi} && X \ar@
%{..>}[dl]^{\widehat{\phi}} \ar@{->>}[r] & X/D^{\prime} \\
%& A.
%}} }%
%%%%%%%%%%%%%%%%%
\xymatrix{
D^{\prime} \ar@{^{(}->}[rr]^{i } \ar@{^{(}->}[dr]_{\phi} && X \ar@
{..>}[dl]^{\widehat{\phi}} \ar@{->>}[r] & X/D^{\prime} \\
& A.
}
%%%%%%%%%%%%%%%%%
\label{lsivi1}%
\end{equation}
The obstruction to the existence of such a lift is analyzed by applying
$\operatorname*{Hom}(-,A)$ to the top sequence. We get the first row of the
following diagram%
\begin{equation}%
%%%%%%%%%%
%\operatorname{Hom}(X/D^{\prime}, A) \ar[r] \ar[d] & \operatorname
%{Hom}(X, A) \ar[r] \ar[d] & \overset{\phi\in}{\operatorname{Hom}(D^{\prime
%}, A)} \ar[r]^-{\partial} \ar[d] & \operatorname{Ext}^{1}(X/D^{\prime}%
%, A) \ar[r] \ar[d] & \cdots\\
%\operatorname{Hom}(X/D^{\prime}, \widehat{A}) \ar[r] & \underset{\widehat
%{\phi} \in}{\operatorname{Hom}(X, \widehat{A})} \ar[r] & \underset{\tilde
%{\phi} \in}{\operatorname{Hom}(D^{\prime}, \widehat{A})} \ar[r]_-{\partial}
%& \operatorname{Ext}^{1}(X/D^{\prime}, \widehat{A}) \ar[r] & \cdots}}} }%
%%%%%%%%%%%%%%%%%
\Scale[0.87]{\xymatrix{
\operatorname{Hom}(X/D^{\prime}, A) \ar[r] \ar[d] & \operatorname
{Hom}(X, A) \ar[r] \ar[d] & \overset{\phi\in}{\operatorname{Hom}(D^{\prime
}, A)} \ar[r]^-{\partial} \ar[d] & \operatorname{Ext}^{1}(X/D^{\prime}%
, A) \ar[r] \ar[d] & \cdots\\
\operatorname{Hom}(X/D^{\prime}, \widehat{A}) \ar[r] & \underset{\widehat
{\phi} \in}{\operatorname{Hom}(X, \widehat{A})} \ar[r] & \underset{\tilde
{\phi} \in}{\operatorname{Hom}(D^{\prime}, \widehat{A})} \ar[r]_-{\partial}
& \operatorname{Ext}^{1}(X/D^{\prime}, \widehat{A}) \ar[r] & \cdots}}
%%%%%%%%%%%%%%%%%
\label{lsivi2}%
\end{equation}
and by the exactness of rows, a lift exists if and only if $\partial\phi=0$.
However, whatever the class $[\partial\phi]\in\operatorname*{Ext}%
\nolimits^{1}(X/D^{\prime},A)$ may be, using Cor. \ref{cor_YonedaTrick2} there
exists an admissible monic $a\colon A\hookrightarrow\widehat{A}$ such that the
image of $[\partial\phi]$ in the bottom row is zero. Hence, at least the
composite $\tilde{\phi}\colon D^{\prime}\hookrightarrow A\hookrightarrow
\widehat{A}$ admits a lift, call it $\widehat{\phi}$. Thus, we may form the
diagram%
\begin{equation}%
%%%%%%%%%%
%M^{\prime} = K^{\prime} \oplus A^{\prime} \oplus D^{\prime} \ar@{^{(}%
%->}[r]^-{u} \ar@{=}[d] & K \oplus A \oplus D = X \ar@{->>}[r] \ar
%[d]_{\operatorname{id}_{K\oplus A} + \widehat{\phi}} & X^{\prime\prime}
%\ar[d]^{d} \\
%M^{\prime} = K^{\prime} \oplus A^{\prime} \oplus D^{\prime} \ar@{^{(}%
%->}[r]_{u\mid_{K^{\prime} \oplus A^{\prime}} + \tilde{\phi}} & \underset
%{=:M}{K \oplus A \oplus\widehat{A}} \ar@{->>}[r] & M^{\prime\prime}. \\
%}} }%
%%%%%%%%%%%%%%%%%
\xymatrix{
M^{\prime} = K^{\prime} \oplus A^{\prime} \oplus D^{\prime} \ar@{^{(}%
->}[r]^-{u} \ar@{=}[d] & K \oplus A \oplus D = X \ar@{->>}[r] \ar
[d]_{\operatorname{id}_{K\oplus A} + \widehat{\phi}} & X^{\prime\prime}
\ar[d]^{d} \\
M^{\prime} = K^{\prime} \oplus A^{\prime} \oplus D^{\prime} \ar@{^{(}%
->}[r]_{u\mid_{K^{\prime} \oplus A^{\prime}} + \tilde{\phi}} & \underset
{=:M}{K \oplus A \oplus\widehat{A}} \ar@{->>}[r] & M^{\prime\prime}. \\
}
%%%%%%%%%%%%%%%%%
\label{lsivi3}%
\end{equation}
The map $u$ comes from $X^{\prime}\hookrightarrow X$ in Eq. \ref{lwups3} (by
the filtration $\mathcal{F}$ the image of $K^{\prime}\oplus A^{\prime}$ lands
in $K\oplus A$). We need to check that the left square commutes: This can be
checked on all direct summands $K^{\prime}\oplus A^{\prime}\oplus D^{\prime}$
individually. On $K^{\prime}\oplus A^{\prime}$ it is clear: By the filtration,
$u$ sends this to $K\oplus A$, and on this summand the middle downward arrow
is just the identity. On the other hand, on the direct summand $D^{\prime}$,
the commutativity stems from the commutativity of Diagram \ref{lsivi1} (resp.
the computation in Eq. \ref{lsivi2}). This settles the commutativity of the
left square in Eq. \ref{lsivi3}. The object $M^{\prime\prime}$ is merely
defined as the quotient to make the bottom row exact. As $K,A$ and
$\widehat{A}$ are all quasi-adelic, so must be $M^{\prime\prime}$ (This
follows from the filtration $\mathcal{F}$ since the map to $M^{\prime\prime}$
is surjective, so there cannot be a discrete part. Alternatively, one may
consider the forgetful functor to $\mathsf{LCA}_{F}$ and rely on
\cite{kthyartin}). The right square stems from the universal properties of
cokernels. Finally, observe that the bottom row in Eq. \ref{lsivi3} lies in
$\mathsf{M}$. This finishes the proof of Condition \textit{C2}.
\end{proof}

We point out that in the previous proof the key complication is to ensure that
the lift $\widehat{\phi}$ is continuous. The existence of a merely algebraic
lift is much easier.

\begin{example}
The Condition \textit{C2} in the proof is the same as demanding that
$\mathsf{M}\hookrightarrow\mathsf{LCA}_{R}$ is right special in the sense of
\cite[Def. 2.12]{MR3510209}. However, $\mathsf{M}\hookrightarrow
\mathsf{LCA}_{R}$ is \textit{not} right filtering: Take $R:=\mathbf{Q}%
[T,T^{-1}]$, which is spanned by units and thus has Property $(\mathbf{F})$.
There is a continuous injective map%
\[
f\colon\mathbf{Q}[T,T^{-1}]\longrightarrow\mathbf{R}%
\]
sending $T$ to $\pi$ (or any other transcendental number). If $\mathsf{M}%
\hookrightarrow\mathsf{LCA}_{R}$ were right filtering, there would exist a
factorization of $f$ of the shape $\mathbf{Q}[T,T^{-1}]\overset{g}%
{\twoheadrightarrow}D\rightarrow\mathbf{R}$ with $D\in\mathsf{Mod}_{R,0}$. As
$f$ is injective, $\ker(g)=0$ and thus such a $D$ cannot exist. In summary,
$\mathsf{M}\hookrightarrow\mathsf{LCA}_{R}$ is not right $s$-filtering and
therefore Schlichting's localization techniques from \cite{MR2079996} cannot
be employed. However, the inclusion $\mathsf{M}\hookrightarrow\mathsf{LCA}%
_{R}$ is left filtering.
\end{example}

\subsection{Computations}

For an exact category $\mathsf{E}$, let $(\mathsf{C}^{b}(\mathsf{E}),q)$
denote the dg category (as well as complicial exact category with weak
equivalences \cite[\S 3.2.9, but see \S 3.2.2-\S 3.2.10 for the wider
context]{MR2762556}) whose objects are bounded complexes in $\mathsf{E}$ and
the weak equivalences, denoted by $q$, are the quasi-isomorphisms. Thanks to
Lemma \ref{lemma_JOverlineIsExact}, the functor $\overline{J}$, induces a dg
functor%
\[
(\mathsf{C}^{b}(\mathsf{LCA}_{R}),q)\overset{\overline{J}}{\longrightarrow
}(\mathsf{C}^{b}(\mathsf{Mod}_{R}/\mathsf{Mod}_{R,0}),q)\text{.}%
\]

\begin{lemma}
\label{lemma_MIsFullyExactSubcat}The category $\mathsf{M}$ is fully exact (cf.
\cite[Def. 10.21]{MR2606234}) inside $\mathsf{LCA}_{R}$.
\end{lemma}

\begin{proof}
It suffices to check that the full subcategory $\mathsf{M}$ is
extension-closed. If $X^{\prime}\hookrightarrow X\twoheadrightarrow
X^{\prime\prime}$ is an exact sequence in $\mathsf{LCA}_{R}$ with $X^{\prime
},X^{\prime\prime}\in\mathsf{M}$, applying $\overline{J}$ yields the exact
sequence $0\hookrightarrow\overline{J}X\twoheadrightarrow0$, forcing
$\overline{J}X=0$, thus $X\in\mathsf{M}$.
\end{proof}

Consider the commutative diagram%
\begin{equation}%
%%%%%%%%%%
%\begin{tikzcd}
%	{(\mathsf{C}^{b}(\mathsf{M}),q)} \\
%	\\
%	{(\mathsf{C}^{b}(\mathsf{LCA}_{R})^{w},q)} && {(\mathsf{C}^{b}(\mathsf
%{LCA}_{R}),q)} && {(\mathsf{C}^{b}(\mathsf{Mod}_{R}/\mathsf{Mod}_{R,0}),q),}
%	\arrow["U"', from=1-1, to=3-1]
%	\arrow[from=3-1, to=3-3]
%	\arrow["{ \overline{J}}", from=3-3, to=3-5]
%\end{tikzcd}
%}} }%
%%%%%%%%%%%%%%%%%
{
\begin{tikzcd}
	{(\mathsf{C}^{b}(\mathsf{M}),q)} \\
	\\
	{(\mathsf{C}^{b}(\mathsf{LCA}_{R})^{w},q)} && {(\mathsf{C}^{b}(\mathsf
{LCA}_{R}),q)} && {(\mathsf{C}^{b}(\mathsf{Mod}_{R}/\mathsf{Mod}_{R,0}),q),}
	\arrow["U"', from=1-1, to=3-1]
	\arrow[from=3-1, to=3-3]
	\arrow["{ \overline{J}}", from=3-3, to=3-5]
\end{tikzcd}
}
%%%%%%%%%%%%%%%%%
\label{lvm1}%
\end{equation}
where the superscript `$(-)^{w}$' signifies the dg subcategory\footnote{and
fully exact subcategory with weak equivalences \cite[\S 3.2.10]{MR2762556}} of
objects $X$ such that $\overline{J}(X)$ is weakly equivalent to zero. The
lower row is a Verdier localization sequence by construction. The left
downward arrow $U$ induces an equivalence on the level of homotopy categories.
We provide a detailed argument: First, define%
\[
U\colon\mathsf{M}\longrightarrow\mathsf{LCA}_{R}\text{,}%
\]
as an exact functor between exact categories ($U$ preserves exactness by Lemma
\ref{lemma_MIsFullyExactSubcat}). This, in turn, induces a dg functor of dg
categories%
\[
\mathsf{C}^{b}(\mathsf{M})\overset{U}{\longrightarrow}\mathsf{C}%
^{b}(\mathsf{LCA}_{R})\text{.}%
\]
Since $\overline{J}(X)=0$ for all $X\in\mathsf{M}$ (write $X\simeq K\oplus
A\oplus D$ as in Eq. \ref{lwwmix1}, and evidently each summand gets sent to
zero), this functor restricts to a dg functor%
\[
\mathsf{C}^{b}(\mathsf{M})\overset{U}{\longrightarrow}\mathsf{C}%
^{b}(\mathsf{LCA}_{R})^{w}\text{.}%
\]
We will now prove the following:

\begin{proposition}
\label{prop_EquivOfDerCats}The functor $U$ induces an equivalence on
$K$-theory%
\[
K(\mathsf{M})\overset{\sim}{\longrightarrow}K(\mathsf{C}^{b}(\mathsf{LCA}%
_{R})^{w})\text{.}%
\]

\end{proposition}

\begin{proof}
First, recall that%
\begin{equation}
\mathsf{D}^{b}(\mathsf{M})\rightarrow\mathsf{D}^{b}(\mathsf{LCA}_{R})
\label{lh4}%
\end{equation}
is fully faithful by Lemma \ref{lemma_toolbox} (3), so we may regard
$\mathsf{D}^{b}(\mathsf{M})$ as a full subcategory. To start our proof, we
need to identify what $(\mathsf{C}^{b}(\mathsf{LCA}_{R})^{w},q)$ means: An
object in $(\mathsf{C}^{b}(\mathsf{LCA}_{R})^{w},q)$ is a bounded complex
$C_{\bullet}$ in $\mathsf{LCA}_{R}$ which, after applying $\overline{J}$,
becomes weakly equivalent to the zero complex in $\mathsf{Mod}_{R}%
/\mathsf{Mod}_{R,0}$. Rephrased: the complex $\overline{J}(C_{\bullet})$ is
acyclic in $\mathsf{Mod}_{R}/\mathsf{Mod}_{R,0}$.\newline We can use this as
follows: Suppose $C_{\bullet}$ is such a complex. Then by the naturality of
the filtration in Definition \ref{def_PropF} we obtain a subcomplex%
\begin{equation}
\mathcal{F}_{1}C_{\bullet}\hookrightarrow C_{\bullet}\twoheadrightarrow
C_{\bullet}/\mathcal{F}_{1}C_{\bullet}\text{.} \label{lh7}%
\end{equation}
As a subcomplex and quotient complex, both $\mathcal{F}_{1}C_{\bullet}$ and
$C_{\bullet}/\mathcal{F}_{1}C_{\bullet}$ are also bounded complexes in
$\mathsf{LCA}_{R}$. Since the objects in $\mathcal{F}_{1}C_{\bullet}$ have no
discrete pieces $D$ in their direct sum decomposition of Eq. \ref{lwwmix1}, we
obtain $\mathcal{F}_{1}C_{\bullet}\in\mathsf{C}^{b}(\mathsf{M})$ and so
$\mathcal{F}_{1}C_{\bullet}\in\mathsf{D}^{b}(\mathsf{M})$, recalling that we
may treat $\mathsf{D}^{b}(\mathsf{M})$ as a full subcategory of $\mathsf{D}%
^{b}(\mathsf{LCA}_{R})$. Next, we know that $C_{\bullet}/\mathcal{F}%
_{1}C_{\bullet}=\overline{J}(C_{\bullet})$ is a bounded acyclic complex, when
considered in $\mathsf{Mod}_{R}/\mathsf{Mod}_{R,0}$. So, rephrased,
$C_{\bullet}/\mathcal{F}_{1}C_{\bullet}$ is a bounded complex in
$\mathsf{Mod}_{R}$ whose homology lies in $\mathsf{Mod}_{R,0}$. As its
homology is finitely generated by Lemma \ref{lemma_toolbox} (1)%
\[
\mathsf{D}^{b}(\mathsf{Mod}_{R,fg})\overset{\sim}{\rightarrow}\mathsf{D}%
_{\mathsf{Mod}_{R,fg}}^{b}(\mathsf{Mod}_{R})
\]
it is quasi-isomorphic to a bounded complex of finitely generated $R$-modules.
Its homology has not changed, so instead of plain $\mathsf{D}^{b}%
(\mathsf{Mod}_{R,fg})$, our complex even lies in $\mathsf{D}_{\mathsf{Mod}%
_{R,0}}^{b}(\mathsf{Mod}_{R,fg})$. Next, Lemma \ref{lemma_toolbox} (2)%
\[
\mathsf{D}^{b}(\mathsf{Mod}_{R,0})\overset{\sim}{\rightarrow}\mathsf{D}%
_{\mathsf{Mod}_{R,0}}^{b}(\mathsf{Mod}_{R,fg})
\]
shows that it is quasi-isomorphic to a bounded complex of modules from
$\mathsf{Mod}_{R,0}$. Hence, $C_{\bullet}/\mathcal{F}_{1}C_{\bullet}%
\in\mathsf{D}^{b}(\mathsf{M})$. As Eq. \ref{lh7} is exact, it determines a
distinguished triangle in $\mathsf{D}^{b}$, so we deduce that $C_{\bullet}%
\in\mathsf{D}^{b}(\mathsf{M})$ as the other two terms lie in $\mathsf{D}%
^{b}(\mathsf{M})$. It follows that%
\[
(\mathsf{C}^{b}(\mathsf{M}),q)\overset{U}{\longrightarrow}(\mathsf{C}%
^{b}(\mathsf{LCA}_{R})^{w},q)
\]
induces an essentially surjective functor on the level of derived categories%
\[
\mathsf{D}^{b}(\mathsf{M})=\mathsf{C}^{b}(\mathsf{M})[\operatorname*{qis}%
\nolimits^{-1}]\longrightarrow\mathsf{C}^{b}(\mathsf{LCA}_{R})^{w}%
[\operatorname*{qis}\nolimits^{-1}]\text{,}%
\]
but as in Eq. \ref{lh4}, we already know that the functor is also fully
faithful. Thus, $U$ induces an equivalence on the homotopy categories, and
thus on the level of non-commutative motives. In particular,%
\[
K(\mathsf{M})\overset{\sim}{\longrightarrow}K((\mathsf{C}^{b}(\mathsf{LCA}%
_{R})^{w},q))
\]
is an equivalence on $K$-theory.
\end{proof}

\subsection{Proof of Theorem \ref{introThmB}}

We recall that for a scheme $X$ over $\operatorname*{Spec}\mathcal{O}_{F}$, we
write $X_{F}:=X\times_{\mathcal{O}_{F}}F$ for the generic fiber. Write $j$ for
the induced map $j\colon X_{F}\rightarrow X$. For brevity, we denote by
$\mathcal{M}_{F}$ the pullback $j^{\ast}\mathcal{M}$ of a sheaf on $X$. Let
$\mathcal{N}_{\leq d}\mathsf{Coh}_{X}$ denote full subcategory of
$\mathcal{O}_{F}$-flat coherent sheaves on $X$ such that $\mathcal{M}_{F}$ has
support of dimension $\leq d$. Since $\mathsf{Coh}_{X}$ is an abelian category
and the functor $(-)_{F}=j^{\ast}$ is exact, invoking Eq. \ref{lhh1} shows
that $\mathcal{N}_{\leq d}\mathsf{Coh}_{X}$ is an extension-closed subcategory
of $\mathsf{Coh}_{X}$. In particular, it is itself an exact category. For
those schemes whose structure morphism $X\rightarrow\operatorname*{Spec}%
\mathcal{O}_{F}$ factors over $\operatorname*{Spec}F$ (in brief: schemes
defined over $F$), all sheaves are automatically $\mathcal{O}_{F}$-flat, so
$\mathcal{N}_{\leq d}\mathsf{Coh}_{X}$ is the usual niveau/dimension
filtration of the category of coherent sheaves. In particular, in this case
$\mathcal{N}_{\leq d}\mathsf{Coh}_{X}$ is even an abelian category (because it
is a Serre subcategory in all coherent sheaves).

\begin{lemma}
\label{keylocallemma}Let $R$ be a finitely generated $F$-algebra satisfying
Property $(\mathbf{F})$. Let $\mathcal{R}$ be a finitely generated flat
$\mathcal{O}_{F}$-algebra such that $\mathcal{R}\otimes_{\mathcal{O}_{F}%
}F\simeq R$. Define%
\[
X:=\operatorname*{Spec}R\qquad\text{and}\qquad\mathcal{X}%
:=\operatorname*{Spec}\mathcal{R}\text{.}%
\]
As usual, $X_{v}:=\operatorname*{Spec}R\otimes_{F}F_{v}$ and $\mathcal{X}%
_{v}:=\operatorname*{Spec}\mathcal{R}\otimes_{\mathcal{O}_{F}}\mathcal{O}_{v}%
$. Then the sequence%
\[
K(\mathcal{N}_{\leq0}\mathsf{Coh}_{X})\longrightarrow\left.  \underset{v\in
S\;}{\prod\nolimits^{\prime}}\right.  K(\mathcal{N}_{\leq0}\mathsf{Coh}%
_{X_{v}}):K(\mathcal{N}_{\leq0}\mathsf{Coh}_{\mathcal{X}_{v}})\longrightarrow
K(\mathsf{LCA}_{R})
\]
induced by the exact functor $(-)\mapsto(-)\otimes_{F}\mathbb{A}_{S}$ as the
first arrow, is a fiber sequence of spectra.
\end{lemma}

\begin{proof}
\textit{(Step 1)} We will construct a commutative diagram%
\[%
%%%%%%%%%%
%K(\mathsf{Mod}_{R,0}) \ar[r] \ar[d]_{\eta} & K(\mathsf{Mod}_{R}) \ar
%[r] \ar[d] & K({\mathsf{Mod}_{R}}/{\mathsf{Mod}_{R,0}}) \ar@{=}[d] \\
%K(\mathsf{M}) \ar[r] & K(\mathsf{LCA}_{R}) \ar[r]_-{\overline{J}}
%& K({\mathsf{Mod}_{R}}/{\mathsf{Mod}_{R,0}})
%}}}%
%%%%%%%%%%%%%%%%%
\xymatrix{
K(\mathsf{Mod}_{R,0}) \ar[r] \ar[d]_{\eta} & K(\mathsf{Mod}_{R}) \ar
[r] \ar[d] & K({\mathsf{Mod}_{R}}/{\mathsf{Mod}_{R,0}}) \ar@{=}[d] \\
K(\mathsf{M}) \ar[r] & K(\mathsf{LCA}_{R}) \ar[r]_-{\overline{J}}
& K({\mathsf{Mod}_{R}}/{\mathsf{Mod}_{R,0}})
}%
%%%%%%%%%%%%%%%%%
\]
in spectra, where $\eta$ and the middle downward arrow send a module to
itself, equipped with the discrete topology. The bottom row induces a fiber
sequence in $K$-theory, attached to the bottom row in Diag. \ref{lvm1}, along
with the identification of Prop. \ref{prop_EquivOfDerCats} for the left
downward arrow \textit{loc. cit.} The top row can be obtained as the
$K$-theory localization sequence using that $\mathsf{Mod}_{R,0}$ is a Serre
subcategory of the abelian category $\mathsf{Mod}_{R}$. We leave the details
to the reader. As the right downward arrow is (trivially) an equivalence, the
left square is bi-Cartesian. Then $K(\mathsf{Mod}_{R})=0$ by the Eilenberg
swindle, giving the fiber sequence%
\[
K(\mathsf{Mod}_{R,0})\overset{\eta}{\longrightarrow}K(\mathsf{M}%
)\longrightarrow K(\mathsf{LCA}_{R})\text{.}%
\]
\textit{(Step 2)} Let $\mathsf{M}^{\prime}\subseteq\mathsf{M}$ be the full
subcategory of objects isomorphic to $K\oplus A$ alone. Suppose $D\in
\mathsf{Mod}_{R,0}$. Then $D$ is finite-dimensional as an $F$-vector space.
Hence, we may tensor the ad\`{e}le sequence (over $F$) with $D$, giving an
exact sequence $D\hookrightarrow D\otimes_{F}\mathbb{A}_{S}\twoheadrightarrow
(D\otimes_{F}\mathbb{A}_{S})/D$ in $\mathsf{LCA}_{F}$, but since $D$ still
carries the left action of $R$, this exact sequence also lives in
$\mathsf{LCA}_{R}$ (this is the same construction that we had already used in
Eq. \ref{lsivi1a}). Clearly, $D\otimes_{F}\mathbb{A}_{S}$ is adelic and the
quotient compact, so this is a resolution of $D$ by objects in $\mathsf{M}%
^{\prime}$. Modifying the proof of Prop. \cite[Prop. 4.26]{klca1} using this
idea, we obtain an equivalence of stable $\infty$-categories $\mathsf{D}%
_{\infty}^{b}\left(  \mathsf{M}^{\prime}\right)  \overset{\sim}{\rightarrow
}\mathsf{D}_{\infty}^{b}\left(  \mathsf{M}\right)  $. There is a further
localization sequence of exact categories $\mathsf{LCA}%
_{R,\operatorname*{compact}}\hookrightarrow\mathsf{M}^{\prime}%
\twoheadrightarrow$ $\mathsf{LCA}_{R,ad}$, proven similarly as in our proof of
\cite[Thm. 4.30]{klca1}. Next, $K(\mathsf{LCA}_{R,\operatorname*{compact}})=0$
by the Eilenberg swindle. This yields the fiber sequence%
\[
K(\mathsf{Mod}_{R,0})\longrightarrow K(\mathsf{LCA}_{R,ad})\longrightarrow
K(\mathsf{LCA}_{R})\text{.}%
\]
Finally, Lemma \ref{lemma_IdentifyAdelicBlocksVersion2} and the commutativity
of $K$-theory with infinite products (resp. filtering colimits) of exact
categories and yields our claim, except for the explicit description of $\eta
$.\newline\textit{(Step 3)} Finally, we claim that the map $\eta$ is also
induced by the functor $(-)\mapsto(-)\otimes_{F}\mathbb{A}_{S}$, i.e., the
functor in the claim of our result. This is based on the Additivity Theorem
for $K$-theory. Write $\mathcal{E}(-)$ for the exact category of exact
sequences, cf. \cite[Exercise 3.9]{MR2606234}. There is an exact functor%
\[
K(\mathsf{Mod}_{R,0})\rightarrow\mathcal{E}K(\mathsf{M})\text{,}\qquad
M\mapsto\left(  M\hookrightarrow M\otimes_{F}\mathbb{A}_{S}\twoheadrightarrow
(M\otimes_{F}\mathbb{A}_{S})/M\right)  \text{,}%
\]
where the three terms are equipped with the discrete, adelic resp. compact
quotient topology respectively. By Additivity, if $i_{1},i_{2},i_{3}$ denotes
the exact functors to the three individual terms in the short exact sequence,
$i_{2\ast}=i_{1\ast}+i_{3\ast}$ holds for the induced maps $K(\mathsf{Mod}%
_{R,0})\rightarrow K(\mathsf{M})$. As $i_{3}$ factors over the full
subcategory of compact modules, $K(\mathsf{Mod}_{R,0})\rightarrow
K(\mathsf{LCA}_{R,\operatorname*{compact}})\rightarrow K(\mathsf{M})$, but
$K(\mathsf{LCA}_{R,\operatorname*{compact}})=0$ (as we had already used
earlier), so we deduce $i_{3\ast}=0$. Hence $\eta_{\ast}=i_{1\ast}=i_{2\ast}$,
settling the claim we had made at the beginning of this step.\newline
\end{proof}

This being settled, we are ready to prove Theorem \ref{introThmB}.

\begin{theorem}
\label{ThmBc}Let $S$ be the set of all places and suppose $\pi\colon
X\rightarrow\operatorname*{Spec}F$ is an integral separated scheme of finite
type over the number field $F$. Then for any open $U\subseteq X$,
$K{\mathsf{LC}}_{X}^{\operatorname*{naive}}(U)$ agrees with%
\[
\operatorname{cofib}\left(  \bigoplus\limits_{x\in U_{(0)}}K\left(
\kappa(x)\right)  \longrightarrow\left.  \underset{v\in S\;}{\prod
\nolimits^{\prime}}\right.  \bigoplus\limits_{x\in U_{v,(0)}}K\left(
\kappa(x)\right)  :\bigoplus\limits_{x\in U_{v,(0)}}K(\mathcal{O}_{\kappa
(x)})\right)  \text{,}%
\]
where $U_{(0)}$ denotes the set of closed points in the open $U$,
$U_{v}:=U\times_{F}F_{v}$ and $U_{v,(0)}$ the closed points of $U_{v}$, and
$\kappa(x)$ the respective residue fields. For $x\in U_{v,(0)}$, the residue
field $\kappa(x)$ is a finite field extension of $F_{v}$, and $\mathcal{O}%
_{\kappa(x)}$ denotes its ring of integers (or $F_{v}$ itself in the case $v$
is a real or complex place, as in \S \ref{sect_Setup}). If $U$ is affine and
$\mathcal{O}_{X}(U)$ satisfies Property $(\mathbf{F})$, we also have%
\[
K{\mathsf{LC}}_{X}^{\operatorname*{naive}}(U)=K{\mathsf{LC}}%
_{X,\operatorname*{pre}}^{\operatorname*{naive}}(U)\text{,}%
\]
i.e., $K{\mathsf{LC}}_{X}^{\operatorname*{naive}}(U)$ agrees with the
$K$-theory of $\mathsf{LCA}_{\mathcal{O}_{X}(U)}$.
\end{theorem}

We do not know whether $K{\mathsf{LC}}_{X}^{\operatorname*{naive}}(U)$ agrees
with $K{\mathsf{LC}}_{X,\operatorname*{pre}}^{\operatorname*{naive}}(U)$ on
all affine opens.

\begin{proof}
We first assume that $X$ is affine and satisfies Property $(\mathbf{F})$. We
unravel the terms in Lemma \ref{keylocallemma}. We claim that $K(\mathcal{N}%
_{\leq0}\mathsf{Coh}_{X})\cong\bigoplus_{x\in X_{(0)}}K(\kappa(x))$ (and
analogously for $X_{v}$). This is harmless: For any Noetherian scheme $Y$
defined over $F$, all coherent sheaves are automatically $\mathcal{O}_{F}%
$-flat and $\mathcal{M}\cong\mathcal{M}_{F}$, so the filtration $\mathcal{N}%
_{\leq d}\mathsf{Coh}_{X}$ simplifies to agree with the niveau filtration.
Using this, every object $M\in\mathcal{N}_{\leq0}\mathsf{Coh}_{Y}$ can first
be written as a finite direct sum over its finite support, i.e.,%
\[
M\simeq\bigoplus_{y\in\operatorname*{supp}M\subseteq Y_{(0)}}i_{y\ast}%
i_{y}^{\ast}M\text{,}%
\]
where each $i_{y\ast}i_{y}^{\ast}M$ is supported on the single closed point
$y$. Each such $i_{y\ast}i_{y}^{\ast}M$ then has finite length. Since
$\mathcal{N}_{\leq0}\mathsf{Coh}_{Y}$ is an abelian category, devissage
applies with respect to the full subcategory of sheaves which are additionally
of length one \cite[Ch. V, Thm. 4.1]{MR3076731}. This category in turn is
equivalent to $\bigoplus_{y\in Y_{(0)}}\mathsf{Mod}_{\kappa(y),fg}$. Hence,
$K(\mathcal{N}_{\leq0}\mathsf{Coh}_{Y})\cong\bigoplus_{y\in X_{(0)}}%
K(\kappa(y))$. As the used hypotheses are satisfied for $X$ and $X_{v}$, the
claim is proven. A computation of $K(\mathcal{N}_{\leq0}\mathsf{Coh}%
_{\mathcal{X}_{v}})$ is a little more complicated. Unlike for the schemes over
$F$, $\mathcal{N}_{\leq0}\mathsf{Coh}_{\mathcal{X}_{v}}$ is merely an exact
category which does not contain all cokernels (because of the $\mathcal{O}%
_{F}$-flatness condition, which for example need not hold for the cokernel of
multiplication by $n\geq2$). For each $\mathcal{M}\in\mathcal{N}_{\leq
0}\mathsf{Coh}_{\mathcal{X}_{v}}$, the base change $\mathcal{M}_{F}$ has
zero-dimensional support and since $\mathcal{M}\subseteq\mathcal{M}_{F}$ must
have full rank, $\mathcal{M}$ is an $\mathcal{O}_{v}$-order in the
finite-dimensional $F_{v}$-algebra $\mathcal{M}_{F}$ (as in \cite{MR1972204}).
This order depends on the choice of our integral model $\mathcal{R}$ in Lemma
\ref{keylocallemma}, and as discussed in Rmk. \ref{rmk_SimilarOrders},
changing $\mathcal{R}$ will only affect finitely many factors $v\in S$ and the
adelic product washes out this difference. We may assume it is the maximal
order in $\kappa(x)$ for $x\in X_{v,(0)}$. Since $\kappa(x)$ is a finite
extension of $F_{v}$, we write $\mathcal{O}_{\kappa(x)}$ for its ring of
integers. We then get%
\begin{equation}
K(\mathsf{LCA}_{X})\cong\operatorname{cofib}\left(  \bigoplus\limits_{x\in
X_{(0)}}K\left(  \kappa(x)\right)  \longrightarrow\left.  \underset{v\in
S\;}{\prod\nolimits^{\prime}}\right.  \bigoplus\limits_{x\in X_{v,(0)}%
}K\left(  \kappa(x)\right)  :\bigoplus\limits_{x\in X_{v,(0)}}K(\mathcal{O}%
_{\kappa(x)})\right)  \label{lsim1}%
\end{equation}
for $X$ affine and satisfying $(\mathbf{F})$. However, the right side in Eq.
\ref{lsim1} satisfies Zariski co-descent on the subfamily of affines such that
$(\mathbf{F})$ holds. More precisely, under corestriction from a distinguished
affine open%
\[
V\subseteq U\qquad\text{with}\qquad R=\mathcal{O}_{X}(U)\text{\qquad
and\qquad}R\left[  \frac{1}{f}\right]  =\mathcal{O}_{X}(V)
\]
such that $R$ and $R[\frac{1}{f}]$ satisfy $(\mathbf{F})$ (for example, this
holds if $R$ is spanned by units, and then Lemma
\ref{lem_propertiespreservedforringsspannedbyunits} implies the same for
$R[\frac{1}{f}]$), the functor
\[
K\left(  \mathsf{LCA}_{R\left[  \frac{1}{f}\right]  }\right)  \longrightarrow
K\left(  \mathsf{LCA}_{R}\right)
\]
corresponds to the inclusion of direct summands over closed points in Eq.
\ref{lsim1}. Thanks to Lemma \ref{lem3} every finite Zariski open cover can be
refined to one such that each open is affine and satisfies $(\mathbf{F})$. As
the left-hand side is defined by co-descent (Definition \ref{def_lcax}), this
shows agreement in general.
\end{proof}

\begin{acknowledgement}
I thank Dustin Clausen, Markus Spitzweck and Matthias\ Wendt for some of their
ideas that they have shared very generously with me, and probably have shaped
my thinking conscious- and subconsciously. Fangzhou Jin and Markus Spitzweck
kindly answered some technical questions which I\ later felt embarrassed to
have asked. I thank P. Arndt, M. Groechenig, D. Macias Castillo, K.
R\"{u}lling for their support and interest.
\end{acknowledgement}

\bibliographystyle{amsalpha}
\bibliography{ollinewbib}

\def\cprime{$'$} \def\polhk#1{\setbox0=\hbox{#1}{\ooalign{\hidewidth
  \lower1.5ex\hbox{`}\hidewidth\crcr\unhbox0}}} \def\cprime{$'$}
  \def\cprime{$'$} \def\cprime{$'$} \def\cprime{$'$}
\providecommand{\bysame}{\leavevmode\hbox to3em{\hrulefill}\thinspace}
\providecommand{\MR}{\relax\ifhmode\unskip\space\fi MR }
% \MRhref is called by the amsart/book/proc definition of \MR.
\providecommand{\MRhref}[2]{%
  \href{http://www.ams.org/mathscinet-getitem?mr=#1}{#2}
}
\providecommand{\href}[2]{#2}
\begin{thebibliography}{BHvR24}

\bibitem[AB19]{kthyartin}
P.~Arndt and O.~Braunling, \emph{On the automorphic side of the {$K$}-theoretic
  {A}rtin symbol}, Selecta Math. (N.S.) \textbf{25} (2019), no.~3, Art. 38, 47.
  \MR{3954369}

\bibitem[BCM20]{MR4110725}
B.~Bhatt, D.~Clausen, and A.~Mathew, \emph{Remarks on {$K (1)$}-local
  {$K$}-theory}, Selecta Math. (N.S.) \textbf{26} (2020), no.~3, Paper No. 39,
  16. \MR{4110725}

\bibitem[BE{\O}21]{etalemotivicspectravoevodskys}
T.~Bachmann, E.~Elmanto, and P.-A. {\O}stv{\ae}r, \emph{On \'etale motivic
  spectra and {V}oevodsky's convergence conjecture}, 2021.

\bibitem[BGT13]{MR3070515}
A.~Blumberg, D.~Gepner, and G.~Tabuada, \emph{A universal characterization of
  higher algebraic {$K$}-theory}, Geom. Topol. \textbf{17} (2013), no.~2,
  733--838. \MR{3070515}

\bibitem[BGW16]{MR3510209}
O.~Braunling, M.~Groechenig, and J.~Wolfson, \emph{Tate objects in exact
  categories}, Mosc. Math. J. \textbf{16} (2016), no.~3, 433--504, With an
  appendix by Jan {\v{S}}{\v{t}}ov{\'{\i}}{\v{c}}ek and Jan Trlifaj.
  \MR{3510209}

\bibitem[BH93]{MR1251956}
W.~Bruns and J.~Herzog, \emph{Cohen-{M}acaulay rings}, Cambridge Studies in
  Advanced Mathematics, vol.~39, Cambridge University Press, Cambridge, 1993.
  \MR{1251956}

\bibitem[BHvR21]{MR4358282}
O.~Braunling, R.~Henrard, and A.-C. van {{R}oosmalen}, \emph{{$K$}-theory of
  locally compact modules over orders}, Israel J. Math. \textbf{246} (2021),
  no.~1, 315--333. \MR{4358282}

\bibitem[BHvR24]{noncomclassgroup}
\bysame, \emph{A noncommutative analogue of {C}lausen's view of the id\`ele
  class group}, Journal of the Institute of Mathematics of Jussieu (2024),
  1--48.

\bibitem[BM20]{MR4121155}
A.~Blumberg and M.~Mandell, \emph{{$K$}-theoretic {T}ate-{P}oitou duality and
  the fiber of the cyclotomic trace}, Invent. Math. \textbf{221} (2020), no.~2,
  397--419. \MR{4121155}

\bibitem[Bra19]{MR4028830}
O.~Braunling, \emph{On the relative {$K$}-group in the {ETNC}}, New York J.
  Math. \textbf{25} (2019), 1112--1177. \MR{4028830}

\bibitem[Bra23]{klca1}
\bysame, \emph{{L}ocal compactness as the $\operatorname{K}(1)$-local dual of
  finite generation}, 2023.

\bibitem[BS01]{MR1813503}
P.~Balmer and M.~Schlichting, \emph{Idempotent completion of triangulated
  categories}, J. Algebra \textbf{236} (2001), no.~2, 819--834. \MR{1813503
  (2002a:18013)}

\bibitem[BT82]{MR658304}
R.~Bott and L.~W. Tu, \emph{{D}ifferential {F}orms in {A}lgebraic {T}opology},
  Graduate Texts in Mathematics, vol.~82, Springer-Verlag, New York-Berlin,
  1982. \MR{658304}

\bibitem[B{\"u}h10]{MR2606234}
T.~B{\"u}hler, \emph{Exact categories}, Expo. Math. \textbf{28} (2010), no.~1,
  1--69. \MR{2606234 (2011e:18020)}

\bibitem[Cho25]{cho2025}
M.~Cho, \emph{{K}-theoretic {T}ate-{P}oitou duality at prime 2}, 2025.

\bibitem[Cla13]{clausenthesis}
D.~Clausen, \emph{Thesis: {A}rithmetic {D}uality in algebraic {$K$}-theory},
  {M}{I}{T} {T}hesis (2013).

\bibitem[Cla17]{clausen}
\bysame, \emph{A {K}-theoretic approach to {A}rtin maps}, arXiv:1703.07842
  [math.KT] (2017).

\bibitem[CM21]{MR4296353}
D.~Clausen and A.~Mathew, \emph{Hyperdescent and \'{e}tale {$K$}-theory},
  Invent. Math. \textbf{225} (2021), no.~3, 981--1076. \MR{4296353}

\bibitem[ELS{\O}22]{MR4444265}
E.~Elmanto, M.~Levine, M.~Spitzweck, and P.-A. {\O}stv{\ae}r, \emph{Algebraic
  cobordism and \'{e}tale cohomology}, Geom. Topol. \textbf{26} (2022), no.~2,
  477--586. \MR{4444265}

\bibitem[Fre11]{MR2827171}
C.~Frei, \emph{Sums of units in function fields}, Monatsh. Math. \textbf{164}
  (2011), no.~1, 39--54. \MR{2827171}

\bibitem[Fre12]{MR2881334}
\bysame, \emph{On rings of integers generated by their units}, Bull. Lond.
  Math. Soc. \textbf{44} (2012), no.~1, 167--182. \MR{2881334}

\bibitem[Fu15]{MR3380806}
L.~Fu, \emph{Etale cohomology theory}, revised ed., Nankai Tracts in
  Mathematics, vol.~14, World Scientific Publishing Co. Pte. Ltd., Hackensack,
  NJ, 2015. \MR{3380806}

\bibitem[Gei10]{MR2680487}
T.~Geisser, \emph{Duality via cycle complexes}, Ann. of Math. (2) \textbf{172}
  (2010), no.~2, 1095--1126. \MR{2680487}

\bibitem[Gou20]{MR4175370}
F.~Gouv\^{e}a, \emph{{$p$}-adic numbers}, Universitext, Springer, Cham, [2020]
  \copyright 2020, An introduction, Third edition of [ 1251959]. \MR{4175370}

\bibitem[GPS98]{MR1645560}
B.~Goldsmith, S.~Pabst, and A.~Scott, \emph{Unit sum numbers of rings and
  modules}, Quart. J. Math. Oxford Ser. (2) \textbf{49} (1998), no.~195,
  331--344. \MR{1645560}

\bibitem[GS18a]{MR3867292}
T.~Geisser and A.~Schmidt, \emph{Poitou-{T}ate duality for arithmetic schemes},
  Compos. Math. \textbf{154} (2018), no.~9, 2020--2044. \MR{3867292}

\bibitem[GS18b]{MR3901158}
J.~Greenlees and V.~Stojanoska, \emph{Anderson and {G}orenstein duality},
  Geometric and topological aspects of the representation theory of finite
  groups, Springer Proc. Math. Stat., vol. 242, Springer, Cham, 2018,
  pp.~105--130. \MR{3901158}

\bibitem[Har77]{MR0463157}
R.~Hartshorne, \emph{Algebraic {G}eometry}, Springer-Verlag, New York, 1977,
  Graduate Texts in Mathematics, No. 52. \MR{0463157 (57 \#3116)}

\bibitem[Har20]{MR4174395}
D.~Harari, \emph{Galois cohomology and class field theory}, Universitext,
  Springer, Cham, [2020], Translated from the 2017 French original by Andrei
  Yafaev. \MR{4174395}

\bibitem[Hen74]{MR349745}
M.~Henriksen, \emph{Two classes of rings generated by their units}, J. Algebra
  \textbf{31} (1974), 182--193. \MR{349745}

\bibitem[HM07]{MR2327028}
R.~Hahn and S.~Mitchell, \emph{Iwasawa theory for {$K(1)$}-local spectra},
  Trans. Amer. Math. Soc. \textbf{359} (2007), no.~11, 5207--5238. \MR{2327028}

\bibitem[HS07]{MR2329311}
N.~Hoffmann and M.~Spitzweck, \emph{Homological algebra with locally compact
  abelian groups}, Adv. Math. \textbf{212} (2007), no.~2, 504--524. \MR{2329311
  (2009d:22006)}

\bibitem[HSS17]{MR3607274}
M.~Hoyois, S.~Scherotzke, and N.~Sibilla, \emph{Higher traces, noncommutative
  motives, and the categorified {C}hern character}, Adv. Math. \textbf{309}
  (2017), 97--154. \MR{3607274}

\bibitem[Jin19]{jinpaper}
F.~Jin, \emph{Algebraic {G}-theory in motivic homotopy categories}, 2019.

\bibitem[Kel96]{MR1421815}
B.~Keller, \emph{Derived categories and their uses}, Handbook of algebra,
  {V}ol.\ 1, North-Holland, Amsterdam, 1996, pp.~671--701. \MR{1421815
  (98h:18013)}

\bibitem[Kry97]{MR1620000}
N.~Kryuchkov, \emph{Injective and projective objects in the category of locally
  compact modules over the ring of integers of a global field}, Mat. Zametki
  \textbf{62} (1997), no.~1, 118--123. \MR{1620000}

\bibitem[KS06]{MR2182076}
M.~Kashiwara and P.~Schapira, \emph{Categories and sheaves}, Grundlehren der
  mathematischen Wissenschaften [Fundamental Principles of Mathematical
  Sciences], vol. 332, Springer-Verlag, Berlin, 2006. \MR{2182076}

\bibitem[Lev00]{MR1740880}
M.~Levine, \emph{Inverting the motivic {B}ott element}, $K$-Theory \textbf{19}
  (2000), no.~1, 1--28. \MR{1740880}

\bibitem[Lur17]{LurieHA}
J.~Lurie, \emph{Higher algebra}, unpublished, 2017.

\bibitem[Lur18]{LurieSAG}
\bysame, \emph{{S}pectral {A}lgebraic {G}eometry ({F}ebruary 2018 version)},
  unpublished, 2018.

\bibitem[Mat89]{MR1011461}
H.~Matsumura, \emph{Commutative ring theory}, second ed., Cambridge Studies in
  Advanced Mathematics, vol.~8, Cambridge University Press, Cambridge, 1989,
  Translated from the Japanese by M. Reid. \MR{1011461 (90i:13001)}

\bibitem[Mil06]{milne2006}
J.S. Milne, \emph{Arithmetic duality theorems}, second ed., BookSurge, LLC,
  2006.

\bibitem[Mor77]{MR0442141}
S.~Morris, \emph{Pontryagin duality and the structure of locally compact
  abelian groups}, Cambridge University Press, Cambridge-New York-Melbourne,
  1977, London Mathematical Society Lecture Note Series, No. 29. \MR{0442141}

\bibitem[Mos67]{MR0215016}
M.~Moskowitz, \emph{Homological algebra in locally compact abelian groups},
  Trans. Amer. Math. Soc. \textbf{127} (1967), 361--404. \MR{0215016}

\bibitem[NSW08]{MR2392026}
J.~Neukirch, A.~Schmidt, and K.~Wingberg, \emph{Cohomology of number fields},
  second ed., Grundlehren der mathematischen Wissenschaften [Fundamental
  Principles of Mathematical Sciences], vol. 323, Springer-Verlag, Berlin,
  2008. \MR{2392026}

\bibitem[Ras95]{MR1327282}
W.~Raskind, \emph{Abelian class field theory of arithmetic schemes},
  {$K$}-theory and algebraic geometry: connections with quadratic forms and
  division algebras ({S}anta {B}arbara, {CA}, 1992), Proc. Sympos. Pure Math.,
  vol.~58, Amer. Math. Soc., Providence, RI, 1995, pp.~85--187. \MR{1327282}

\bibitem[Rei03]{MR1972204}
I.~Reiner, \emph{Maximal orders}, London Mathematical Society Monographs. New
  Series, vol.~28, The Clarendon Press, Oxford University Press, Oxford, 2003,
  Corrected reprint of the 1975 original, With a foreword by M. J. Taylor.
  \MR{1972204}

\bibitem[RSW25]{linmota}
M.~Ramzi, V.~Sosnilo, and C.~Winges, \emph{Every motive is the motive of a
  stable $\infty$-category}, 2025.

\bibitem[Sai89]{MR1045856}
S.~Saito, \emph{A global duality theorem for varieties over global fields},
  Algebraic {$K$}-theory: connections with geometry and topology ({L}ake
  {L}ouise, {AB}, 1987), NATO Adv. Sci. Inst. Ser. C: Math. Phys. Sci., vol.
  279, Kluwer Acad. Publ., Dordrecht, 1989, pp.~425--444. \MR{1045856}

\bibitem[Sch04]{MR2079996}
M.~Schlichting, \emph{Delooping the {$K$}-theory of exact categories}, Topology
  \textbf{43} (2004), no.~5, 1089--1103. \MR{2079996 (2005k:18023)}

\bibitem[Sch10]{MR2600285}
\bysame, \emph{Hermitian {$K$}-theory of exact categories}, J. K-Theory
  \textbf{5} (2010), no.~1, 105--165. \MR{2600285}

\bibitem[Sch11]{MR2762556}
\bysame, \emph{Higher algebraic {$K$}-theory}, Topics in algebraic and
  topological {$K$}-theory, Lecture Notes in Math., vol. 2008, Springer,
  Berlin, 2011, pp.~167--241. \MR{2762556}

\bibitem[Sch21]{MR4293796}
\bysame, \emph{Higher {$K$}-theory of forms {I}. {F}rom rings to exact
  categories}, J. Inst. Math. Jussieu \textbf{20} (2021), no.~4, 1205--1273.
  \MR{4293796}

\bibitem[Sch22]{sixfunct}
P.~Scholze, \emph{{L}ectures on $6$-{F}unctor {F}ormalisms}, 2022.

\bibitem[Sri10]{MR2675727}
A.~K. Srivastava, \emph{A survey of rings generated by units}, Ann. Fac. Sci.
  Toulouse Math. (6) \textbf{19} (2010), no.~Fascicule Sp\'{e}cial, 203--213.
  \MR{2675727}

\bibitem[St{\"o}69]{MR262223}
K.~O. St{\"o}hr, \emph{Dualit\"{a}ten in der {K}ategorie der lokal kompakten
  {M}oduln \"{u}ber einem {D}edekindschen {R}ing}, J. Reine Angew. Math.
  \textbf{239 (240)} (1969), 239--255. \MR{262223}

\bibitem[St{\"o}71]{MR291157}
\bysame, \emph{Funktoren in der {K}ategorie der lokal kompakten {M}oduln}, J.
  Reine Angew. Math. \textbf{246} (1971), 180--188. \MR{291157}

\bibitem[Sto12]{MR2946825}
V.~Stojanoska, \emph{Duality for topological modular forms}, Doc. Math.
  \textbf{17} (2012), 271--311. \MR{2946825}

\bibitem[Tho85]{MR826102}
R.~W. Thomason, \emph{Algebraic {$K$}-theory and \'{e}tale cohomology}, Ann.
  Sci. \'{E}cole Norm. Sup. (4) \textbf{18} (1985), no.~3, 437--552.
  \MR{826102}

\bibitem[Wal84]{MR764579}
F.~Waldhausen, \emph{Algebraic {$K$}-theory of spaces, localization, and the
  chromatic filtration of stable homotopy}, Algebraic topology, {A}arhus 1982
  ({A}arhus, 1982), Lecture Notes in Math., vol. 1051, Springer, Berlin, 1984,
  pp.~173--195. \MR{764579}

\bibitem[Wei94]{MR1269324}
C.~Weibel, \emph{An introduction to homological algebra}, Cambridge Studies in
  Advanced Mathematics, vol.~38, Cambridge University Press, Cambridge, 1994.
  \MR{1269324}

\bibitem[Wei13]{MR3076731}
\bysame, \emph{The {$K$}-book}, Graduate Studies in Mathematics, vol. 145,
  American Mathematical Society, Providence, RI, 2013, An introduction to
  algebraic $K$-theory. \MR{3076731}

\end{thebibliography}

\end{document}